\documentclass[12pt]{article}

\usepackage[margin=1in]{geometry}		
\usepackage{amsmath}				
\usepackage{amssymb}				
\usepackage{fancyhdr}				
\usepackage{amsthm}
\usepackage{enumerate}
\usepackage{bm}
\usepackage[breaklinks=True]{hyperref}

\usepackage{xcolor}
\hypersetup{
    colorlinks,
    linkcolor={red!50!black},
    citecolor={blue!50!black},
    urlcolor={blue!80!black}
}

\RequirePackage[osf]{mathpazo}

\theoremstyle{definition}
\newtheorem{thm}{Theorem}[section]
\newtheorem{lem}[thm]{Lemma}
\newtheorem{prop}[thm]{Proposition}
\newtheorem{cor}[thm]{Corollary}

\theoremstyle{definition}
\newtheorem{defn}[thm]{Definition}
\newtheorem{exmp}[thm]{Example}

\theoremstyle{remark}
\newtheorem*{rem}{Remark}
\newtheorem*{conv}{Convention}

\newcommand{\tup}[1]{\left( #1 \right)}
\newcommand{\sbr}[1]{\left[ #1 \right]} 
\newcommand{\interval}[1]{\left[ #1 \right]} 
\newcommand{\ponl}[1]{\left[ #1 \right]} 
\newcommand{\set}[1]{\left\{ #1 \right\}}
\newcommand{\cyc}[1]{\operatorname{cyc}_{#1}}

\newcommand{\card}[1]{\left| #1 \right|} 
\newcommand{\abs}[1]{\left| #1 \right|} 

\newcommand{\NN}{\mathbb{N}}
\newcommand{\ZZ}{\mathbb{Z}}
\newcommand{\QQ}{\mathbb{Q}}
\newcommand{\RR}{\mathbb{R}}

\newcommand{\kk}{\mathbf{k}}
\newcommand{\ksn}{\kk\left[S_n\right]}
\newcommand{\kf}{\kk\mathcal{F}}
\newcommand{\kfsn}{\left(\kf\right)^{S_n}}
\newcommand{\dsn}{D\left(S_n\right)}

\newcommand{\A}{\mathbf{A}}
\newcommand{\BB}{\mathbf{B}}
\newcommand{\BBA}{\BB_{\tup{1,n-1}}}
\newcommand{\BBa}{\mathbf{B}_{\alpha}}
\newcommand{\BBt}{\widetilde{\mathbf{B}}}
\newcommand{\BBta}{\widetilde{\mathbf{B}}_{\alpha}}
\newcommand{\BBi}{\mathbf{B}_{I}}
\newcommand{\TT}{\mathbf{T}}
\newcommand{\DD}{\mathbf{D}}
\newcommand{\FF}{\mathbf{F}}
\newcommand{\GG}{\mathbf{G}}

\newcommand{\CF}{\mathcal{F}}

\newcommand{\CM}{\mathcal{M}}
\newcommand{\CR}{\mathcal{R}}

\newcommand{\xx}{\mathbf{x}}

\newcommand{\bgamma}{\bm{\gamma}}

\newcommand{\wo}{w_0} 
\newcommand{\wotil}{\widetilde{w}_0}
\newcommand{\na}{n_{ \alpha } }
\newcommand{\nA}{n_{ \tup{1,n-1} } }
\newcommand{\Na}{N_{ \alpha } }
\newcommand{\nn}{\widetilde{n}}
\newcommand{\nna}{\widetilde{n}_{ \alpha } }

\newcommand{\Omal}{\Omega_{\alpha}}

\newcommand{\Des}{\operatorname{Des}}

\newcommand{\sign}{\operatorname{sign}}

\newcommand{\End}{\operatorname{End}}

\newcommand{\rowsums}{\operatorname{rowsums}}
\newcommand{\colsums}{\operatorname{colsums}}
\newcommand{\readw}{\operatorname{read}}
\newcommand{\readb}{\mathop{\overline{\operatorname{read}}}}
\newcommand{\mat}{\operatorname{mat}}

\newcommand{\gaps}{\operatorname{gaps}}
\newcommand{\sub}{\operatorname{sub}}
\newcommand{\rev}{\operatorname{rev}}
\newcommand{\red}{\operatorname{red}}
\newcommand{\type}{\operatorname{type}}
\newcommand{\Sin}{\operatorname{Sin}}
\newcommand{\id}{\operatorname{id}}

\newcommand{\con}{\twoheadrightarrow} 

\newcommand{\defin}[1]{\textbf{#1}}

\usepackage[backend=bibtex]{biblatex}
\usepackage{todonotes}

\title{Top to random and reverse: analysis of a new descent algebra shuffle}
\author{Darij Grinberg\footnote{Drexel University, Philadelphia, PA. \href{mailto:darijgrinberg@gmail.com}{\texttt{darijgrinberg@gmail.com}}}, Jonathan Parlett\footnote{University of Georgia, Athens, GA. \href{mailto:jonathan.m.parlett@gmail.com}{\texttt{jonathan.m.parlett@gmail.com}}}}
\date{Extended and corrected version of the second author's honors thesis at Drexel University (2025). \medskip\\
July 9, 2026}

\begin{document}
\maketitle

\begin{abstract}
We study the ``top-to-random-and-reverse shuffle'', defined as the top-to-random shuffle in the symmetric group algebra composed with the permutation $w_0$ (which sends each $i$ to $n+1-i$). More generally, we analyze the composition of any B-basis element of the descent algebra with $w_0$. We show that the minimal polynomial of any such composition (over $\mathbb{Q}$) factors into distinct linear factors, which correspond to the ``signed knapsack numbers'' of set compositions. This is a counterpart to an analogous property of the B-basis elements themselves, which was proved by Brown using Bidigare's face monoid. In the case of the top-to-random-and-reverse shuffle, the minimal polynomial turns out to be $\prod_{k \in \set{-n+2} \cup \interval{-n+4, n-3} \cup \set{0} \cup \set{n}} \tup{x-k}$ (for $n>1$).

Even more generally, we prove the same property for any conic (i.e., nonnegative real) combination of B-basis elements, multiplied with $w_0$. Along the way, we give a detailed exposition of Bidigare's approach to the descent algebra using the face monoid.
\end{abstract}

\section{\label{sec.intro}Introduction}

Consider the symmetric group $S_n$, consisting of permutations of the $n$-element set $\interval{n} := \set{1, 2, \ldots, n}$.
The permutations $w \in S_n$ are often visualized as a shuffled deck of $n$ cards that are labeled $1,2,\dots,n$:
Any permutation $w \in S_n$ corresponds to the state in which the cards are ordered $w\tup{1},w\tup{2},\dots,w\tup{n}$.
We may define a random state (i.e., a probability distribution on the set of states) to be a formal linear combination of permutations $\sum_{w \in S_n} a_{w}w \in \mathbb{R}[S_n]$ where the $a_{w} \in \mathbb{R}$ are non-negative and $\sum_{w \in S_n} a_{w}=1$.
Then $a_{w}$ is the probability that the deck is in state $w$.
We may drop the $\sum_{w \in S_n} a_{w}=1$ condition as we may simply ``normalize''\ our sum (divide by $\sum_{w \in S_n }a_{w}$) to obtain a random state.\footnote{The fine print: We still want $\sum_{w \in S_n} a_w$ to be nonzero for this probabilistic interpretation to work.}
For example, the element $1 + \cyc{1,2,3}$ (where $\cyc{1,2,3}$ is the $3$-cycle $1 \mapsto 2 \mapsto 3 \mapsto 1$) corresponds to the random state where the probability of the deck being ordered $1,2,3$ is $1/2$ and ordered $2,3,1$ is $1/2$.

A random shuffle can then be defined as a Markov chain on $S_n$, or, equivalently, a linear endomorphism of the $\RR$-vector space $\RR[S_n]$ that is given by a stochastic matrix.
However, $\RR[S_n]$ is itself an $\RR$-algebra -- the group algebra of the group $S_n$ -- and thus each of its elements gives rise to two endomorphisms, one by left and one by right multiplication.
Hence, each element of $\RR[S_n]$ with nonnegative coefficients gives rise to two random shuffles.
For instance, left multiplication by $1 + \cyc{1,2}$ swaps the cards labeled $1$ and $2$ with probability $1/2$ and otherwise leaves the deck unchanged; whereas right multiplication by $1 + \cyc{1,2}$ swaps the top two cards (whatever their labels are) with probability $1/2$ and otherwise leaves the deck unchanged.
This rests on the fact that the multiplication in $\RR[S_n]$ is just extending by linearity the multiplication in the symmetric group $S_n$, which corresponds to composing permutations (performing one after the other).





This algebraic viewpoint on card shuffling first rose to prominence when Diaconis and Shahshahani bounded the mixing time of some shuffles by computing the eigenvalues of the respective operators \cite{DiaSha81}; this answers the (admittedly open-ended) question ``how many times should we shuffle a deck of cards to obtain a well-shuffled deck''.
See \cite{DiaconisFulman,Diaconis03} for surveys of work done along these lines.
Recent research efforts have focused on identifying shuffles whose eigenvalues admit a combinatorial description \cite{GriLaf22, DiaSha81, Phatar91, DieSal18}.
We extend this line of research by studying a new shuffle derived from the well-known top-to-random shuffle.

Let $\TT_k$ be the element of the group algebra $\RR[S_n]$ given by
\[
\TT_k = \sum_{\substack{w \in S_n;\\ w^{-1}\tup{ k+1 } < \cdots < w^{-1}\tup{ n }}} w \qquad\tup{ \text{for each integer } 0 \leq k \leq n}.
\]
The element $ \TT_k $ is known as the \defin{$k$-top-to-random} shuffle and has been studied in \cite{Phatar91}.
In particular it is known that $\TT_1$ is diagonalizable with eigenvalues $0, 1, \ldots, n-2, n$ (for $n>1$).
That is, the \defin{minimal polynomial} of $\TT_1$ -- meaning the lowest-degree monic polynomial that vanishes when applied to $\TT_1$ -- is a product of linear factors, namely
\[
\prod_{k \in \set{0, 1, \ldots, n-2, n}} \tup{x-k} \in \RR[x]
\]
(note the conspicuous absence of the $x-\tup{n-1}$ factor).
This entails that when acting on $\RR[S_n]$ either by left or right multiplication, $\TT_1$ is a diagonalizable endomorphism with eigenvalues $0, 1, \ldots, n-2, n$; their multiplicities too can be computed.
This analysis (sans the diagonalizability) can be extended to arbitrary base rings $\kk$ instead of $\RR$ (we are restricting ourselves to $\RR$ in this introduction for simplicity's sake),
and to all the $\TT_k$'s (in fact, each $\TT_k$ is a polynomial in $\TT_1$, thus inheriting the diagonalizability of the latter over $\RR$).
See \cite{MO308600} (where $\TT_k$ is denoted $\BB_k$) for details and further references.
Informally, $\TT_k$ -- acting as a random shuffle by right multiplication -- corresponds to taking the top $k$ cards of the deck and moving them to random positions.

The \defin{antipode} of $ \ksn $ is the $ \kk $-linear map $ S : \ksn \to \ksn $ that sends each permutation $ w \in S_{n} $ to its inverse $ w^{-1} $.
This is a $ \kk $-algebra anti-isomorphism, and thus preserves the minimal polynomial.
Hence, to determine the minimal polynomial of some element it suffices to study its antipode.

The antipodes of the elements $\TT_k$ are instances of a wider class of shuffle-like elements, which form a basis of the \defin{descent algebra} of $\ksn$.
These are indexed by the \defin{compositions} of $n$, that is, by the tuples of positive integers with sum $n$.
Given a composition $\alpha = \tup{\alpha_1, \alpha_2, \ldots, \alpha_k}$ of $n$, we set
\[
\BBa = \sum w,
\]
where the sum ranges over all permutations $w \in S_n$ that increase on the smallest $\alpha_1$ elements of $\interval{n}$, on the next-smallest $\alpha_2$ elements of $\interval{n}$, on the next-smallest $\alpha_3$ elements of $\interval{n}$, etc. (but can decrease between these blocks).
Note that $S\tup{ \TT_k } = \BBa$ for the composition $\alpha = \tup{\underbrace{1, 1, \ldots, 1}_{k \text{ many $1$s}}, n-k}$ (where the entry $n-k$ is omitted if $k=n$).
These elements $\BBa$ are known as the ``$\alpha$-shuffles'' in the shuffling world, and are diagonalizable as well; their eigenvalues have been known since Bidigare's pioneering work \cite[\S 4.1.9]{Bidigare-thesis}.
Moreover, any linear combination of the $\BBa$ with nonnegative real coefficients is still diagonalizable and its eigenvalues are known, by a result of Brown \cite[Theorem 4.1]{Schocker}.

In this work, we extend this analysis to a ``mirror version''\ of the $\alpha$-shuffles and their linear combinations.
The ``mirroring'' refers to reversing the deck, which is modelled by an application of the \defin{longest word} $\wo$.
Formally, $\wo$ is defined as the permutation in $S_n$ with one-line notation $ \ponl{ n,n-1,\dots,1 } $.
We shall study the shuffles $ \wo\BBa $ and $ \BBa\wo $, which are conjugate because $\wo^{-1} = \wo$.

We shall now summarize our results in the simple-looking (but already highly nontrivial) case when $\alpha = \tup{1, n-1}$, in which case $\BBa = S\tup{ \TT_1 }$.
The $1$-top-to-random shuffle $\TT_1$ is denoted by $\A$, and is known as the \defin{top-to-random shuffle} or the \defin{Tsetlin library}; it has appeared (among other places) in Lie algebra cohomology \cite[Appendix]{Wallach} and machine learning \cite[Lemma 29]{Reizen19}.

Informally, the shuffle $\TT_1\wo$ corresponds to picking the top card, moving it to a random position, and reversing the order of all cards in the deck, while $\wo\TT_1$ corresponds to first reversing the deck, then moving the top card to a random position.
We thus call $ \wo\TT_{1} $ and $ \TT_{1}\wo $ the \defin{reverse top-to-random} shuffles.

Our first main result for this case says that $ \wo\TT_1 $ is diagonalizable with integer eigenvalues. To be more precise, we have the following theorem:
\begin{thm}[Theorem~\ref{bb1_thm2} further on]
    \label{bb1_thm}
    Let $n > 1$. The minimal polynomial of $ \wo\A $ and $ \A\wo $ over $\RR$ is
    \begin{equation}
    \label{bb1_thm.eq}
    \prod_{k \in \set{-n+2} \cup \interval{-n+4, n-3} \cup \set{0} \cup \set{n}} \tup{x-k} \in \RR\sbr{x}
    \end{equation}
(where $\interval{a,b} := \set{k \in \ZZ : a \le k \le b}$).
Thus, the elements $\wo\A$ and $\A\wo$ -- acting either by left or by right multiplication on $\RR[S_n]$ -- are diagonalizable linear endomorphisms whose eigenvalues are $0$, as well as $\pm k$ for $k \in \set{1,2,\ldots,n-4}$, as well as $n-3$ (for $n \geq 4$) and $-n+2$ and $n$.
\end{thm}


Our path to proving Theorem \ref{bb1_thm} begins with a tour of Solomon's descent algebra $\dsn$, which we define in Section \ref{sec.desalg} (including its \defin{B-basis} elements $\BBa$), and its connection to Bidigare's face algebra $ \kf $, which we introduce in Section \ref{sec.facealg}.
In these two sections, we will prove various combinatorial and algebraic properties of these objects, including some classical results of Bidigare and Brown.\footnote{Some of the more classical proofs will be deferred to the Appendix (Section~\ref{sec.omitted_proofs}), in order to avoid cluttering the section while keeping the paper self-contained.}
The underlying combinatorial objects are \defin{compositions} (i.e., tuples of positive integers) and \defin{set compositions} (also known as ``ordered set partitions'', i.e., tuples of disjoint nonempty sets with a given union).
We will refer to set compositions as ``\defin{faces}'', as they are known to correspond to the faces of the (type-$A$) braid arrangement; however, we will not use any geometry in this paper.
We let $\CF$ be the set of all set compositions of $\interval{n}$; this set is equipped with a monoid structure (Definition~\ref{def.facemon}).

In Section \ref{sec.facealg}, we will also introduce the \defin{knapsack numbers} $ \na\tup{ F } $, which -- roughly speaking -- count the ways that the blocks of a given set composition $F$ of $ \interval{n}$ can be ``packed'' into bags of sizes $ \alpha_1, \alpha_2, \ldots, \alpha_k $ (the entries of the composition $\alpha$).
The set of all these knapsack numbers is denoted $ \na\tup{ \CF } $.
These numbers (for $\alpha$ fixed and $F$ ranging over all faces) turn out to be the eigenvalues of $\BBa$, as (implicitly) shown by Brown \cite[Theorem 4.1]{Schocker}.


One of our main results is a similar description of the eigenvalues of $\wo\BBa$:
We will show (Theorem~\ref{thm.mp.1}) that these eigenvalues are the \defin{signed knapsack numbers} $\nna\tup{F} = \tup{-1}^{n - \ell\tup{F}} \na\tup{F}$ of all faces $F$, where $\ell\tup{F}$ denotes the number of blocks of $F$.
Again, the element $\wo\BBa$ is diagonalizable (over $\QQ$), so that its minimal polynomial is therefore
\[
\prod_{k \in \nna\tup{\CF}} \tup{x-k} .
\]
This will be proved in Section~\ref{sec.minpol}, after the necessary tools have been developed in Sections~\ref{sec.longpol},~\ref{sec.altsum} and~\ref{sec.shortpol}.


In Section~\ref{sec.ttr}, we will apply our results to the particular case $\alpha = \tup{1, n-1}$, proving Theorem \ref{bb1_thm} (restated as Theorem \ref{bb1_thm2}).
The perhaps strange-looking set $\set{-n+2} \cup \interval{-n+4, n-3} \cup \set{0} \cup \set{n}$ will be revealed there as the set $\nn_{\tup{1,n-1}}\tup{\CF}$ of signed knapsack numbers.

In Section \ref{sec.pos}, we will extend our analysis from a single element $\BBa$ to an arbitrary linear combination $\BB_{\bgamma} := \sum_{\alpha} \gamma_\alpha \BBa$ (where $\alpha$ ranges over the compositions of $n$) of such elements with nonnegative real coefficients $\gamma_\alpha \geq 0$.
Just like Brown determined the minimal polynomial of any such combination $\BB_{\bgamma}$, we will compute it for any $\wo \BB_{\bgamma}$.
The result (Theorem~\ref{thm.gen.mp.1}) is that any such $\wo \BB_{\bgamma}$ is again diagonalizable, and its eigenvalues are the respective linear combinations $\sum_{\alpha} \gamma_\alpha \nna\tup{F}$ of the signed knapsack numbers.
The proof is rather similar to the case of a single $\BBa$, and so we will just outline the differences.
Note that the nonnegativity of the coefficients $\gamma_\alpha$ is important for the diagonalizability, but not for the computation of the eigenvalues.

With the eigenvalues computed, an obvious question to ask is about their algebraic multiplicities.
This question can be asked about any element of the descent algebra (nonnegative coefficients are irrelevant), and it comes in four versions: the element can act either by left or by right multiplication, and it can do so either on the descent algebra or on the whole $\ksn$.
These four questions have three different answers (not four, because left and right action on $\ksn$ yield the same multiplicities).
The answer in the $\ksn$ case can be pieced together from the work of Bidigare \cite[Sections 3.8 and 4.1]{Bidigare-thesis} and Brown \cite[Appendix B]{Brown-Methods} (see \cite[Theorem 1.2]{dc2023}, \cite[Proposition 4.2]{ReSaWe14} and \cite[Corollary 4.1.3]{Bidigare-thesis} for the specific results needed), and can be found explicitly stated in \cite[Corollary 4 and Proposition 4]{GriVas26};
the descent algebra case leads into the combinatorics of integer matrices (via Solomon's Mackey formula, which is Theorem~\ref{thm.smf}). 
Thus, we will avoid this topic in the present paper, hoping to study it in more detail in future work.

As mentioned above, we used the base ring $\RR$ in this introduction for the sake of concreteness.
In the rest of the paper, we shall work over an arbitrary commutative ring $\kk$ much of the time; occasionally, we will restrict ourselves to fields of characteristic $0$ (when discussing diagonalizability and minimal polynomials).
When $\kk$ is a field of positive characteristic, the polynomials that would normally be minimal still vanish at the respective elements, but are not necessarily minimal any more; sometimes diagonalizability is also lost.
(For example, for $n = 4$ and $\kk = \mathbb{F}_3$, the element $\wo\A$ has minimal polynomial $x\tup{x-1}^2$.)
We do not venture any guesses on what the minimal polynomials over finite fields shall be.

\subsubsection*{Acknowledgments}

The first author would like to thank Sarah Brauner, Patricia Commins, Franco Saliola and Katya Vassilieva for many enlightening conversations about the descent algebra.

\section{\label{sec.not}Notation}

\begin{enumerate}[]
    \item Let $\NN = \set{0,1,2,\ldots}$.
    For each $n \in \ZZ$, we let $\interval{n}$ denote the set $ \set{1,2,\dots,n} $. (This is empty if $n \leq 0$.)
	More generally, if $a, b \in \ZZ$ are any two integers, then $\interval{a, b}$ will denote the set $\set{x \in \ZZ \mid a \leq x \leq b} = \set{a,a+1,a+2,\ldots,b}$. (Thus, $\interval{n} = \interval{1, n}$.) Note that $\interval{a, b} = \varnothing$ if $a > b$.

    \item We fix a nonnegative integer $ n \in \NN $.

    \item The symmetric group on $ n $ letters is denoted $ S_n $; it is the set of all bijections from $\interval{n}$ to $\interval{n}$.
	Its multiplication is given by $\tup{\alpha \circ \beta}\tup{i} = \alpha\tup{\beta\tup{i}}$.

    \item The \defin{one-line notation} of a permutation $\sigma \in S_n$ is the $n$-tuple $\tup{\sigma\tup{1}, \sigma\tup{2}, \ldots, \sigma\tup{n}} \in \interval{n}^n$.
	Usually, we enclose this $n$-tuple in square brackets instead of parentheses.

	By abuse of notation, we sometimes identify a permutation $\sigma$ with its one-line notation -- e.g., we write ``the permutation $\ponl{2,1,4,3}$'' for ``the permutation with one-line notation $\ponl{2,1,4,3}$''.
	We hope this will not be confused with the interval notation $\interval{a, b}$.

	\item Let $\wo$ denote the permutation in $S_n$ with one-line notation $ \ponl{ n,n-1,\dots,1 } $. More explicitly, $\wo\tup{ i } = n-i+1$ for each $i \in \interval{n}$.

    \item Fix a commutative ring $\kk$. (The reader might take $\kk = \ZZ$ without losing much generality.)

    \item We let $ \ksn $ denote the group algebra of $ S_n $ over $\kk$.
\end{enumerate}

\section{\label{sec.desalg}The Descent Algebra}

\subsection{Compositions, subsets and the bijections between them}

\begin{defn}\phantomsection\ %
    \begin{enumerate}[]
        \item\textbf{(a)} A \defin{composition} of an integer $n \in \NN$ into $k$ parts is a $k$-tuple $\alpha = \tup{\alpha_1, \alpha_2, \dots, \alpha_k} \in \interval{n}^k$ such that $\alpha_1 + \alpha_2 + \cdots + \alpha_k = n$.
		A tuple $\alpha$ is called simply a composition of $n$ if it is a composition of $n$ into $k$ parts for some $k \geq 0$. In this case we may write $\alpha \models n$.
        \item\textbf{(b)} The \defin{length} of a composition $ \alpha $, denoted $ \ell\tup{ \alpha } $, is its number of parts: $ \ell\tup{ \alpha_1, \alpha_2, \dots, \alpha_k } = k $.
    \end{enumerate}
    
\end{defn}

\begin{exmp}
    The distinct compositions of $ n=4 $, grouped by their length, are
    \begin{align*}
        \begin{tabular}{r|l}
        \text{length} & \text{compositions of this length} \\ \hline
        $1$ & $\tup{ 4 }$; \\
        $2$ & $\tup{ 3,1 }, \quad \tup{ 2,2 }, \quad \tup{ 1,3 }$; \\
        $3$ & $\tup{ 2,1,1 }, \quad \tup{ 1,2,1 }, \quad \tup{ 1,1,2 } $; \\
        $4$ & $\tup{ 1,1,1,1 }$.
        \end{tabular}
    \end{align*}
\end{exmp}

\begin{defn}
  Let $\mathcal{P}\tup{\interval{n-1}}$ denote the set of all subsets of $\interval{n-1} = \set{1,2,\dots,n-1}$,
  and $C\tup{n} = \set{ \alpha \models n}$ denote the set of compositions of $n$.
  We define the following maps:
    \begin{enumerate}[]
        \item\textbf{(a)} The map $\gaps : \mathcal{P}\tup{\interval{n-1}} \to C\tup{n}$ is given by
        \[
        \gaps\tup{J}=\tup{j_{1}-j_0,j_{2}-j_{1},\dots,j_m-j_{m-1}},
        \]
        where $J = \set{j_{1},j_{2},\dots,j_{m-1}} \subseteq \interval{n-1}$ with $j_{i} < j_{i+1}$ and where we set $j_0 = 0$ and $j_m=n$.
		Visually speaking, $\gaps\tup{J}$ is the composition whose entries are the distances between adjacent elements of $\set{0} \cup J \cup \set{n}$, that is, the lengths of the little intervals into which the set $J$ splinters the interval $\interval{n}$.
        \\
        \textbf{Exception:} If $n = 0$, then we define $\gaps\tup{\varnothing}$ to be the empty composition $\tup{}$, rather than the non-composition $\tup{0}$ that the above definition would yield.
        \item\textbf{(b)} The map $\rev : C\tup{n} \to C\tup{n}$ is given by
        \[
        \rev\tup{\alpha_1, \alpha_2, \dots, \alpha_k} = \tup{\alpha_k, \alpha_{k-1}, \dots, \alpha_1}.
        \]
		The composition $\rev \alpha$ is called the \defin{reverse} of the composition $\alpha$.
        \item\textbf{(c)} The map $\sub : \mathcal{P}\tup{\interval{n-1}} \to \mathcal{P}\tup{\interval{n-1}}$ is given by
        \[
        \sub\tup{J}=n-J = \set{n - j : j \in J}.
        \]
		Visually speaking, $\sub\tup{J}$ is the reflection of $J$ across the point $n/2$ on the number line.
\end{enumerate}

\end{defn}

\begin{prop}
    These three maps $\gaps$, $\rev$ and $\sub$ are bijections.
	In particular, $ \sub$ and $\rev $ are involutions, and we have
    \begin{equation}
    \label{eq.sub.1}
    \sub = \gaps^{-1} \circ \rev \circ \gaps .
    \end{equation}
\end{prop}

\begin{proof}
Visually, this is obvious:%
\footnote{A formalized proof can be found in \cite[ancillary file comps-compiled.pdf, Theorem 3.8 (i), Corollary 3.10]{GriVas24} (where \eqref{eq.sub.1} is stated in the equivalent form ``$D\tup{\rev \alpha} = \rev_n\tup{D\tup{\alpha}}$ for any composition $\alpha$ of $n$'', using the notations $D$ and $\rev_n$ for what we call $\gaps^{-1}$ and $\sub$).}
Reflecting a set $J \subseteq \interval{n-1}$ across the point $n/2$ causes the distances between adjacent elements of $\set{0} \cup J \cup \set{n}$ to get reversed (i.e., the same distances appear in opposite order), since $0$ and $n$ trade places;
thus, the composition $\gaps J$ is replaced by its reverse $\rev \tup{\gaps J}$.
\end{proof}

\begin{defn}[Descents]
    Let $w \in S_n$ be a permutation.
    \begin{enumerate}[]
    \item\textbf{(a)} An element $i \in \interval{n-1}$ is called a \defin{descent} of $w$ if $w\tup{i} > w\tup{i+1}$.
    \item\textbf{(b)} The \defin{descent set} of $w$ is the set of the descents of $w$.
	It is denoted by $\Des w$.
	That is,
    \[
    \Des w = \set{i \in \interval{n-1} : w\tup{i} > w\tup{i+1}} .
    \]
    \end{enumerate}
\end{defn}

\begin{exmp}
    The permutation $w = \ponl{ 5,2,3,4,1 } \in S_{5}$ has descents $ 1 $ (since $ w\tup{ 1 }=5 > w\tup{ 2 }=2 $) and $4$, but not $ 2 $ or $ 3 $.
	Thus, $ \Des w = \set{1,4} $.
\end{exmp}

\subsection{The descent algebra and its B- and D-bases}

\begin{defn}\phantomsection
    \label{dsn_def}\ %
    \begin{enumerate}[]
        \item\textbf{(a)} For any $I \subseteq \interval{n-1}$, we define an element
        \[
        \BB_{I} := \sum_{\substack{w \in S_n; \\ \Des w \subseteq I}} w \in \ksn.
        \]
        \item\textbf{(b)} For any $I \subseteq \interval{n-1}$, we define an element
        \[
        \DD_{I} := \sum_{\substack{w \in S_n; \\ \Des w = I}} w \in \ksn.
        \]
        \item\textbf{(c)} We may also index the elements $\BB_{I}$ by the composition of $n$ that corresponds to the subset $I$ under the map $\gaps$.
        That is, for each $\alpha \models n$, we set
        \[
        \BB_{\alpha} := \BB_{\gaps^{-1}\tup{\alpha}}
        = \sum_{\substack{w \in S_n; \\ \Des  w \subseteq \gaps^{-1}\tup{\alpha}}} w .
        \]
        \item\textbf{(d)} The \defin{descent algebra} $\dsn$ is defined as the $\kk$-linear span of the
        elements $\BB_I$ for all $I \subseteq \interval{n-1}$ inside $\ksn$.
        It follows from \eqref{eq.BD-cob.B} and \eqref{eq.BD-cob.D} below
        that this span is also the span of the elements $\DD_I$ for all $I \subseteq \interval{n-1}$.
        Both families $\tup{\BB_I}_{I \subseteq \interval{n-1}}$ and
        $\tup{\DD_I}_{I \subseteq \interval{n-1}}$ are known to be bases of the $\kk$-module $\dsn$
        (this is easy to see\footnote{The family $\tup{\DD_I}_{I \subseteq \interval{n-1}}$
        is linearly independent, since its entries $\DD_I$ are sums of disjoint nonempty
        sets of permutations
		(nonempty because each $I \subseteq \interval{n-1}$ is the descent set of at
        least one permutation $w \in S_n$).
        Thus, the family $\tup{\BB_I}_{I \subseteq \interval{n-1}}$ is also linearly
        independent, since Proposition~\ref{prop.BD-cob} establishes a
        triangular change-of-basis relationship between these two families.
        Therefore, this latter family $\tup{\BB_I}_{I \subseteq \interval{n-1}}$
        is a basis of its span, which is what we call $\dsn$.
        Thus, $\tup{\DD_I}_{I \subseteq \interval{n-1}}$ is a basis of $\dsn$ as well,
        again because of Proposition~\ref{prop.BD-cob}.});
        we call them the \defin{$\BB$-basis} and the
        \defin{$\DD$-basis}.
    \end{enumerate}
\end{defn}

As the name suggests, the descent algebra $\dsn$ is actually a subalgebra of $\ksn$.
This is a celebrated result of Solomon \cite[Theorem 1]{Sol76}, which we will say more about later (Theorem~\ref{dsn_kfsn_morph}); at the present moment we will not need it.

\begin{exmp}
    For $ n=4 $, we have
    \begin{align*}
        \DD_{\set{1,3}} &= \ponl{ 2,1,4,3 } + \ponl{ 3,1,4,2 } + \ponl{ 3,2,4,1 } + \ponl{ 4,1,3,2 } + \ponl{ 4,2,3,1 }, \\
        \DD_{\set{1}} &= \ponl{ 2,1,3,4 } + \ponl{ 3,1,2,4 } + \ponl{ 4,1,2,3 }, \\
        \DD_{\set{3}} &= \ponl{ 2,3,4,1 } + \ponl{ 1,3,4,2 } + \ponl{ 1,2,4,3 }, \\
		\DD_{\varnothing} &= \ponl{ 1,2,3,4 },
    \end{align*}
    and 
    \begin{align*}
        \BB_{\set{1,3}} &= \DD_{\set{1,3}} + \DD_{\set{1}} + \DD_{\set{3}} + \DD_{\varnothing} \\
                        &= \ponl{ 2,1,4,3 } + \ponl{ 3,1,4,2 } + \ponl{ 3,2,4,1 } + \ponl{ 4,1,3,2 } + \ponl{ 4,2,3,1 }\\
                        & \qquad + \ponl{ 2,1,3,4 } + \ponl{ 3,1,2,4 } + \ponl{ 4,1,2,3 }
						+ \ponl{ 2,3,4,1 } + \ponl{ 1,3,4,2 } + \ponl{ 1,2,4,3 } \\
						& \qquad + \ponl{ 1,2,3,4 }.
    \end{align*}

\end{exmp}

\begin{exmp}
\label{exam.Bn-1}
    The element $\BB_{\interval{n-1}} \in \ksn$ is
	defined as the sum of all permutations $w\in S_n$ satisfying
	$\Des w \subseteq \interval{n-1}$.
	Thus, it is simply the
	sum of all permutations $w\in S_n$ (since the condition
	$\Des w \subseteq \interval{n-1}$ is a tautology).
	In other words,
	\begin{align}
	\BB_{\interval{n-1}} = \sum_{w \in S_n} w.
	\label{eq.exam.Bn-1.1}
	\end{align}
\end{exmp}

The following proposition gives change-of-basis formulas between the $\BB$-basis and the $\DD$-basis:

\begin{prop}
\label{prop.BD-cob}
For any $I \subseteq \interval{n-1}$, we have
    \begin{align}
    \BB_{I} = \sum_{J \subseteq I} \DD_J
    \label{eq.BD-cob.B}
    \end{align}
and
    \begin{align}
    \DD_I = \sum_{J \subseteq I} \tup{-1}^{\card{I\setminus J}} \BB_J.
    \label{eq.BD-cob.D}
    \end{align}
\end{prop}

\begin{proof}
    The first formula is obvious.
	The second follows from it by M\"obius inversion on the Boolean lattice $\tup{\mathcal{P}\tup{\interval{n-1}},\subseteq}$ (see, e.g., \cite[Theorem 6.2.10]{21s} or \cite[the dual form of Theorem 2.1.1]{EC1}).\footnote{In more detail: Apply \cite[Theorem 6.2.10]{21s} to $S = \interval{n-1}$ and $A = \ksn$ and $a_I = \DD_I$ and $b_I = \BB_I$. Then, the theorem shows that \eqref{eq.BD-cob.D} follows from \eqref{eq.BD-cob.B}.}
\end{proof}

Note that $\wo$ is the unique permutation in $S_n$ having descent set $\interval{n-1}$.
Thus, $\wo = \DD_{\interval{n-1}} \in \dsn$.
Hence, the change-of-basis formula \eqref{eq.BD-cob.D} yields the following corollary.

\begin{cor}
   \label{w0_dsn_cor}
   Let $n \geq 1$. Then,
   \begin{equation}
       \wo = \sum_{I \subseteq \interval{n-1}} \tup{-1}^{n-\card{I}-1} \BB_{I}.
	   \label{eq.w0_dsn_cor.eq}
   \end{equation}
\end{cor}

\begin{proof}
Applying \eqref{eq.BD-cob.D} to $I = \interval{n-1}$, we find
\[
    \DD_{\interval{n-1}}
	= \sum_{J \subseteq \interval{n-1}} \tup{-1}^{\card{\interval{n-1}\setminus J}} \BB_J
	= \sum_{I \subseteq \interval{n-1}} \tup{-1}^{\card{\interval{n-1}\setminus I}} \BB_I
	= \sum_{I \subseteq \interval{n-1}} \tup{-1}^{n-\card{I}-1} \BB_{I},
\]
since each $I \subseteq \interval{n-1}$ satisfies
$\card{\interval{n-1}\setminus I} = \tup{n-1} - \card{I} = n-\card{I}-1$
(note that $n \geq 1$ was used here!).
Since $\wo = \DD_{\interval{n-1}}$, this is precisely the claim of the corollary.
\end{proof}

Note that $\wo$ is an involution: $\wo^2 = \id$, thus $\wo^{-1} = \wo$. Hence, conjugation by $\wo$ takes any element $a$ of $\ksn$ to $\wo a \wo$.
On the descent algebra, this takes the following simple form:

\begin{prop}
    \label{w0conj_prop}
    Let $I \subseteq \interval{n-1}$. Then we have
    \[
        \wo \BBi \wo = \BB_{\sub I}.
    \]
\end{prop}

In order to prove this, we show two simple combinatorial lemmas:

\begin{lem}
    \label{w0conj_lem}
   Let $\sigma \in S_n$. Then, $\Des \tup{\wo \sigma \wo} = \sub \tup{\Des \sigma}$. 
\end{lem}

\begin{proof}
    Let $i \in \interval{n-1}$. It suffices to show that $ i \in \Des \tup{\wo \sigma \wo} $ if and only if $ i \in \sub \tup{\Des \sigma}$.
    
    The permutation $\wo$ sends each $k \in \interval{n}$ to $n+1-k$. By the definition of $\Des w$, we have the following chain of equivalences:%
	\footnote{In the following computation, we will use the shorthand notation $uk$ for the image $u\tup{k}$ of an element $k \in \interval{n}$ under a permutation $u \in S_n$.
	Thus, an expression of the form $uvk$ with $u,v \in S_n$ and $k \in \interval{n}$ can be read either as $\tup{uv}k$ or as $u\tup{vk}$.
	This ambiguity is harmless, since $\tup{uv}k$ and $u\tup{vk}$ are equal.
	Likewise, expressions like $uvwk$ can be made sense of.}
    \begin{align*}
             & \ \tup{i \in \Des \tup{\wo \sigma \wo}} \\
        \iff & \ \tup{\wo \sigma \wo i > \wo \sigma \wo \tup{i+1}}
		            \qquad \tup{\text{by the definition of a descent}} \\
        \iff & \ \tup{\sigma \wo i < \sigma \wo \tup{i+1}} \\
        & \qquad \tup{\begin{array}{c}\text{since $\wo$ is strictly decreasing, so that} \\ \text{two elements $p,q \in \interval{n}$ satisfy $\wo p > \wo q$ if and only if $p<q$} \end{array}} \\
        \iff & \ \tup{\sigma\tup{n+1-i} < \sigma\tup{n-i}} \\
        & \qquad \tup{\text{since $\wo i = n+1-i$ and $\wo\tup{i+1} = n-i$}} \\
        \iff & \ \tup{\sigma\tup{n-i} > \sigma\tup{n+1-i}} \\ 
        \iff & \ \tup{\sigma\tup{n-i} > \sigma\tup{(n-i)+1}} \\
        \iff & \ \tup{n-i \in \Des \sigma}
		            \qquad \tup{\text{by the definition of a descent}} \\
        \iff & \ \tup{i \in \sub \tup{\Des \sigma}}
			        \qquad \tup{\text{by the definition of the map $\sub$}}.
    \end{align*}
    This proves the desired equivalence.
\end{proof}

\begin{lem}
    \label{w0conj_lem2}
    Let $\sigma \in S_n$, and $I \subseteq \interval{n-1}$.
    Then, $\Des \sigma \subseteq I$ if and only if $\Des \tup{\wo \sigma \wo} \subseteq \sub I$.
\end{lem}

\begin{proof}
    The map $\sub : \mathcal{P}\tup{\interval{n-1}} \to \mathcal{P}\tup{\interval{n-1}}$ clearly respects containment of sets: Two sets $U, V \in \mathcal{P}\tup{\interval{n-1}}$ satisfy $U \subseteq V$ if and only if $\sub U \subseteq \sub V$.
    Hence, we have $\Des \sigma \subseteq I$ if and only if $\sub \tup{\Des \sigma} \subseteq \sub I$. But Lemma~\ref{w0conj_lem} yields that $\sub \tup{\Des \sigma} = \Des \tup{\wo \sigma \wo}$,
    so this can be rewritten as ``$\Des \sigma \subseteq I$ if and only if $\Des \tup{\wo \sigma \wo} \subseteq \sub I$''.
    This proves the lemma.
\end{proof}

\begin{proof}[Proof of Proposition \ref{w0conj_prop}]
    By the definition of $\BBi$, we have $\BBi = \sum_{\substack{\sigma \in S_n; \\ \Des \sigma \subseteq I}} \sigma$. Thus,
   \begin{align*}
       \wo \BBi \wo
        &= \wo \sum_{\substack{\sigma \in S_n; \\ \Des \sigma \subseteq I}} \sigma \wo
        = \sum_{\substack{\sigma \in S_n; \\ \Des \sigma \subseteq I}} \wo \sigma \wo \\
        &= \sum_{\substack{\sigma \in S_n; \\ \Des \tup{ \wo \sigma \wo}  \subseteq \sub I}} \wo \sigma \wo
        \qquad \tup{\text{by Lemma \ref{w0conj_lem2}}} \\
        &= \sum_{\substack{\sigma \in S_n; \\ \Des \sigma  \subseteq \sub I}} \sigma
        \qquad \tup{\text{here, we have substituted $\sigma$ for $\wo \sigma \wo$}}
        \\
        &= \BB_{\sub I} \qquad \tup{\text{by the definition of $\BB_{\sub I}$}},
   \end{align*}
   as desired.
\end{proof}

Translating from the set to the composition indexing of the $ \BB $-basis, we obtain the following corollary.

\begin{cor}
    \label{cor.woBBawo}
    Let $\alpha \models n$. Then,
    \[
    \wo \BBa \wo = \BB_{\rev \alpha}.
    \]
\end{cor}

\begin{proof}
    Let $I = \gaps^{-1}\tup{\alpha}$. Then,
	$\BBa = \BBi$ by Definition \ref{dsn_def} \textbf{(c)}.
	Moreover, \eqref{eq.sub.1} yields 
	\[
	\sub I
	= \tup{\gaps^{-1} \circ \rev \circ \gaps}\tup{I}
	= \gaps^{-1}\tup{\rev \tup{\gaps I}}
	= \gaps^{-1}\tup{\rev \alpha}
	\]
	(since the definition of $I$ shows that
	$\gaps I = \alpha$).
	Hence, $\BB_{\sub I} = \BB_{\gaps^{-1}\tup{\rev \alpha}}
	= \BB_{\rev \alpha}$
	(again by Definition \ref{dsn_def} \textbf{(c)}).
	But Proposition \ref{w0conj_prop} yields
	$\wo \BBi \wo = \BB_{\sub I} = \BB_{\rev \alpha}$.
	In view of $\BBa = \BBi$, this rewrites as
	$\wo \BBa \wo = \BB_{\rev \alpha}$.
	This proves the corollary.
\end{proof}

\section{\label{sec.facealg}The Face Algebra}

\subsection{Faces and their algebra}

We shall now introduce another combinatorial notion -- that of \emph{set compositions}.
Realizing their relevance to the descent algebra was Bidigare's main contribution \cite{Bidigare-thesis}.
\footnote{Bidigare used them to reprove Solomon's main results from \cite{Sol76} in simpler and more conceptual ways. \par
Recently, a similar approach was taken by Kenyon, Kontsevich, Ogievetsky, Pohoata, Sawin and Shlosman \cite{KKOPSS24} in a more general setting:
In fact, if their poset $X$ is an antichain, their matrix $M^X$ represents the left multiplication by a generic element $\sum_{I \subseteq \interval{n-1}} a_I \DD_I$ of the descent algebra $\dsn$ on the symmetric group algebra $\ksn$ (though they index the sum not by the subsets of $[n-1]$ but by the ``$\varepsilon$ functions'', i.e., bitstrings of length $n-1$); and their main result \cite[Theorem 1]{KKOPSS24} is saying that the eigenvalues of this matrix $M^X$ are $\ZZ$-linear combinations of the coefficients $a_I$.
This is implicit in Bidigare's thesis \cite{Bidigare-thesis} and \cite[Theorem 4.1]{Schocker}, and explicit in \cite[Proposition 3.12]{ncsf2} (each of the sources describes the eigenvalues).
Their proof of this fact proceeds by embedding the descent algebra in the face algebra $\kf$, just as Bidigare did (their ``filters''\ are in bijection with the set compositions of $\interval{n}$), but subsequently they (implicitly) extend $\kf$ to a larger space indexed by the \emph{tournaments} on $\interval{n}$.
Thus, their approach is an interesting variation and new viewpoint on Bidigare's. }
We will use them to describe the eigenvalues of our operators.

\begin{defn}[Set compositions]\phantomsection
    \label{def.setcomp}\ %
    \begin{enumerate}[]
        \item\textbf{(a)} A \defin{set composition} of $\interval{n}$ -- or, for short, a \defin{face} of $\interval{n}$ -- means an (ordered) tuple $F = \tup{F_{1},F_{2},\dots,F_{k}}$ of nonempty sets satisfying $\bigcup\limits_{i=1}^k F_{i} = \interval{n}$ and $F_{i} \cap F_{j} = \varnothing$ for each $i \ne j$.
        If $F$ is a set composition of $\interval{n}$, we may write $F \models \interval{n}$.
        \item\textbf{(b)} A \defin{block} of a set composition $F=\tup{F_{1},F_{2},\dots,F_{k}}$ is one of its entries $F_{i}$ for $i \in \interval{k}$.
        \item\textbf{(c)} The \defin{length} of a set composition is its number of blocks. We denote the length of $F$ by $\ell \tup{F}$. More formally, if $F = \tup{F_{1},F_{2},\dots,F_k}$, then $\ell \tup{F}=k$.
        \item\textbf{(d)} We define the \defin{type} of a set composition $\tup{ F_{1},F_2,\dots,F_{k} }$ of $\interval{n}$ to be the composition
        \[
        \type\tup{ F_{1},F_2,\dots,F_{k} } := \tup{ \card{F_{1}},\card{F_2}, \dots, \card{F_{k}} } \models n.
        \]
        \item\textbf{(e)} For an integer $ n \in \NN $, we denote the set of all set compositions of $\interval{n}$ by $ \CF $ or by $ \CF\tup{ n } $ when $ n $ is not clear from context.
    \end{enumerate}
\end{defn}

\begin{exmp}
    Let $ n=3$. Then, $ \tup{ \set{1,2,3} }$, $\tup{ \set{1,2},\set{3} } $ and $ \tup{ \set{3},\set{1,2} } $ are three distinct set compositions / faces of the set $ \interval{3} = \set{1,2,3} $.
    When it is not confusing to do so, we will omit braces and commas when writing a face. With this convention, the above three faces are $ \tup{ 123 }$, $\tup{ 12,3 } $ and $ \tup{ 3,12 } $. The blocks of $ \tup{ 12,3 } $ are $ \set{1,2} $ and $ \set{3} $. The faces $ \tup{ 12,3 }$ and $\tup{ 3,12 } $ have length $ 2 $, while $ \tup{ 123 } $ has length $ 1 $. We have $ \type\tup{ 12,3 }=\tup{ 2,1 } $, $ \type\tup{ 3,12 }=\tup{ 1,2 } $ and $ \type\tup{ 123 }=\tup{ 3 } $.
\end{exmp}

\begin{exmp}
    \label{exa.F3type}
    Here are all set compositions of $\interval{3}$, grouped by type:
    \begin{align*}
        \begin{tabular}{r|l}
        \text{type} & \text{set compositions of this type} \\ \hline
        $\tup{3}$ & $\tup{ 123 }$; \\
        $\tup{2,1}$ & $\tup{ 12,3 }, \quad \tup{ 13,2 }, \quad \tup{ 23,1 }$; \\
        $\tup{1,2}$ & $\tup{ 1,23 }, \quad \tup{ 2,13 }, \quad \tup{ 3,12 } $; \\
        $\tup{1,1,1}$ & $\tup{ 1,2,3 }, \quad \tup{ 1,3,2 }, \quad \tup{ 2,1,3 }, \quad \tup{ 2,3,1 }, \quad \tup{ 3,1,2 }, \quad \tup{ 3,2,1 }$.
        \end{tabular}
    \end{align*}
\end{exmp}

The word ``face''\ harkens back to the bijection between set compositions and the faces of the $n$-th braid arrangement (see \cite[\S 4.1.4]{Bidigare-thesis} or \cite[\S 1.2.1]{Saliola}).
We will not use this bijection -- or any geometry for that matter -- but we will still often use the word ``face''\ for its brevity.

\begin{defn}[Face monoid/algebra]\phantomsection
    \label{def.facemon}\ %
    \begin{enumerate}[]
        \item\textbf{(a)} If $H = \tup{H_1,H_2,\ldots,H_m}$ is a list of sets, then its \defin{reduction} $H^{\red}$ means the list obtained by removing all the empty sets from $H$.
        \item\textbf{(b)} Recall that $\CF$ is the set of set compositions of $\interval{n}$. We define a multiplication on $\CF$ as follows:
		For $F=\tup{F_{1},F_{2},\dots,F_k} \in \CF$ and $G=\tup{G_{1},G_{2},\dots,G_{m}} \in \CF$, set
         \[
         \tup{F_1, F_2, \dots, F_k} \cdot \tup{G_1, G_2, \dots, G_m}
		 = \tup{F_1 \cap G_1, \dots, F_1 \cap G_m, \dots, F_k \cap G_m}^{\red},
         \]
        where the list on the right hand side contains all the pairwise intersections $F_i \cap G_j$ in the order of lexicographically increasing $\tup{i,j}$.
		In other words, $FG$ is obtained from $F$ by replacing each block $F_i$ with its $m$ intersections $F_i \cap G_1, \ F_i \cap G_2, \ \ldots, \ F_i \cap G_m$ (in this order) and throwing away the empty intersections.
        By Lemma~\ref{lem.monoid} below, this turns $\CF$ into a monoid.
        \item\textbf{(c)} 
        Let $\kf$ be the monoid algebra of the monoid $\CF$ over $\kk$. This is the free $\kk$-module with basis $\CF$, equipped with a multiplication that extends the multiplication of $\CF$ 
		in the usual bilinear way.
   \end{enumerate}
\end{defn}

\begin{exmp}
    Consider the two faces $ \tup{ 12,3 },\tup{ 3,2,1 } \in \CF\tup{ 3 } $. Then we have
    \begin{align*}
        \tup{ 12,3 }\tup{ 3,2,1 } &= \tup{ 12 \cap 3, 12 \cap 2, 12 \cap 1, 3 \cap 3, 3 \cap 2, 3 \cap 1 }^{ \red}\\ 
                                  &= \tup{ \varnothing, 2,1,3,\varnothing,\varnothing}^{ \red}=\tup{ 2,1,3 },
    \end{align*}
    while $ \tup{ 3,2,1 }\tup{ 12,3 }=\tup{ 3,2,1 } $. We also have
    \[
    \tup{ 12,3 }\tup{ 123 }=\tup{ 12 \cap 123, 3 \cap 123 }=\tup{ 12,3 },
    \]
    and more generally, the set composition $ \tup{ \interval{n} } $ is the unity of $ \CF\tup{ n } $ if $n\geq 1$. Let $ \kk = \ZZ $; then $ 3\tup{ 12,3 } $ and $ \tup{ 123 }-\tup{ 3,12 } $ are elements of the algebra $ \ZZ\CF\tup{ 3 } $. Their product is
    \begin{align*}
        3\tup{ 12,3 }\tup{ \tup{ 123 }-\tup{ 3,12 } } &= 3\tup{ \tup{ 12,3 }\tup{ 123 }-\tup{ 12,3 }\tup{ 3,12 } }\\
                                                      &= 3\tup{ \tup{ 12,3 }-\tup{ 12 \cap 3, 12 \cap 12, 3 \cap 3, 3 \cap 12 }^{ \red} }\\
                                                      &= 3\tup{ \tup{ 12,3 }-\tup{ 12,3 } }= 3 \cdot 0 = 0.
    \end{align*}
    Thus we see that $\kf$ has zero divisors in general. 
    
\end{exmp}

Note that the monoid $\CF$ (and thus the algebra $\kf$) is not commutative in general: For two faces $F$ and $G$, the products $FG$ and $GF$ have the same blocks, but usually appearing in different orders.

\begin{lem}
\label{lem.monoid}
The multiplication defined in Definition~\ref{def.facemon} \textbf{(b)} is associative and has neutral element $\tup{\interval{n}}$ (or $\tup{}$ if $n=0$); thus, $\CF$ really is a monoid.
\end{lem}

\begin{proof}
We define a \defin{set decomposition} of $\interval{n}$ in the same way as a set composition, but without requiring that its blocks be nonempty (so it is a tuple of disjoint sets
whose union is $\interval{n}$).
Thus, each set decomposition $F = \tup{F_{1},F_{2},\dots,F_k}$ gives rise to a set composition $F^{\red} = \tup{F_{1},F_{2},\dots,F_k}^{\red}$. Of course, when $F$ is itself a set composition, we just have $F^{\red} = F$.

Now we shall show the following \defin{auxiliary formula}:
\begin{quote}
\textbf{Auxiliary formula:}
If $\tup{F_{1},F_{2},\dots,F_k}$ and $\tup{G_{1},G_{2},\dots,G_{m}}$ are two set decompositions of $\interval{n}$, then the multiplication on $\CF$ satisfies
        \[
        \tup{F_{1},F_{2},\dots,F_{k}}^{\red} \cdot \tup{G_{1},G_{2},\dots,G_{m}}^{\red} = \tup{F_{1} \cap G_{1}, \dots,F_{1} \cap G_{m},\dots, F_{k} \cap G_{m}}^{\red}.
        \]
\end{quote}
Indeed, if a block $F_i$ is empty, then its intersections $F_i \cap G_j$ are empty for all $j$, and likewise an empty block $G_j$ gives rise to empty intersections $F_i \cap G_j$.
Hence, the right hand side of the auxiliary formula does not change if we remove an empty block from either $\tup{F_{1},F_{2},\dots,F_{k}}$ or $\tup{G_{1},G_{2},\dots,G_{m}}$.
Therefore, in particular, it does not change if we remove all the empty blocks, i.e., if we replace $\tup{F_{1},F_{2},\dots,F_{k}}$ and $\tup{G_{1},G_{2},\dots,G_{m}}$ by $\tup{F_{1},F_{2},\dots,F_{k}}^{\red}$ and $\tup{G_{1},G_{2},\dots,G_{m}}^{\red}$.
But then the above auxiliary formula just boils down to the definition of the multiplication on $\CF$.
So the auxiliary formula is proved.

Now, if $E = \tup{E_1, E_2, \ldots, E_p}$, $F = \tup{F_1, F_2, \ldots, F_q}$ and $G = \tup{G_1, G_2, \ldots, G_r}$ are three set compositions of $\interval{n}$, then we can use the above auxiliary formula to compute both products $\tup{EF} G$ and $E \tup{FG}$: Indeed, we get
\begin{align*}
F G  & = \tup{  F_{1},F_{2},\ldots,F_{q} }
\tup{ G_{1},G_{2},\ldots,G_{r} }
= \tup{ F_{1}\cap G_{1},\ldots,F_{1}\cap G_{r},\ldots,F_{q}\cap
G_{r} }^{\red}
\end{align*}
and thus
\begin{align*}
E \tup{F G}
& =
\underbrace{\tup{ E_{1},E_{2},\ldots ,E_{p} }}_{=
\tup{ E_{1},E_{2},\ldots,E_{p} }^{\red} }
\cdot\tup{ F_{1}\cap G_{1},\ldots,F_{1}\cap G_{r}
,\ldots,F_{q}\cap G_{r} }^{\red}\\
& = \tup{ E_{1},E_{2},\ldots,E_{p} }^{\red}
\cdot \tup{ F_{1}\cap G_{1},\ldots,F_{1}\cap G_{r},
\ldots,F_{q}\cap G_{r} }^{\red}\\
& = \tup{ E_{1}\cap\tup{ F_{1}\cap G_{1} }, \ldots,
E_{1}\cap\tup{ F_{1}\cap G_{r} }, \ldots,
E_{1}\cap\tup{ F_{q}\cap G_{r} }, \ldots,
E_{p}\cap\tup{ F_{q}\cap G_{r} } }^{\red} \\
& \qquad \qquad \qquad
\left(  \text{by the auxiliary formula}\right)  \\
& = \tup{ E_{1}\cap F_{1}\cap G_{1}, \ldots,
E_{p}\cap F_{q}\cap G_{r} }^{\red},
\end{align*}
where the triple intersections $E_i \cap F_j \cap G_k$ on the
right hand side are listed in the order of lexicographically
increasing $\tup{i, \tup{j, k}}$, or, equivalently, of
lexicographically increasing $\tup{i, j, k}$. A similar
computation yields the same expression for $\tup{E F} G$. Hence,
$E \tup{FG} = \tup{EF} G$, and thus associativity is proved.

It remains to show that $\tup{\interval{n}}$ (or $\tup{}$ if
$n=0$) is a neutral element for the multiplication on $\CF$.
But this is clear, since any block $F_i$ of any face $F \in \CF$
satisfies $\interval{n} \cap F_i = F_i \cap \interval{n} = F_i$.
\end{proof}

\begin{defn}[Actions of $S_n$ on $\CF$ and $\kf$; invariant subalgebra of $ \kf $]
    \label{kfsn_defn}
    Consider the left action of the symmetric group $ S_n $ on the set $ \CF $ that applies the permutation to each entry of each block of the set composition.
	That is, a permutation $ w \in S_n $ acts on a set composition $ \tup{ F_{1}, F_{2},\dots,F_k } \in \CF $ by
    \[
        w\tup{ F_{1}, F_{2},\dots,F_k } := \tup{ w\tup{ F_{1} }, w \tup{ F_{2} },\dots,w\tup{ F_k }}.
    \]
    This action extends linearly to an action of $S_n$ on $ \kf $. Then the invariant subalgebra $ \kfsn \subseteq \kf $ is defined by
    \[
    \kfsn := \set{\FF \in \kf : w\tup{ \FF } = \FF \text{ for all } w \in S_n}.
    \]
\end{defn}

\begin{prop}[Invariant subalgebra of $ \kf $]
    \label{kfsn_prop1}
    This invariant subalgebra $ \kfsn $ is actually a $\kk$-subalgebra of $\kf$.
\end{prop}

\begin{proof}
    \begin{enumerate}
    \item The action of $S_n$ on $\CF$ introduced in Definition~\ref{kfsn_defn} 
	respects the multiplication on $\CF$, meaning that any permutation $w \in S_n$ and any two set compositions $F, G \in \CF$ satisfy $w\tup{F \cdot G} = \tup{wF} \cdot \tup{wG}$.
    This is clear from the definitions, since $w\tup{F_i \cap G_j} = w\tup{F_i} \cap w\tup{G_j}$ for any blocks $F_i, G_j$ of $F, G$.
    \item Thus, $S_n$ acts on $\CF$ by monoid automorphisms, and hence (by linearity) acts on the monoid algebra $\kf$ by $\kk$-algebra automorphisms.
    Hence, the invariant subalgebra $\kfsn$ is a $\kk$-subalgebra of $\kf$
	(since this holds for every action by $\kk$-algebra automorphisms).
	\qedhere
	\end{enumerate}
\end{proof}

\begin{defn}
    \label{BBta_defn}
	For each composition $\alpha \models n$, we define the element
	\[
    \BBta := \sum_{\substack{F \in \CF; \\ \type F = \alpha}} F \in \kf.
    \]
\end{defn}

\begin{prop}[The $\BBta$ form a basis]
    \label{kfsn_prop2}
    The family of the elements $\BBta$, where $\alpha$ ranges over all compositions of $n$, is a basis of $\kfsn$.
\end{prop}

\begin{proof}
    We must show that the elements $\BBta$ actually belong to $\kfsn$ and form a basis of $\kfsn$.
    
    The $S_n$-representation $\kf$ is a permutation module: The action of $S_n$ permutes the basis $\CF$ of $\kf$.
	Thus, the invariant subspace $\kfsn$ has a basis (as a $\kk$-module) consisting of the orbit sums (i.e., of the sums $\sum_{F \in \mathcal{O}} F$ for each orbit $\mathcal{O}$ of the $S_n$-action on $\CF$).\ \ \ \ \footnote{What we are using here is the following folklore fact: If a group $G$ acts on a finite set $X$, then the invariant subspace $\tup{\kk X}^G := \set{\xx \in \kk X : w\xx = \xx \text{ for all } w \in G}$ of the permutation module $\kk X$ has a basis consisting of the orbit sums (i.e., of the sums $\sum_{x \in \mathcal{O}} x$ for each orbit $\mathcal{O}$ of the $G$-action on $X$).
	For a proof of this fact, see Proposition~\ref{prop.perm-fix-basis} in the Appendix (Subsection~\ref{subsec.omitted_proofs.perm-fix-basis}).}
    It remains to show that the elements $\BBta$ are precisely these orbit sums.\footnote{We invite the reader to take another look at Example~\ref{exa.F3type}. The rows of the table in that example are precisely the orbits of the $S_n$-action on $\CF$ for $n = 3$.}
    
    The $S_n$-action on $\CF$ preserves the type of a set composition: If $F \in \CF$ and $w \in S_n$, then $\type\tup{wF} = \type{F}$, since $\card{w\tup{F_i}} = \card{F_i}$ for each block $F_i$ of $F$.
    Hence, two set compositions $F$ and $G$ lying in the same $S_n$-orbit must have the same type.
    Conversely, if two set compositions $F$ and $G$ have the same type, then they lie in the same $S_n$-orbit (indeed, we can easily construct a permutation $w \in S_n$ that satisfies $wF = G$: just let it map the elements of each block of $F$ to the elements of the corresponding block of $G$, in some arbitrarily chosen order).
	Summarizing, we see that two set compositions $F$ and $G$ in $\CF$ lie in the same $S_n$-orbit if and only if they have the same type.

    Hence, each orbit $\mathcal{O}$ of the $S_n$-action on $\CF$ is the set of all set compositions $F \in \CF$ having a given type $\alpha \models n$.
	The orbit sum $\sum_{F \in \mathcal{O}} F$ of the former orbit can thus be written as $\sum_{\substack{F \in \CF; \\ \type F = \alpha}} F = \BBta$.
	Moreover, the correspondence between the orbits $\mathcal{O}$ and the types $\alpha \models n$ is a bijection\footnote{Injectivity is obvious (different orbits correspond to different types), and surjectivity is easy (for each composition $\alpha = \tup{\alpha_1, \alpha_2, \ldots, \alpha_k} \models n$, there is at least one set composition $F \in \CF$ of type $\alpha$, since we can certainly subdivide $\interval{n}$ into subsets of sizes $\alpha_1, \alpha_2, \ldots, \alpha_k$).}.
    Therefore, the orbit sums $\sum_{F \in \mathcal{O}} F$ are precisely 
	the elements $\BBta$. This completes our proof.
\end{proof}

\subsection{Bidigare's anti-isomorphism}

As we know, the family $\tup{\BB_I}_{I \subseteq \interval{n-1}}$ is a basis of the descent algebra $\dsn$ (as $\kk$-module).
Reindexing this family using compositions instead of subsets (see Definition~\ref{dsn_def} \textbf{(c)}), we conclude that the family $\tup{\BBa}_{\alpha \models n}$ is a basis of $\dsn$.
This allows us to define the following crucial map (due to Bidigare \cite[Theorem 3.8.1]{Bidigare-thesis}):

\begin{defn}
\label{defn_rho}
	Let $\rho : \dsn \to \kfsn$ be the $\kk$-linear map defined by
    \begin{align}
    \rho\tup{ \BBa } = \BBta \qquad \text{ for each } \alpha \models n
    \label{eq.rhoBa}
    \end{align}
    and extended linearly to all of $\dsn$.
\end{defn}

(We have used the fact that the $\BBta$ belong to $\kfsn$; this is part of Proposition~\ref{kfsn_prop2}.)

\begin{thm}[Bidigare]
    \label{dsn_kfsn_morph}
    The set $\dsn$ is a $\kk$-subalgebra of $\ksn$.
    The map $ \rho : \dsn \to \kfsn $ is a $ \kk $-algebra anti-isomorphism. 
\end{thm}

\begin{proof}
    This is Bidigare's result \cite[Theorem 3.8.1]{Bidigare-thesis}, and appears all across the literature.
	We give a self-contained proof in the Appendix (Subsection~\ref{subsec.omitted_proofs.rho}).
    Here, however, let us outline how it can be obtained from more recent sources:
    Saliola, in \cite[proof of Theorem 2.1]{Saliola}, constructs a $\kk$-algebra anti-morphism $\xi : \kfsn \to \ksn$ and shows that it satisfies $\xi\tup{\BBta} = \BBa$ for each composition $\alpha \models n$ (in his notations, this takes the form $\xi\tup{a_B} = x_J$, with $B$ being a set composition of type $\alpha$ and $J$ being the set $\gaps^{-1}\tup{\alpha}$).
    The same is also proved in \cite[Theorem 1]{Hsiao} (take $G = \set{1}$ and observe that $\sigma_\alpha = \BBta$ and $X_\alpha = \BBa$) and in \cite[Appendix B, a few lines above Remark B.5]{BlessenohlSchocker} (their $q$ is our $\alpha$, so that $X^q = \BBta$ and $\Xi^q = \BBa$),
    as well as in \cite[Theorem 7]{Brown-SRM}.
	
	Either way, we now know that the $\kk$-algebra anti-morphism $\xi : \kfsn \to \ksn$ satisfies $\xi\tup{\BBta} = \BBa$ for each composition $\alpha \models n$.
    Since the $\BBta$ form a basis of $\kfsn$ (by Proposition~\ref{kfsn_prop2}), this shows that the image of $\xi$ is the span of the $\BBa$'s for $\alpha \models n$, which of course is the descent algebra $\dsn$.
	Hence, $\dsn$ is a $\kk$-subalgebra of $\ksn$ (since the image of a $\kk$-algebra anti-morphism is always a $\kk$-subalgebra of its target).
	Moreover, the $\kk$-algebra anti-morphism $\xi$ is clearly injective (since the $\BBa$'s are $\kk$-linearly independent), and thus is a $\kk$-algebra anti-\textbf{iso}morphism from $\kfsn$ to its image $\dsn$.
    Thus, we have found an anti-isomorphism from $\kfsn$ to $\dsn$ that sends each $\BBta$ to $\BBa$.
    Its inverse must obviously be our map $\rho$.
	This shows that $\rho$ is a $\kk$-algebra anti-isomorphism.
\end{proof}

Using this anti-isomorphism $\rho$ we can transport our identities from the descent algebra to the face algebra and back.
Let us set
\begin{equation}
\wotil := \rho\tup{\wo}
\label{eq.wotil=}
\end{equation}
(this is allowed, since $\wo \in \dsn$).
Applying $ \rho $ to Proposition \ref{w0conj_prop}, we obtain the following corollary:

\begin{cor}
    \label{w0conj_cor}
    For any $\alpha \models n$, we have
    \begin{align}
    \wotil \BBta \wotil = \BBt_{\rev \alpha}
    \label{eq.w0conj_cor.1}
    \end{align}
    and
    \begin{align}
    \tup{\wotil \BBta}^2 = \BBt_{\rev \alpha} \BBta .
    \label{eq.w0conj_cor.2}
    \end{align}
\end{cor}

\begin{proof}
Let $\alpha \models n$.
Then, Corollary~\ref{cor.woBBawo} yields $\wo \BBa \wo = \BB_{\rev \alpha}$.
Applying the algebra anti-morphism $\rho$ to this equality, we obtain precisely \eqref{eq.w0conj_cor.1}
(since $\rho$ sends $\wo$ to $\wotil$ and sends each $\BB_\beta$ to $\BBt_\beta$).

To prove \eqref{eq.w0conj_cor.2}, we multiply the equality \eqref{eq.w0conj_cor.1} by $\BBta$ from the right, obtaining
\begin{align*}
\wotil \BBta \wotil \BBta = \BBt_{\rev \alpha} \BBta.
\end{align*}
This proves \eqref{eq.w0conj_cor.2}, since the left hand side here is $\tup{\wotil \BBta}^2$.
\end{proof}

To make this more explicit, we give an explicit formula for $\wotil = \rho\tup{\wo}$:

\begin{prop}
    \label{w0face.prop}
    We have
    \begin{align}
        \wotil = \sum_{F \in \CF} \tup{ -1 }^{n-\ell\tup{ F }}F.
        \label{eq.w0face}
    \end{align}
\end{prop}

\begin{proof}
    We assume that $n \geq 1$, since the $n = 0$ case is trivial.
    Applying the $\kk$-linear map $\rho$ to both sides of \eqref{eq.w0_dsn_cor.eq}, we find
    \begin{align*}
        \rho\tup{ \wo }
		&= \sum_{I \subseteq \interval{n-1}} \underbrace{\tup{ -1 }^{n-\card{I}-1}}_{\substack{= \tup{-1}^{n - \ell\tup{\gaps I}} \\ \text{(since $\ell\tup{\gaps I} = \card{I} + 1$)}}} \ \underbrace{\rho\tup{ \BBi }}_{\substack{= \rho\tup{\BB_{\gaps I}} \\ \text{(since }\BBi = \BB_{\gaps I} \\ \text{by Definition \ref{dsn_def} \textbf{(c)})}}}
        = \sum_{I \subseteq \interval{n-1}} \tup{ -1 }^{n-\ell(\gaps I)}\rho\tup{ \BB_{\gaps I} }.
    \intertext{
    Since $\gaps : \mathcal{P}\tup{\interval{n-1}} \to C\tup{n}$ is a bijection, we can now substitute $\alpha$ for $\gaps I$ on the right hand side, and this becomes
    }
        \rho\tup{ \wo }
        &= \sum_{\alpha \models n} \tup{ -1 }^{n-\ell\tup{ \alpha }}\rho\tup{ \BBa }\\
         &= \sum_{\alpha \models n} \tup{ -1 }^{n-\ell\tup{ \alpha }}\BBta
            \qquad \tup{\text{by \eqref{eq.rhoBa}}} \\
         &= \sum_{\alpha \models n} \tup{ -1 }^{n-\ell\tup{ \alpha }} \sum_{\substack{F \in \CF; \\ \type F = \alpha}} F
            \qquad \tup{\text{by the definition of $\BBta$}} \\
         &= \sum_{\alpha \models n}\sum_{\substack{F \in \CF; \\ \type F = \alpha}} \tup{ -1 }^{n-\ell\tup{ \alpha }}F
		 = \sum_{F \in \CF} \tup{ -1 }^{n-\ell\tup{ F }}F
    \end{align*}
    (since each set composition $F$ has a unique type $\alpha$ and satisfies $\ell\tup{\alpha} = \ell\tup{F}$).
    Since the left hand side is $\wotil$, we have thus proved the proposition.
\end{proof}

Thanks to the anti-isomorphism $\rho$, we can use the element $\wotil \BBta$ of $\kfsn$ as a stand-in for the elements $\BBa \wo$ and $\wo \BBa$ of $\dsn$.
The precise way this works is formalized in the following lemma:

\begin{lem}
\label{lem.standin}
Let $f\tup{x} \in \kk\sbr{x}$ be a polynomial.
Let $\alpha \models n$.
Then, we have the equivalences
\[
\tup{f\tup{\BBa \wo} = 0}
\ \Longleftrightarrow\ %
\tup{f\tup{\wo \BBa} = 0}
\ \Longleftrightarrow\ %
\tup{f\tup{\wotil \BBta} = 0}.
\]
\end{lem}

\begin{proof}
The map $\rho$ is a $\kk$-algebra anti-isomorphism
(by Theorem~\ref{dsn_kfsn_morph}).
Thus,
\begin{align}
\rho\tup{\BBa\wo}
& = \rho\tup{\wo} \rho\tup{\BBa}
= \wotil \BBta
\label{pf.lem.standin.rhoBw}
\end{align}
(since $\rho\tup{\wo} = \wotil$ and
$\rho\tup{\BBa} = \BBta$).

Let $\omega$ be the map $\dsn\to\dsn, \ a \mapsto \wo a \wo^{-1}$.
This map $\omega$ is simply conjugation by $\wo$ in the
$\kk$-algebra $\dsn$.
Thus, it is a $\kk$-algebra isomorphism, and sends $\BBa\wo$
to
\begin{equation}
\omega\tup{\BBa\wo}
= \wo\tup{\BBa\wo}\wo^{-1} = \wo\BBa .
\label{pf.lem.standin.omBw}
\end{equation}

However, polynomials are known to commute with algebra
morphisms and anti-morphisms:
That is, if $A$ and $B$ are any two $\kk$-algebras,
if $\phi : A \to B$ is a $\kk$-algebra morphism or
anti-morphism, and if $g\tup{x} \in \kk\sbr{x}$ is any
polynomial, then
\begin{equation}
\phi\tup{g\tup{a}} = g\tup{\phi\tup{a}}
\qquad
\text{ for any } a \in A .
\label{pf.lem.standin.phig}
\end{equation}

Applying this fact to $A = \dsn$ and $B = \kfsn$ and
$\phi = \rho$ (a $\kk$-algebra anti-morphism) and
$g\tup{x} = f\tup{x}$ and $a = \BBa\wo$, we obtain
$\rho\tup{f\tup{\BBa\wo}} = f\tup{\rho\tup{\BBa\wo}}
= f\tup{\wotil \BBta}$ (by \eqref{pf.lem.standin.rhoBw}).

Applying \eqref{pf.lem.standin.phig} to $A = \dsn$ and
$B = \dsn$ and $\phi = \omega$ (a $\kk$-algebra morphism) and
$g\tup{x} = f\tup{x}$ and $a = \BBa\wo$, we obtain
$\omega\tup{f\tup{\BBa\wo}} = f\tup{\omega\tup{\BBa\wo}}
= f\tup{\wo\BBa}$ (by \eqref{pf.lem.standin.omBw}).

Now, the map $\rho$ is a $\kk$-module isomorphism
(being a $\kk$-algebra anti-isomorphism).
Hence, we have the equivalences
\[
\tup{f\tup{\BBa \wo} = 0}
\ \Longleftrightarrow\ %
\tup{\rho\tup{f\tup{\BBa\wo}} = 0}
\ \Longleftrightarrow\ %
\tup{f\tup{\wotil \BBta} = 0}
\]
(since $\rho\tup{f\tup{\BBa\wo}} = f\tup{\wotil \BBta}$).
Furthermore, the map $\omega$ is a $\kk$-algebra
isomorphism. Hence, we have the equivalences
\[
\tup{f\tup{\BBa \wo} = 0}
\ \Longleftrightarrow\ %
\tup{\omega\tup{f\tup{\BBa\wo}} = 0}
\ \Longleftrightarrow\ %
\tup{f\tup{\wo \BBa} = 0}
\]
(since $\omega\tup{f\tup{\BBa\wo}} = f\tup{\wo \BBa}$).
These two equivalences prove the lemma.
\end{proof}

A similar (but simpler) argument shows the following:

\begin{lem}
\label{lem.standin2}
Let $f\tup{x} \in \kk\sbr{x}$ be a polynomial.
Let $\alpha \models n$.
Then, we have the equivalence
\[
\tup{f\tup{\BBa} = 0}
\ \Longleftrightarrow\ %
\tup{f\tup{\BBta} = 0}.
\]
\end{lem}

\begin{proof}
Left to the reader.
\end{proof}

\subsection{The containment relation}

We return to the combinatorics of set compositions.

\begin{defn}
    Let $F = \tup{F_1, F_2, \dots, F_k}$ and $G = \tup{G_1, G_2, \dots, G_m}$ be two faces in $\CF$.
    We say $F$ is \defin{contained} in $G$ -- and we write $F \con G$ -- if every block of $F$ is a subset of a block of $G$. (More formally: if for each $i \in \interval{k}$, there is a $j \in \interval{m}$ such that $F_i \subseteq G_j$.)
\end{defn}

\begin{exmp}
   If $G = \tup{45,123}$ (we omit the commas and the set braces, so for example $45$ means $\set{4,5}$), then the faces $\tup{5,4,123}$, $\tup{123,4,5}$, $\tup{45,2,13}$, $\tup{5,1,2,4,3}$ (and several others) are contained in $G$, while $\tup{5,134,2}$ is not (as are many others).
\end{exmp}

Clearly, the relation $\con$ is transitive and reflexive.
It is not antisymmetric, because reordering the blocks of a face $F$ yields a new face $G$ that satisfies $F \con G \con F$.
More generally, the relation $F \con G$ depends only on the \textbf{set} of blocks of $F$ and the \textbf{set} of blocks of $G$, not on their respective orders.
Thus, the relation $\con$ is better understood as a relation between set partitions (known as the \defin{refinement order} of set partitions), not set compositions.
However, we will find it useful on set compositions.

If a face $F$ is contained in a face $G$, then $G$ can be obtained from $F$ by merging and reordering its blocks. What blocks are merged and where they are placed can be encoded by a set composition of $\interval{\ell \tup{F}}$, which is obtained from $F$ by renaming the blocks of $F$ as $1,2,\ldots,\ell\tup{F}$ from left to right.
For example, for $F = \tup{123,4,5}$ and $G = \tup{45,123}$, we can obtain $G$ from $F$ by merging the last two blocks and swapping the result with the first block; this is encoded by the set composition $\tup{23, 1}$ of $\interval{\ell\tup{F}} = \interval{3}$.
In full generality, the encoding is described by the following proposition:

\begin{prop}
    \label{ccb_prop}
     Let $F = \tup{F_1, F_2, \ldots, F_m} \in \CF$ be a face (i.e., a set composition of $\interval{n}$). Then, the map
     \begin{align*}
        f : \set{G \in \CF : F \con G} &\to \set{H : H\models \interval{m}}, \\
        \tup{G_1, G_2, \ldots, G_k} &\mapsto \tup{H_1, H_2, \ldots, H_k} \text{ where } H_{i} = \set{j \in \interval{m} : F_{j} \subseteq G_{i}}
    \end{align*}
    is a bijection. This bijection is furthermore length-preserving:
    \[
    \ell \tup{f\tup{G}} = \ell \tup{G}.
    \]
\end{prop}
The proof of this proposition is a standard argument in combinatorics (see, e.g., \cite[Example 3.10.4]{EC1} for the analogue for unordered set partitions); we give it in the Appendix (Subsection \ref{subsec.omitted_proofs.ccb_prop}).

The following simple properties of the containment are also easy to prove:

\begin{prop}
    \label{cp_prop}
    Let $F$ and $G$ be two faces in $\CF$. Then:
    \begin{enumerate}[]
        \item\textbf{(a)} If $F$ is contained in $G$, then $\ell \tup{F} \ge \ell \tup{G}$.
        \item\textbf{(b)} The product $FG$ is contained in both $F$ and $G$.
        \item\textbf{(c)} We have $FG = F$ if and only if $F$ is contained in $G$.
        \item\textbf{(d)} We have $\ell \tup{FG} \ge \ell \tup{F}$. Equality holds if and only if $F$ is contained in $G$.
        \item\textbf{(e)} We have $\ell \tup{FG} \ge \ell \tup{G}$.
    \end{enumerate}
\end{prop}

We again refer to the Appendix (Subsection \ref{subsec.omitted_proofs.cp_prop}) for the proof of Proposition~\ref{cp_prop}.

\subsection{The knapsack numbers $ \na\tup{F} $}

The following definition will be crucial to the study of $\BBa$ and related elements:

\begin{defn}
    \label{knap_num_def}
    Let $\alpha \models n$ be a composition.
	
	\begin{enumerate}
	
	\item[\textbf{(a)}] Given a face $F \in \CF$, we define the finite set $\Na\tup{F}$ by
	\[
    \Na\tup{F} = \set{G \in \CF  : F \con G \text{ and } \type G = \alpha},
    \]
    and we define the nonnegative integer $\na\tup{F}$ by
    \[
    \na\tup{F} = \card{\Na\tup{F}}.
    \]
	Thus, $\na\tup{F}$ is the number of faces $G \in \CF$ satisfying $F \con G$ and $\type G = \alpha$.
    We refer to $\na\tup{F}$ as the \defin{knapsack number} of $F$ with respect to $\alpha$.

	\item[\textbf{(b)}]
	We will furthermore write $\na\tup{\CF}$ for the set
	$\set{\na\tup{F} : F \in \CF}$, which consists
	of the knapsack numbers of all faces.
	This is the image of $ \CF $ under the map $\na$.
	It is obviously a finite set of nonnegative integers.
	
	\end{enumerate}
\end{defn}

\begin{exmp}
Let $n = 4$ and $\alpha = \tup{2,2}$. Then:
\begin{enumerate}
\item If $F = \tup{4,23,1}$, then
\[
\Na\tup{F} = \set{\tup{14,23}, \ \tup{23,14}}, \qquad \text{and thus }
\na\tup{F} = 2.
\]
\item If $F = \tup{123,4}$, then $\Na\tup{F} = \varnothing$ and thus $\na\tup{F} = 0$.
\item If $F = \tup{3,2,1,4}$, then $\Na\tup{F}$ consists of all six faces of type $\alpha = \tup{2,2}$, and thus $\na\tup{F} = 6$.
\item Systematically computing $\na\tup{F}$ for all faces $F \in \CF$ shows that $\na\tup{\CF} = \set{0, 2, 6}$.
\end{enumerate}
\end{exmp}

\begin{exmp}
\label{exa.knap_num_min}
Let $n \geq 1$ and $\alpha = \tup{n}$.
Then, the only face $G \in \CF$ of type $\alpha$ is $\tup{\interval{n}}$.
Each face $F \in \CF$ satisfies $F \con \tup{\interval{n}}$
and thus
\[
N_{\tup{n}}\tup{F}
= \set{G \in \CF  : F \con G \text{ and } \type G = \tup{n}}
= \set{\tup{\interval{n}}},
\]
so that
\begin{align}
n_{\tup{n}}\tup{F} = \card{\set{\tup{\interval{n}}}} = 1.
\label{eq.exa.knap_num_min.n}
\end{align}
\end{exmp}

The word ``knapsack number'' refers to the knapsack problem in combinatorial optimization.
Indeed, given a composition $\alpha = \tup{\alpha_1, \alpha_2, \ldots, \alpha_k} \models n$ and a face $F \in \CF$, we can view the knapsack number $\na\tup{F}$ as the number of ways to distribute the blocks of $F$ into $k$ (labeled) bags such that the total size of the blocks in the $i$-th bag is $\alpha_i$ for each $i$.
(In fact, such a distribution corresponds to a face $G \in \Na\tup{F}$, namely the face $G = \tup{G_1, G_2, \ldots, G_k}$ whose $i$-th block $G_i$ is the union of the blocks of $F$ in the $i$-th bag.)

We shall now show some simple properties of knapsack numbers.

\begin{prop}
    \label{na.prop.00}
    Let $\alpha \models n$.
	Let $F,G \in \CF$ be such that $F \con G$.
    Then, we have $\na\tup{F} \geq \na\tup{G}$.
\end{prop}

\begin{proof}
   Recall that $\na\tup{F} = \card{\Na\tup{F}}$ and
   likewise $\na\tup{G} = \card{\Na\tup{G}}$.
   However, for each $H \in \Na\tup{G}$, we have
   $G \con H$ and $\type H = \alpha$ (by the definition
   of $\Na\tup{G}$), so that
   $F \con G \con H$
   (since the containment relation is transitive),
   and thus $H \in \Na\tup{F}$ (since $F \con H$ and
   $\type H = \alpha$).
   Thus, we have shown that
   \[
   \Na\tup{G}
   \subseteq
   \Na\tup{F}.
   \]
   Hence,
   $\card{\Na\tup{G}}
   \leq
   \card{\Na\tup{F}}$.
   In other words, $\na\tup{G} \le \na\tup{F}$
   (since $\na\tup{F} = \card{\Na\tup{F}}$ and
   likewise $\na\tup{G} = \card{\Na\tup{G}}$).
   This yields the proposition.
\end{proof}

\begin{prop}
    \label{na.prop.0}
    Let $\alpha \models n$.
	Let $F,G \in \CF$.
    Then, we have $\na\tup{FG} \geq \na\tup{F}$ and $\na\tup{FG} \geq \na\tup{ G }$.
\end{prop}

\begin{proof}
    Proposition \ref{cp_prop} \textbf{(b)} yields $FG \con F$ and $FG \con G$.
    Hence, by Proposition~\ref{na.prop.00}, we conclude that
    $\na\tup{FG} \geq \na\tup{F}$ and $\na\tup{FG} \geq \na\tup{ G }$.
\end{proof}

\begin{prop}
    \label{na.prop.1}
    Let $\alpha \models n$.
	Let $F,G \in \CF$ be such that $\type G = \alpha$ and $F \not\con G$.
	Then, we have $\na\tup{FG} > \na\tup{F}$.
\end{prop}

\begin{proof}
   Recall that $\na\tup{F} = \card{\Na\tup{F}}$ and
   likewise $\na\tup{FG} = \card{\Na\tup{FG}}$.
   Hence, in order to prove that $\na\tup{FG} > \na\tup{F}$, 
   it suffices to show that $\Na\tup{F}$ is a proper subset of $\Na\tup{FG}$.
   
   It is easy to see that $\Na\tup{F}$ is a subset of $\Na\tup{FG}$: In fact, any $H \in \Na\tup{F}$ satisfies $F \con H$ and thus (by Proposition \ref{cp_prop} \textbf{(b)}) also $FG \con F \con H$, so that $H \in \Na\tup{FG}$.
   
   It remains to show that the subset $\Na\tup{F}$ of $\Na\tup{FG}$ is proper.
   For this, we must find an element that belongs to $\Na\tup{FG}$ but not to $\Na\tup{F}$.
   But $G$ is such an element: Indeed, $G \in \Na\tup{FG}$ (because Proposition \ref{cp_prop} \textbf{(b)} shows that $FG \con G$ and $\type G = \alpha$) and $G \notin \Na\tup{F}$ (since $F \not\con G$).
   Thus, the proof is finished.
\end{proof}

\begin{prop}
    \label{na.prop.2}
    Let $\alpha \models n$ and $F \in \CF$.
	Then, $\na\tup{F} = n_{\rev \alpha}\tup{F}$.
\end{prop}

\begin{proof}
    If $F$ is contained in some face $G$, then $F$ will still be contained in any face obtained by reordering the blocks of $G$.
	Thus, in particular, $F$ is contained in a face $G = \tup{G_{1},G_{2},\dots,G_{k}}$ if and only if it is contained in its reverse $\rev G := \tup{G_{k},G_{k-1},\dots,G_{1}}$.
	Furthermore, a face $G$ satisfies $\type G = \alpha$ if and only if $\type\tup{\rev G} = \rev(\alpha)$, since we always have $\type\tup{\rev G} = \rev(\type G)$.
	Thus, the map
	\begin{align*}
	\Na\tup{F} &\to N_{\rev \alpha}\tup{F}, \\
	G &\mapsto \rev G
	\end{align*}
	is well-defined.
	This map is furthermore a bijection, since an inverse map can be defined in the exact same way (note that $\rev\tup{\rev \alpha} = \alpha$).
	By the bijection principle, we thus obtain
	$\card{\Na\tup{F}} = \card{N_{\rev \alpha}\tup{F}}$.
	In other words, $\na\tup{F} = n_{\rev \alpha}\tup{F}$.
\end{proof}

\begin{rem}
    Let $\alpha \models n$ and $F \in \CF$.
	In our above proof, we essentially showed that $\na\tup{F}$ does not change if we permute the entries of the composition $\alpha$.
	But there are further symmetries: $\na\tup{F}$ also remains unchanged if we permute the blocks of $F$ or apply a permutation $w \in S_n$ to $F$.
	Hence, $\na\tup{F}$ depends on the \textbf{unordered} type of $F$ (that is, of the partition of $n$ obtained from the composition $\type F$ by sorting the entries in decreasing order) rather than on $F$ itself.
	In particular, the size of the set $\na\tup{\CF}$ is at most the number $p\tup{n}$ of partitions of $n$.
\end{rem}

\section{\label{sec.longpol}The knapsack filtration and the first annihilating polynomials}

From now on, \textbf{we fix a composition $\alpha$ of $n$}.

Our goal in this section is to establish polynomials that annihilate\footnote{We say that a polynomial $f \in \kk\sbr{x}$ \defin{annihilates} an element $a$ of a $\kk$-algebra if $f\tup{a} = 0$.} $\BBa$ and $\wo \BBa$ (Theorem~\ref{main.thm.3a} and Theorem~\ref{thm.main.1}, respectively).
The former polynomial is well-known, and the latter is suboptimal (i.e., has higher degree than necessary);
however, these polynomials make natural stepping stones on the way to our main result (Theorem~\ref{main.thm.3}), allowing us to introduce the relevant ideas one at a time.

\subsection{A polynomial annihilating $\BBa$}

The following theorem (a particular case of \cite[Proposition 3.12]{ncsf2} and of \cite[Theorem 4.1]{Schocker})
gives a polynomial that annihilates $\BBa$:

\begin{thm}
\label{main.thm.3a}
The element $\BBa \in \dsn$ satisfies
\begin{equation}
\prod_{k\in \na\tup{\CF}} \tup{\BBa - k}
=0.
\label{eq.main.3a}
\end{equation}
(This product ranges over all \textbf{distinct}
elements $k \in \na\tup{\CF}$. Thus, even if several different
faces $F$ have the same knapsack number $k$, there
is only one $\BBa - k$ factor in the product.)
\end{thm}

To prove this theorem, we need some preparations.
We begin with a definition, which will be equally important in further proofs later on:

\begin{defn}[The knapsack filtration of $\kf$]\phantomsection
    \label{def.ksf}\ %
    \begin{enumerate}[]
        \item\textbf{(a)} For each $k \in \NN \cup \set{\infty}$, we define
		$\CF_{k} := \set{F \in \CF : \na\tup{F} \ge k}$.
		(Thus, $\CF_\infty = \varnothing$.)
		
		We furthermore let $\kf_k$ denote the $\kk$-linear span of $\CF_k$ inside $\kf$.
		
        \item\textbf{(b)} Write the set $ \na\tup{ \CF } $ as $ \set{k_0 < k_1 < k_2 < \cdots < k_m} $.
        Set $k_{m+1} = \infty$.
        Then the $\CF_{k_i}$ form a filtration of $\CF$:
            \[
            \CF = \CF_{k_{0}} \supseteq \CF_{k_{1}} \supseteq \cdots \supseteq \CF_{k_{m}} \supseteq \CF_{k_{m+1}} = \varnothing.
            \]
        Therefore their $\kk$-linear spans form a filtration of $\kf$:
            \[
            \kf = \kf_{k_{0}} \supseteq \kf_{k_{1}} \supseteq \cdots \supseteq \kf_{k_{m}} \supseteq \kf_{k_{m+1}} = 0.
            \]
        We shall call this the \defin{knapsack filtration} of $\kf$.
    \end{enumerate}
\end{defn}

To illustrate Definition~\ref{def.ksf}, we compute the smallest element $k_0$ of $\na\tup{\CF}$:

\begin{prop}
\label{prop.k0}
Assume that $n \geq 1$.
\begin{enumerate}
\item[\textbf{(a)}] If $\alpha \neq \tup{n}$, then $k_0 = 0$.
\item[\textbf{(b)}] If $\alpha = \tup{n}$, then $k_0 = 1$ and $m = 0$.
\end{enumerate}
\end{prop}

\begin{proof}
    \textbf{(a)} Assume that $\alpha \neq \tup{n}$.
	Thus,
    \begin{align*}
        \Na\tup{ \tup{ \interval{n} } }
		&= {\set{G \in \CF : \tup{ \interval{n} } \con G \text{ and } \type G = \alpha }}
		= \varnothing ,
    \end{align*}
    since the face $ \tup{ \interval{n} } $ is contained only in itself and is of type $ \tup{ n } \ne \alpha $ by assumption.
	Therefore, $\na\tup{ \tup{ \interval{n} } } = 0$,
	so that $0 = \na\tup{ \tup{ \interval{n} } } \in \na\tup{ \CF } $.
	Since $\na\tup{ \CF }$ is a set of nonnegative integers, this shows that $0$ is its smallest element; that is, $ k_{0}=0 $.

%
    \textbf{(b)} Assume that $\alpha = \tup{n}$.
	For every face $ F \in \CF $, we have
	$n_{\tup{ n }}\tup{F} = 1$ by
	\eqref{eq.exa.knap_num_min.n}.
	Consequently $ n_{\tup{ n }}\tup{ \CF } = \set{1} $,
	and thus $ k_{m}=k_{0}=1 $ and $m = 0$, as desired.
\end{proof}

The following simple fact will be used without mention:

\begin{prop}
\label{prop.kfk-ideal}
Let $k \in \NN \cup \set{\infty}$.
Then, $\kf_k$ is a (two-sided) ideal of $ \kf $.
\end{prop}

\begin{proof}
The set $\kf_k$ is defined as a $\kk$-linear span, and thus is closed under addition and contains zero.
It remains to prove that it satisfies the absorption axiom -- i.e., that if one of two elements $\FF, \GG \in \kf$ belongs to $\kf_k$, then $\FF \GG \in \kf_k$ as well.
By linearity, it suffices to show that if one of two faces $F, G \in \CF$ belongs to $\CF_k$, then $FG \in \CF_k$ as well (since $\kf_k$ and $\kf$ are the spans of $\CF_k$ and $\CF$).

But this is easy:
Let $F, G \in \CF$ be two faces. Assume that one of them belongs to $\CF_k$.
That is, one of the knapsack numbers $\na\tup{F}$ and $\na\tup{G}$ is $\geq k$.
Hence, we have $\max\set{\na\tup{F}, \na\tup{G}} \geq k$.
But Proposition \ref{na.prop.0} shows that
$\na\tup{FG} \geq \na\tup{F}$ and $\na\tup{FG} \geq \na\tup{G}$,
so that
$\na\tup{FG} \geq \max\set{\na\tup{F}, \na\tup{G}} \geq k$.
In other words, $FG \in \CF_k$.
This is just what we needed to prove.
\end{proof}

This shows that the knapsack filtration of $\kf$ is a filtration by ideals (i.e., by $\tup{\kf,\kf}$-bimodules).
Hence, the subquotients $\kf_{k_i} / \kf_{k_{i+1}}$ of this filtration are $\kf$-$\kf$-bimodules.
The following lemma says that right multiplication by $\BBta$ acts on these subquotients as scaling by $k_i$:

\begin{lem}
    \label{main.lem.1}
    Let $ F \in \CF $. Let $i \in \interval{0, m}$ be such that $\na\tup{F} = k_i$. Then we have
    \[
    F\BBta \equiv k_i F  \mod  \kf_{k_{i+1}}.
    \]
\end{lem}

\begin{proof}
   The definition of $\BBta$ yields $\BBta = \sum_{\substack{G \in \CF; \\ \type G = \alpha}}G$. Hence,
   \begin{align}
       F\BBta
	   &= \sum_{\substack{G \in \CF; \\ \type G = \alpha}} FG
	   = \sum_{\substack{G \in \CF; \\ \type G = \alpha; \\ F \con G}} \,\, \underbrace{FG}_{\substack{=F \\ \text{(by Proposition \ref{cp_prop} \textbf{(c)})}}}
	   + \sum_{\substack{G \in \CF; \\ \type G = \alpha; \\ F \not\con G}} FG
	   \nonumber\\
       &= \sum_{\substack{G \in \CF; \\ \type G = \alpha; \\ F \con G}} F + \sum_{\substack{G \in \CF; \\ \type G = \alpha; \\ F \not\con G}} FG.
	   \label{eq.main.lem.1.1}
   \end{align}
   But the sum $\sum_{\substack{G \in \CF; \\ \type G = \alpha; \\ F \con G}} F$ can be rewritten as $\sum_{G \in \Na\tup{F}} F$ (by the definition of $\Na\tup{F}$).
   Hence,
   \begin{align}
		\sum_{\substack{G \in \CF; \\ \type G = \alpha; \\ F \con G}} F
		= \sum_{G \in \Na\tup{F}} F
		= \underbrace{\card{\Na\tup{F}}}_{=\na\tup{F} = k_i} F
		= k_i F.
		\label{eq.main.lem.1.3}
   \end{align}
   On the other hand, if $G \in \CF$ satisfies $\type G = \alpha$ and $F \not\con G$, then Proposition \ref{na.prop.1} yields $ \na\tup{ FG } > \na\tup{ F } = k_i $ and thus $ \na\tup{ FG } \ge k_{i+1}$ (since $ \na\tup{ FG } $ belongs to the set $ \na\tup{ \CF } = \set{k_0 < k_1 < k_2 < \cdots  < k_m} $, and thus cannot be $> k_i$ without being $\geq k_{i+1}$), so that $FG \in \CF_{k_{i+1}} \subseteq \kf_{k_{i+1}}$ and therefore
   \begin{align}
		FG \equiv 0 \mod \kf_{k_{i+1}}.
	   \label{eq.main.lem.1.2}
   \end{align}
   Thus, \eqref{eq.main.lem.1.1} becomes
   \begin{align*}
       F\BBta
       &= \underbrace{\sum_{\substack{G \in \CF; \\ \type G = \alpha; \\ F \con G}} F}_{\substack{= k_i F \\ \text{(by \eqref{eq.main.lem.1.3})}}}
	   + \sum_{\substack{G \in \CF; \\ \type G = \alpha; \\ F \not\con G}}\ \ \underbrace{FG}_{\substack{\equiv 0 \mod \kf_{k_{i+1}} \\ \text{(by \eqref{eq.main.lem.1.2})}}}
	   \equiv k_i F \mod \kf_{k_{i+1}}.
   \end{align*}
\end{proof}

\begin{lem}
    \label{main.lem.1a}
    Let $i \in \interval{0, m}$. If $\FF \in \kf_{k_i}$, then
    \[
    \FF \tup{\BBta - k_i} \in \kf_{k_{i+1}}.
    \]
\end{lem}

\begin{proof}
By linearity, it suffices to prove that the lemma holds when $\FF$ is a single face $ F \in \CF_{k_i}$.
Thus, let us assume that this is the case; thus, $ \na\tup{ F } \ge k_i $.
Thus, we have either $ \na\tup{ F } > k_i $ or $ \na\tup{ F } = k_i $.
If $ \na\tup{ F } > k_i $, then $\na\tup{F}$ must be at least $ k_{i+1} $
(since 
$\na\tup{F} \in \na\tup{ \CF } = \set{k_0 < k_1 < k_2 < \cdots  < k_{m}} $),
so that $ F \in \CF_{k_{i+1}} $.
In this case, we thus conclude that $\FF = F \in \CF_{k_{i+1}}
\subseteq \kf_{k_{i+1}}$ and therefore
$\FF \tup{\BBta - k_i} \in \kf_{k_{i+1}}$ as well (since
$ \kf_{k_{i+1}} $ is an ideal of $ \kf $).
So Lemma~\ref{main.lem.1a} is proved in this case.

Hence, we WLOG assume that $\na\tup{F} = k_i$.
Thus, Lemma~\ref{main.lem.1} yields
$F\BBta \equiv k_i F \mod  \kf_{k_{i+1}}$.
In other words, $F \BBta - k_i F \in \kf_{k_{i+1}}$.
But $\FF = F$, so that
\[
\FF \tup{\BBta - k_i}
= F \tup{\BBta - k_i}
= F \BBta - k_i F \in \kf_{k_{i+1}}.
\]
This proves Lemma~\ref{main.lem.1a}.
\end{proof}

\begin{proof}[Proof of Theorem \ref{main.thm.3a}]
Define the polynomial
\[
f\tup{x}
:= \prod_{k \in \na\tup{\CF}} \tup{x - k}
\in \kk\sbr{x}.
\]
Thus, our goal is to prove that $f\tup{\BBa} = 0$.
This is
equivalent to proving that $f\tup{\BBta} = 0$,
because of Lemma~\ref{lem.standin2}.
So this is what we need to do.

In other words, we need to prove the equality
\[
\prod_{k \in \na\tup{\CF}} \tup{\BBta - k} = 0
\]
(by the definition of $f$).
Using the notations of Definition~\ref{def.ksf}, we can rewrite this as
\begin{equation}
\prod_{j=0}^{m}
\tup{ \BBta-k_{j} }
= 0
\label{main.thm3a.eq.1}
\end{equation}
(since $k_0, k_1, \ldots, k_m$ are the elements of $\na\tup{\CF}$ in increasing order).

We will obtain this equality as a particular case ($i=m$) of the equality
\begin{equation}
\prod_{j=0}^{i} \tup{\BBta - k_j}
\in \kf_{k_{i+1}},
\label{main.thm3a.eq.2}
\end{equation}
which we claim to hold for each $i\in \interval{-1,m}$.
This latter equality \eqref{main.thm3a.eq.2} will be proved by induction
on $i$:

The \textit{base case} ($i=-1$) is saying that the
empty product belongs to $\kf_{k_0}$. But this is clear,
since $\kf_{k_0} = \kf$.

The \textit{induction step} (from $i-1$ to $i$) follows from Lemma
\ref{main.lem.1a}, because
\begin{align*}
\prod_{j=0}^{i} \tup{\BBta - k_j}
& = \underbrace{\tup{ \prod_{j=0}^{i-1}
\tup{ \BBta - k_j } }  }_{\substack{\in
\kf_{k_i}\\\text{(by the induction hypothesis)}}}
\tup{ \BBta - k_i }  \\
& \in \kf_{k_i} \tup{\BBta - k_i}
\subseteq \kf_{k_{i+1}}
\ \ \ \ \ \ \ \ \ \ \text{by Lemma \ref{main.lem.1a}.}
\end{align*}

Thus, \eqref{main.thm3a.eq.2} is proved.
Applying \eqref{main.thm3a.eq.2} to $i=m$,
we find
$\prod_{j=0}^m \tup{\BBta - k_j}
\in \kf_{k_{m+1}} = 0$,
so that
\eqref{main.thm3a.eq.1} follows.
As explained above, this completes the proof of
Theorem \ref{main.thm.3a}.
\end{proof}

\subsection{A polynomial annihilating $\wo \BBa$}

We now present a polynomial that annihilates $\wo \BBa$:

\begin{thm}
\label{thm.main.1}
    The element $\wo \BBa \in \dsn$ satisfies
    \begin{equation}
    \wo \BBa\prod_{\substack{k \in \na\tup{\CF}; \\ k \ne 0}}\tup{\wo\BBa + k}\tup{\wo\BBa - k} = 0.
    \label{eq.main.1}
    \end{equation}
\end{thm}

The proof will rely on the following two lemmas:

\begin{lem}
    \label{main.cor.1}
    Let $ F \in \CF $. Let $i \in \interval{0, m}$ be such that $ \na\tup{ F } = k_i $. Then
    \[
    F \BBt_{\rev \alpha} \BBta \equiv k_i^2 F \mod \kf_{k_{i+1}}.
    \]
\end{lem}

\begin{proof}
    By Proposition \ref{na.prop.2}, we have
	$n_{\rev \alpha}\tup{ F } = \na\tup{ F } = k_i$.
	But Lemma \ref{main.lem.1} yields
	$F \BBta \equiv k_i F \mod \kf_{k_{i+1}}$.
	Likewise, Lemma \ref{main.lem.1} (applied to $\rev \alpha$ instead of $\alpha$) yields
	$F \BBt_{\rev \alpha} \equiv k_i F \mod \kf_{k_{i+1}}$
	(since $n_{\rev \alpha}\tup{ F } = k_i$, and since Proposition~\ref{na.prop.2} shows that $n_{\rev \alpha}\tup{\CF} = \na\tup{\CF} = \set{k_0 < k_1 < k_2 < \cdots < k_m}$).
    Since $\kf_{k_{i+1}}$ is a right ideal, we can multiply this congruence by $\BBta$, thus obtaining
    \begin{align*}
    F \BBt_{\rev \alpha} \BBta
	&\equiv k_i F \BBta
    \equiv k_i k_i F  \qquad \tup{\text{since } F \BBta \equiv k_i F \mod \kf_{k_{i+1}}}
    \\
    &= k_i^2 F \mod \kf_{k_{i+1}},
    \end{align*}
    as desired.
\end{proof}

\begin{lem}
    \label{main.lem.2}
    Let $i \in \interval{0, m}$. If $ \FF \in \kf_{k_i} $, then 
    \[
        \FF\tup{ \wotil \BBta + k_i }\tup{ \wotil \BBta - k_i } \in \kf_{k_{i+1}}.
    \]
\end{lem}

\begin{proof}
    It suffices to show that the lemma holds when $\FF$ is a single face $ F \in \CF_{k_i}$, so let us assume this is the case; thus, $ \na\tup{ F } \ge k_i $.
	The case when $ \na\tup{ F } > k_i $ can be handled just as in the proof of Lemma~\ref{main.lem.1a} above\footnote{%
	If $ \na\tup{ F } > k_i $, then $\na\tup{F}$ must be at least $ k_{i+1} $, so that $ F \in \CF_{k_{i+1}} $, thus $\FF = F \in \CF_{k_{i+1}} \subseteq \kf_{k_{i+1}}$. Then, since $ \kf_{k_{i+1}} $ is an ideal of $ \kf $, we must have $ \FF \tup{ \wotil\BBta + k_i }\tup{ \wotil\BBta - k_i } \in \kf_{k_{i+1}}$ as well.}.
    Hence, we WLOG assume that $ \na\tup{ F } = k_i $. Then,
	Lemma~\ref{main.cor.1} yields
	$F \BBt_{\rev \alpha} \BBta \equiv k_i^2 F \mod \kf_{k_{i+1}}$.
	In other words,
	\begin{align}
	F \BBt_{\rev \alpha} \BBta - k_i^2 F \in \kf_{k_{i+1}}.
	\label{pf.main.lem.2.4}
	\end{align}
	Now,
	\[
	\tup{ \wotil\BBta + k_i }\tup{ \wotil\BBta - k_i }
	= \tup{ \wotil\BBta }^2 - k_i^2
	= \BBt_{\rev \alpha} \BBta - k_i^2
	\qquad \tup{\text{by \eqref{eq.w0conj_cor.2}}},
	\]
	so that
    \begin{align*}
        \FF \tup{ \wotil\BBta + k_i }\tup{ \wotil\BBta - k_i }
		&= \FF \tup{\BBt_{\rev \alpha} \BBta - k_i^2}
		= F \tup{\BBt_{\rev \alpha} \BBta - k_i^2}
		\qquad \left(\text{since } \FF = F\right) \\
		&= F \BBt_{\rev \alpha} \BBta - k_i^2 F
		\in \kf_{k_{i+1}}
		\qquad \left(\text{by \eqref{pf.main.lem.2.4}}\right).
		\qedhere
    \end{align*}
\end{proof}

\begin{proof}[Proof of Theorem \ref{thm.main.1}]
Define the polynomial
\[
f\tup{x}
:= x \prod_{\substack{k \in \na\tup{\CF}; \\ k \ne 0}}\tup{x + k}\tup{x - k}
\in \kk\sbr{x}.
\]
Thus, our goal is to prove that $f\tup{\wo\BBa} = 0$.
According to Lemma~\ref{lem.standin}, this is
equivalent to proving that $f\tup{\wotil \BBta} = 0$.
So this is what we need to do.

In other words, we need to prove the equality
\[
\wotil\BBta\prod_{\substack{k \in \na\tup{\CF}; \\ k \ne 0}}\tup{\wotil\BBta + k}\tup{\wotil\BBta - k} = 0
\]
(by the definition of $f$).
Using the notations of Definition~\ref{def.ksf}, we can rewrite this as
\begin{equation}
\wotil\BBta\prod_{j=p}^{m}
\tup{ \wotil\BBta+k_{j} } \tup{ \wotil\BBta-k_{j} }
= 0,
\label{main.thm.eq.1}
\end{equation}
where $p$ is the smallest element of $\interval{0, m}$ satisfying $k_p > 0$ (clearly, $p$ is either $0$ or $1$, depending on whether $k_0$ is positive or $0$).

We will obtain this equality as a particular case ($i=m$) of the equality
\begin{equation}
\wotil\BBta
\prod_{j=p}^{i} \tup{ \wotil\BBta + k_{j} }
\tup{ \wotil\BBta - k_{j} }  \in \kf_{k_{i+1}},
\label{main.thm.eq.2}
\end{equation}
which we claim to hold for each $i\in \interval{p-1,m}$.
This latter equality \eqref{main.thm.eq.2} will be proved by induction
on $i$:

The \textit{base case} ($i=p-1$) is saying that $\wotil\BBta\in
\kf_{k_p}$. To prove this, it suffices to show that
$\BBta\in\kf_{k_p}$
(since $\kf_{k_p}$ is an ideal).
But this is easy: Any face $F\in\CF$ that satisfies
$\type F = \alpha$
must satisfy $\Na\tup{F} \neq \varnothing$ (indeed, $F$ itself is a
$G\in\CF$ satisfying $F\con G$ and $\type G = \alpha$;
thus, the set $\Na\tup{F}$ contains at least $F$
as an element) and thus $\na\tup{F} > 0$ and therefore
$\na\tup{F} \geq k_p$
(since the definition of $p$ shows that $k_p$ is the smallest
positive element of $\na\tup{\CF}$),
so that $F \in\CF_{k_p} \subseteq \kf_{k_p}$.
Thus, $\BBta$ (being a sum of such $F$'s) belongs to $\kf_{k_p}$
as well.
This completes the base case.

The \textit{induction step} (from $i-1$ to $i$) follows from Lemma
\ref{main.lem.2}, because
\begin{align*}
& \wotil\BBta\prod_{j=p}^{i} \tup{
\tup{ \wotil\BBta + k_j } \tup{ \wotil\BBta - k_j }
} \\
& =\underbrace{\tup{ \wotil\BBta\prod_{j=p}^{i-1}
\tup{ \wotil\BBta + k_j } \tup{ \wotil\BBta - k_j }
} }_{\substack{\in
\kf_{k_i}\\\text{(by the induction hypothesis)}}}
\tup{ \wotil\BBta + k_i } \tup{ \wotil\BBta - k_i } \\
& \in \kf_{k_i}
\tup{ \wotil\BBta + k_i } \tup{ \wotil\BBta - k_i }
\subseteq \kf_{k_{i+1}}
\ \ \ \ \ \ \ \ \ \ \text{by Lemma \ref{main.lem.2}.}
\end{align*}

Thus, \eqref{main.thm.eq.2} is proved.
Applying \eqref{main.thm.eq.2} to $i=m$
and recalling that $\kf_{k_{m+1}}=0$, we obtain
\eqref{main.thm.eq.1}, and Theorem \ref{thm.main.1} follows.  
\end{proof}

\section{\label{sec.altsum}Alternating sum identities for set compositions}

To improve on the polynomial in Theorem \ref{thm.main.1},
we need two combinatorial identities regarding set compositions:

\begin{lem}
\label{lem.altsum.faces-all}
We have
\[
    \sum_{G \in \CF} \tup{ -1 }^{\ell\tup{ G }}  = \tup{ -1 }^n.
\]
\end{lem}

\begin{proof}
This is essentially \cite[Theorem 2.2.2]{Sag20} (see the first paragraph in the proof of \cite[Theorem 2.2.2]{Sag20} for an explanation).
However, let us give a new proof here, using the tools of this paper.

Set $\kk = \ZZ$.
Consider the composition $\tup{1, 1, \ldots, 1}$
of $n$ (with $n$ many $1$'s).
The definition of $\BBt_{\tup{1, 1, \ldots, 1}}$
yields
\begin{align}
\BBt_{\tup{1, 1, \ldots, 1}}
&= \sum_{\substack{F\in \CF;\\
\type F=\tup{1, 1, \ldots, 1}}} F
= \sum_{\substack{A\in \CF;\\
\type A=\tup{1, 1, \ldots, 1}}} A.
\label{pf.lem.altsum.faces-all.BBt111}
\end{align}
One of the addends in this sum is the face
$\tup{\set{1}, \set{2}, \ldots, \set{n}} \in \CF$.
Hence, the coefficient of this face in
$\BBt_{\tup{1, 1, \ldots, 1}}$ is $1$.

Set
\begin{equation}
\lambda := \sum_{F\in\CF} \tup{-1}^{n-\ell \tup{F}}
\in \ZZ.
\label{pf.lem.altsum.faces-all.0}
\end{equation}

Let $A\in\CF$ be a face of type $\tup{1, 1, \ldots, 1}\models
n$. (Note that there are $n!$ such faces, each of them having the form
$\tup{ \set{\sigma\tup{1}}, \set{\sigma\tup{2}},
\ldots, \set{\sigma\tup{n}} }$
for some permutation $\sigma\in S_n$.)
Then, the blocks of $A$ have size $1$, and therefore
cannot be subdivided any further by intersecting
them with the blocks of another face $F\in\CF$. Hence,
\begin{equation}
AF = A \qquad \text{for each } F\in\CF .
\label{pf.lem.altsum.faces-all.1}
\end{equation}
But Proposition \ref{w0face.prop} yields
\begin{align*}
A\wotil
&= A\sum_{F\in\CF} \tup{-1}^{n-\ell\tup{F} }F
= \sum_{F\in\CF} \tup{-1}^{n-\ell\tup{F} }
\underbrace{AF}_{\substack{=A\\
\text{(by \eqref{pf.lem.altsum.faces-all.1})}}}
= \underbrace{\sum_{F\in\CF}
\tup{-1}^{n-\ell\tup{F} }}_{=\lambda}A
= \lambda A.
\end{align*}

Forget that we fixed $A$. We thus have proved the equality
$A\wotil = \lambda A$ for each face
$A\in\CF$ of type $\tup{1, 1, \ldots, 1}\models n$.
Summing this equality over all such faces $A$, we obtain
\begin{equation}
\sum_{\substack{A\in\CF;\\\type A= \tup{1,1,\ldots,1}
}} A\wotil
= \lambda\sum_{\substack{A\in\CF;\\
\type A=\tup{1, 1, \ldots, 1}}
}A.
\end{equation}
Using \eqref{pf.lem.altsum.faces-all.BBt111}, we can rewrite 
this as
\[
\BBt_{\tup{1, 1, \ldots, 1}}
\wotil
= \lambda\BBt_{\tup{1, 1, \ldots, 1}}.
\]

On the other hand,
$\tup{1, 1, \ldots, 1} = \gaps\tup{\interval{n-1}}$, so that
\[
\BB_{\tup{1, 1, \ldots, 1}}
= \BB_{\interval{n-1}} = \sum_{w \in S_n} w
\]
(by \eqref{eq.exam.Bn-1.1}), and therefore
\[
\wo \BB_{\tup{1, 1, \ldots, 1}}
= \wo \sum_{w \in S_n} w
= \sum_{w \in S_n} \wo w
= \sum_{w \in S_n} w
\]
(here, we have substituted $w$ for $\wo w$ in the
sum). Comparing these two equalities, we obtain
$\wo\BB_{\tup{1, 1, \ldots, 1}}
=\BB_{\tup{1, 1, \ldots, 1}}$.
Applying the map $\rho$ to this equality, we find
\[
\rho \tup{ \wo\BB_{\tup{1, 1, \ldots, 1}} }
=\rho \tup{ \BB_{\tup{1, 1, \ldots, 1}} }
=\BBt_{\tup{1, 1, \ldots, 1}}\qquad\left(  \text{by
the definition of }\rho\right)  ,
\]
so that
\begin{align*}
\BBt_{\tup{1, 1, \ldots, 1}}
& =\rho \tup{ \wo \BB_{\tup{1, 1, \ldots, 1}} } \\
& =\underbrace{\rho \tup{ \BB_{\tup{1, 1, \ldots, 1}} }
}_{=\BBt_{\tup{1, 1, \ldots, 1}}}\underbrace{\rho
\tup{ \wo } }_{=\wotil}\qquad\left(  \text{since }
\rho\text{ is a }\kk\text{-algebra anti-morphism}\right)  \\
& =\BBt_{\tup{1, 1, \ldots, 1}}
\wotil=\lambda\BBt_{\tup{1, 1, \ldots, 1}}.
\end{align*}
Comparing coefficients in front of the face
$\tup{\set{1}, \set{2}, \ldots, \set{n}} \in \CF$
in this equality, we obtain $1 = \lambda$
(since the coefficient of this face in
$\BBt_{\tup{1, 1, \ldots, 1}}$ is $1$).
Thus, $1 = \lambda
=\sum_{F\in\CF} \tup{-1}^{n-\ell\tup{F} }$
(by the definition of $\lambda$). Dividing this
equality by $\tup{-1}^{n}$,
we find $\tup{-1}^{n}=\sum_{F\in\CF} \tup{-1}
^{\ell\tup{F} }=\sum_{G\in\CF} \tup{-1}
^{\ell\tup{G} }$, qed.
\end{proof}

\begin{lem}
\label{lem.altsum.faces}
Let $F \in \CF$ be a face. Then,
\[
    \sum_{\substack{G \in \CF; \\ F \con G}} \tup{ -1 }^{n-\ell\tup{ G }}
	= \tup{ -1 }^{n-\ell\tup{ F }} .
\]
\end{lem}

\begin{proof}
Write the set composition $F$ as $F = \tup{F_1, F_2, \ldots, F_m}$,
so that $\ell\tup{F} = m$.
Then,
Proposition~\ref{ccb_prop}
shows that there is a length-preserving bijection
\begin{align*}
f : \set{G\in\CF : F\con G}   &  \to 
\set{H : H\models \interval{m} }  ,\\
\tup{G_1, G_2, \ldots, G_k} &\mapsto \tup{H_1, H_2, \ldots, H_k}
            \text{ where } H_{i} = \set{j \in \interval{m} : F_{j} \subseteq G_{i}}.
\end{align*}
Using this bijection, we can reindex the sum
$\sum\limits_{\substack{G\in \CF; \\F\con G}}
\tup{-1}^{n-\ell\tup{G}}$ as follows:
\[
\sum_{\substack{G\in\CF; \\F\con G}}
\tup{-1}^{n-\ell\tup{G} }
=\sum_{H\models \interval{m}  }
\tup{-1}^{n-\ell\tup{H} }
=\sum_{G\models \interval{m}  }
\tup{-1}^{n-\ell\tup{G} }
= \tup{-1}^n
\sum_{G\models\interval{m}  }
\tup{-1}^{\ell\tup{G} }.
\]
However, applying Lemma~\ref{lem.altsum.faces-all} to $m$
instead of $n$, we find
\[
\sum_{G\models\interval{m}  } \tup{-1}^{\ell\tup{G} }
=\tup{-1}^m.
\]
In light of this, the previous equality simplifies to
\[
\sum_{\substack{G\in\CF; \\F\con G}}
\tup{-1}^{n-\ell\tup{G} }
=\tup{-1}^{n} \tup{-1}^{m}
=\tup{-1}^{n-m}
=\tup{-1}^{n-\ell\tup{F} }
\]
(since $m = \ell\tup{F}$),
as desired.
\end{proof}

\section{\label{sec.shortpol}The optimal polynomial for $\wo \BBa$}

Recall that $\alpha$ is a fixed composition of $n$.

As we mentioned above, the equality \eqref{eq.main.1}
provides a polynomial that annihilates $\wo \BBa$,
but this polynomial is not optimal: Some of the
factors can be removed from the product without
turning it nonzero.
In this section, we shall find an optimal polynomial
that annihilates $\wo \BBa$
(meaning that -- at least when $\kk$ is a field of
characteristic $0$ -- it is the minimal polynomial
of $\wo \BBa$, as we will see in Theorem~\ref{thm.mp.1}).

First, we need some more notations.

\begin{defn}\phantomsection\ %
    \label{def.specv}
	
	\begin{enumerate}
	
	\item[\textbf{(a)}] The \defin{signed knapsack number $\nna \tup{F}$}
	of a face $F \in \CF$ is the integer defined by
	\begin{align}
	\nna \tup{F}
	= \tup{-1}^{n-\ell\tup{F}} \na \tup{F} .
	\label{eq.def.specv.nnaF=}
	\end{align}

	\item[\textbf{(b)}]
	We will furthermore write $\nna\tup{\CF}$ for the set
	$\set{\nna\tup{F} : F \in \CF}$, which consists of the
	signed knapsack numbers of all faces.
	This is the image of $ \CF $ under the map $\nna$.
	It is obviously a finite set of integers.
	
	\end{enumerate}
\end{defn}

Note that each face $F\in\CF$ satisfies
\begin{align}
\abs{\nna \tup{F}} = \na\tup{F}
\label{eq.nna.abs}
\end{align}
(by \eqref{eq.def.specv.nnaF=},
since $\na\tup{F}$ is nonnegative).
Hence, each signed knapsack number $k \in \nna\tup{\CF}$
satisfies
\[
\abs{k}
\in n_{\alpha} \tup{\CF}
= \set{ k_0 < k_1 < \cdots < k_{m} }
\]
(see Definition~\ref{def.ksf} \textbf{(b)} for the latter equality).

Our main result takes the following form
(recall that $\alpha$ is a fixed composition of $n$):

\begin{thm}
\label{main.thm.3}
The elements $\wo \BBa$ and $\BBa \wo$ of $\dsn$ satisfy
\begin{equation}
\prod_{k \in \nna\tup{\CF}} \tup{\wo\BBa - k}
=0
\label{eq.main.2a}
\end{equation}
and
\begin{equation}
\prod_{k \in \nna\tup{\CF}} \tup{\BBa\wo - k}
=0.
\label{eq.main.2b}
\end{equation}
(Each of these products ranges over all \textbf{distinct}
elements $k \in \nna\tup{\CF}$. Thus, even if several different
faces $F$ have the same signed knapsack number $k$, there
is only one $\wo\BBa - k$ (resp. $\BBa\wo - k$) factor in
each product.)
\end{thm}

The proof of this will rely on a counterpart to Lemma~\ref{main.lem.2}:

\begin{lem}
\label{main.lem.3}
Let $k \in \nna\tup{\CF}$.
Let $i\in \interval{0, m}$ be such that $\abs{k} = k_i$.
Let $\FF\in\kf_{k_i}$. Then:

\begin{enumerate}
\item[\textbf{(a)}] If $-k \notin \nna\tup{\CF}$,
then
\[
\FF \tup{ \wotil - \sign k }
\in\kf_{k_{i+1}}.
\]
Here, $\sign k$ means the sign of $k$ (that is, $1$ if $k>0$ and
$-1$ if $k<0$).

\item[\textbf{(b)}] We have
\[
\FF \tup{ \BBta - \abs{k} }
\in\kf_{k_{i+1}}.
\]

\item[\textbf{(c)}] If
$\nna \tup{\CF} \cap \set{k,-k} = \set{k}$
(that is, if $-k = k$ or $-k \notin \nna\tup{\CF}$),
then
\[
\FF \tup{\wotil\BBta - k}
\in\kf_{k_{i+1}}.
\]

\end{enumerate}
\end{lem}

\begin{proof}
\textbf{(a)} Assume that $-k \notin \nna\tup{\CF}$.
Then, $-k \neq k$ (since $k \in \nna\tup{\CF}$), so that
$k \neq 0$. Thus, $\sign k$ is well-defined.

We must prove that
$\FF \tup{ \wotil - \sign k }
\in\kf_{k_{i+1}}$.
By linearity, we WLOG assume that $\FF$ is a single face
$F = \tup{F_1, F_2, \ldots, F_p} \in \CF_{k_i}$.
Thus, $\na\tup{F} \geq k_i$, so that
$\na\tup{F} > k_i$ or $\na\tup{F} = k_i$.
If
$\na\tup{F} > k_i$, then $\na\tup{F} \geq k_{i+1}$
(since $ \na\tup{ F } $ belongs to the set $ \na\tup{ \CF } = \set{k_0 < k_1 < k_2 < \cdots  < k_m} $, and thus cannot be $> k_i$ without being $\geq k_{i+1}$),
so that
$\FF = F \in \CF_{k_{i+1}} \subseteq \kf_{k_{i+1}}$,
and so the claim follows because $\kf_{k_{i+1}}$
is an ideal of $\kf$. Hence, we WLOG assume that
$\na\tup{F} = k_i$.

Therefore, \eqref{eq.nna.abs} yields
$\abs{\nna\tup{F}} = \na\tup{F} = k_i = \abs{k}$.
In other words, $\nna\tup{F} = \pm k$. But
$\nna \tup{F}$ cannot be $-k$, since $-k \notin
\nna\tup{\CF}$. Hence, we conclude that
$\nna\tup{F} = k$.
Hence,
$k = \nna\tup{F} = \tup{-1}^{n-\ell \tup{F} }\na\tup{F}$,
so that
\begin{equation}
\sign k
= \tup{-1}^{n-\ell\tup{F} }
\label{pf.main.lem.3.a.sign}
\end{equation}
(since $\na\tup{F} \geq 0$ and $k \neq 0$).

Next, we claim that if $G\in\CF$ is any face such that
$F\not \con G$, then
\begin{equation}
FG\in\CF_{k_{i+1}}.
\label{pf.main.lem.3.a.ineq}
\end{equation}

[\textit{Proof of \eqref{pf.main.lem.3.a.ineq}:} Let $G\in \CF$ be any
face such that $F\not \con G$. Then, there exists some block
$F_{u}$ of $F$ that is not a subset of any block of $G$. Consider this block
$F_{u}$. Thus, $F_{u}$ intersects at least two distinct blocks of $G$. Let
$G_{v}$ be one of the blocks of $G$ that $F_{u}$ intersects. Then,
the sets $F_u \cap G_v$ and $F_u \setminus G_v$ are both nonempty
and disjoint, and their union is $F_u$.
Hence, replacing
the block $F_{u}$ of $F$ by the two nonempty blocks $F_{u}\cap G_{v}$ and
$F_{u}\setminus G_{v}$, we obtain a new set composition
\[
H := \tup{ F_{1},F_{2},\ldots,F_{u-1},\ F_{u}\cap G_{v},\ F_{u}\setminus
G_{v},\ F_{u+1},F_{u+2},\ldots,F_{p} } ,
\]
which has length $p+1 = \ell\tup{F} +1$
(since $p = \ell\tup{F}$) and satisfies
$FG\con H\con F$. Thus, Proposition \ref{na.prop.00} yields
$\na\tup{FG} \geq \na\tup{H} \geq \na\tup{F}$.

Recall that $\nna\tup{F} = k$.
Thus, $-\nna\tup{F} = -k \notin \nna\tup{\CF}$.
Hence, in particular, $-\nna\tup{F} \neq \nna\tup{H}$.
However, if we had $\na\tup{H} = \na\tup{F}$,
then we would have $\nna\tup{H} = -\nna\tup{F}$
(since $\ell\tup{H} = \ell\tup{F} + 1$ shows that
the signs $\tup{-1}^{n-\ell\tup{H}}$ and
$\tup{-1}^{n-\ell\tup{F}}$ in the definitions of
$\nna\tup{H}$ and $\nna\tup{F}$ are opposite),
which would contradict $-\nna\tup{F} \neq \nna\tup{H}$.
Thus, we cannot have $\na\tup{H} = \na\tup{F}$.
Therefore, $\na\tup{H} > \na\tup{F}$
(since $\na\tup{H} \geq \na\tup{F}$).
Altogether, 
$\na\tup{FG} \geq \na\tup{H} > \na\tup{F} = k_i$.
Hence, $\na\tup{FG} \geq k_{i+1}$
(again since $ \na\tup{ FG } $ belongs to the set
$ \na\tup{ \CF } = \set{k_0 < k_1 < k_2 < \cdots  < k_m} $)
and therefore $FG\in\CF_{k_{i+1}}$. This proves
\eqref{pf.main.lem.3.a.ineq}.] \medskip

Now, multiplying the equalities $\FF=F$ and \eqref{eq.w0face}, we
obtain
\begin{align*}
\FF\wotil &  =F\sum_{G\in\CF}\tup{ -1 }^{n-\ell
\tup{G} }G=\sum_{G\in\CF}\tup{ -1 }^{n-\ell\tup{G}
}FG\\
&  =\sum_{\substack{G\in\CF;\\F\con G}}
\tup{ -1 }^{n-\ell\tup{G} }FG+\sum_{\substack{G\in\CF;
\\F\not \con G}}\tup{ -1 }^{n-\ell\tup{G}
}\underbrace{FG}_{\substack{\in\CF_{k_{i+1}}\\\text{(by
\eqref{pf.main.lem.3.a.ineq})}}}\\
&  \equiv\sum_{\substack{G\in\CF;\\F\con G}}
\tup{ -1 }^{n-\ell\tup{G} }
\underbrace{FG}_{\substack{=F\\\text{(by
Proposition \ref{cp_prop} \textbf{(c)})}}}\\
&  =\underbrace{\sum_{\substack{G\in\CF;\\F\con
G}}\tup{ -1 }^{n-\ell\tup{G} }}_{\substack{=
\tup{-1}^{n-\ell\tup{F} }\\
\text{(by Lemma \ref{lem.altsum.faces})}}}
F
= \underbrace{\tup{-1}^{n-\ell \tup{F}}}_{\substack{
= \sign k \\ \text{(by \eqref{pf.main.lem.3.a.sign})}}}
 \underbrace{F}_{= \FF}
= \sign k \cdot \FF
\mod\kf_{k_{i+1}}.
\end{align*}
In other words,
\[
\FF \wotil - \sign k \cdot \FF \in \kf_{k_{i+1}}.
\]
Equivalently,
\[
\FF \tup{\wotil - \sign k}
\in \kf_{k_{i+1}}.
\]
Thus, Lemma \ref{main.lem.3} \textbf{(a)} is proved.
\medskip

\textbf{(b)}
Lemma \ref{main.lem.1a} yields
\[
\FF \tup{\BBta - k_i}
\in \kf_{k_{i+1}} .
\]
Since $\abs{k} = k_i$, we can rewrite this as
$\FF\left(  \BBta-\abs{k} \right)
\in \kf_{k_{i+1}}$.
This proves Lemma \ref{main.lem.3} \textbf{(b)}. \medskip

\textbf{(c)} Recall that $\kf_{k_{i+1}}$ is an ideal of
$\kf$, so we can compute modulo it as with integers.

Now, assume that $\nna \tup{\CF} \cap \set{k,-k} = \set{k}$
(that is, that $-k = k$ or $-k \notin \nna\tup{\CF}$).
Thus, we are in one of the following two cases:

\textit{Case 1:} We have $-k = k$.

\textit{Case 2:} We have $-k \notin \nna\tup{\CF}$.

We begin by handling Case 1. In this case, $-k = k$.
Hence, $k = 0$, 
so that our goal is to prove that
$\FF\wotil \BBta \in \kf_{k_{i+1}}$.
But $\FF \in \kf_{k_i}$ and thus
$\FF\wotil \in \kf_{k_i}$
(since $\kf_{k_i}$ is an ideal).
Hence, part \textbf{(b)} (applied to $\FF\wotil$
instead of $\FF$) yields
$\FF\wotil
\left(  \BBta-\abs{k} \right)
\in \kf_{k_{i+1}}$.
In view of $k = 0$, this simplifies to
$\FF\wotil \BBta \in \kf_{k_{i+1}}$,
which is what we desired to prove.
Thus, Lemma \ref{main.lem.3} \textbf{(c)} is proved in Case 1.

Now let us consider Case 2. In this case,
$-k \notin \nna\tup{\CF}$. Hence, part
\textbf{(a)} yields
$\FF \tup{\wotil - \sign k} \in \kf_{k_{i+1}}$.
In other words,
\begin{align}
\FF\wotil \equiv \sign k\cdot\FF
\mod \kf_{k_{i+1}}.
\label{pf.main.lem.3.c.1}
\end{align}
But part \textbf{(b)} yields
$\FF \tup{\BBta - \abs{k}} \in \kf_{k_{i+1}}$.
That is,
\[
\FF\BBta \equiv \abs{k} \FF
\mod \kf_{k_{i+1}}.
\]
Multiplying the congruence \eqref{pf.main.lem.3.c.1}
with $\BBta$ from the right, we find
\begin{align*}
\FF \wotil \BBta
& \equiv\sign k\cdot\underbrace{\FF
\BBta}_{\equiv\abs{k} \FF
\mod\kf_{k_{i+1}}}
\equiv\underbrace{\sign k\cdot\abs{k} }_{=k}\FF
=k\FF\mod\kf_{k_{i+1}}.
\end{align*}
In other words,
\[
\FF\tup{\wotil\BBta - k}\in\kf_{k_{i+1}}.
\]
Thus, Lemma \ref{main.lem.3} \textbf{(c)} is proved in Case 2.
Hence, the proof of Lemma \ref{main.lem.3} \textbf{(c)} is complete in both cases.
\end{proof}

We combine Lemmas \ref{main.lem.2} and \ref{main.lem.3} into one:

\begin{lem}
\label{lem.main.last}
Let $i \in \interval{0, m}$ and $\FF \in \kf_{k_i}$. Then,
\[
\FF\prod_{\substack{h \in \nna\tup{\CF};\\
\abs{h} = k_i}}
\left(  \wotil \BBta -h\right)
\in \kf_{k_{i+1}}.
\]
\end{lem}

\begin{proof}
We have $\na\tup{\CF} = \set{ k_{0}<k_1<\cdots<k_{m} }$
and thus $k_i \in \na\tup{\CF}$.
Hence, there exists a face $F\in \CF$ such that
$\na\tup{F} = k_i$. Consider this face $F$.
Its signed knapsack number $\nna \tup{F}$
must then be either $k_i$ or $-k_i$
(since it is defined to be  $\underbrace{\tup{-1}^{n-\ell\left(
F\right)  }}_{=\pm 1}\underbrace{\na\tup{F}  }_{=k_i}=\pm
k_i$). Therefore, at least one of $k_i$ and $-k_i$
equals $\nna\tup{F}$ and thus belongs to $\nna\tup{\CF}$.
In other words, there exists an element
$k\in \set{k_i, -k_i}$ that belongs to $\nna\tup{\CF}$.
We are thus in one of the following two cases:

\textit{Case 1:} There exists exactly one element $k\in
\set{k_i, -k_i}$ that belongs to $\nna\tup{\CF}$.
(That is, either $k_i = -k_i$, or exactly one of
the two numbers $k_i$ and $-k_i$ belongs to $\nna\tup{\CF}$.)

\textit{Case 2:} There are two elements $k\in
\set{k_i, -k_i}$ that belong to $\nna\tup{\CF}$.
(That is, $k_i$ and $-k_i$ are distinct, and both
belong to $\nna\tup{\CF}$.)


Let us consider Case 1. In this case,
there exists exactly one element
$k \in \set{k_i, -k_i}$ that belongs to $\nna\tup{\CF}$.
Consider this $k$. Then, $k = \pm k_i$, so that
$\abs{k} = \abs{k_i} = k_i$ (since
$k_i = \na\tup{F} \geq 0$). This also yields
$\set{k, -k} = \set{k_i, -k_i}$.
But $k$ is the only element
of $\set{k_i, -k_i}$ that belongs to $\nna\tup{\CF}$.
In other words, $k$ is the only element
of $\set{k, -k}$ that belongs to $\nna\tup{\CF}$
(since $\set{k, -k} = \set{k_i, -k_i}$).
That is, $\nna\tup{\CF} \cap \set{k,-k} = \set{k}$.
Hence,
Lemma \ref{main.lem.3} \textbf{(c)} shows that
\[
\FF\tup{\wotil\BBta - k} \in \kf_{k_{i+1}}
\qquad \left(\text{since $\abs{k} = k_i$}\right).
\]
However, the only $h \in \nna\tup{\CF}$ satisfying
$\abs{h} = k_i$ is $k$ (since the condition
$\abs{h} = k_i$ is equivalent to
$h \in \set{k_i, -k_i}$, but we know that
$k$ is the only element
of $\set{k_i, -k_i}$ that belongs to $\nna\tup{\CF}$).
Thus,
\[
\prod_{\substack{h \in \nna\tup{\CF};\\
\abs{h} = k_i}} \tup{ \wotil\BBta - h }
= \wotil\BBta - k.
\]
Hence,
\[
\FF \prod_{\substack{h \in \nna\tup{\CF};\\
\abs{h} = k_i}} \tup{ \wotil\BBta - h }
= \FF \tup{\wotil\BBta - k} \in \kf_{k_{i+1}}.
\]
This proves Lemma \ref{lem.main.last} in Case 1.

Let us next consider Case 2.
In this case, there are two elements $k\in
\set{k_i, -k_i}$ that belong to $\nna\tup{\CF}$.
Thus, both numbers $k_i$ and $-k_i$ belong to
$\nna\tup{\CF}$ (and are distinct).
Hence, the elements $h \in \nna\tup{\CF}$
satisfying $\abs{h} = k_i$ are $k_i$ and $-k_i$
(and these two values are distinct). Therefore,
\[
\prod_{\substack{h \in \nna\tup{\CF};
\\ \abs{h} = k_i}} \tup{ \wotil\BBta - h }
= \tup{ \wotil \BBta - \tup{-k_i} }\tup{ \wotil \BBta - k_i}
= \tup{ \wotil \BBta + k_i}\tup{ \wotil \BBta - k_i}.
\]
Multiplying this by $\FF$ from the left, we find
\[
\FF \prod_{\substack{h \in \nna\tup{\CF};
\\ \abs{h} = k_i}} \tup{ \wotil\BBta - h }
= \FF \tup{ \wotil \BBta + k_i}\tup{ \wotil \BBta - k_i}
\in \kf_{k_{i+1}}
\]
by Lemma \ref{main.lem.2}.
Thus, Lemma \ref{lem.main.last} is proved in Case 2.
\end{proof}

We are now ready to prove Theorem~\ref{main.thm.3}:

\begin{proof}[Proof of Theorem~\ref{main.thm.3}.]
Define the polynomial
\[
f\tup{x}
:= \prod_{k \in \nna\tup{\CF}} \tup{x - k}
\in \kk\sbr{x}.
\]
Thus, our goal is to prove that $f\tup{\wo\BBa} = 0$
and $f\tup{\BBa\wo} = 0$.
According to Lemma~\ref{lem.standin}, both of these
equalities will follow if we can show that
$f\tup{\wotil \BBta} = 0$.
So this is what we shall now do.

First, we shall prove that each
$i \in \interval{0, m+1}$ satisfies
\begin{equation}
\prod_{\substack{k \in \nna\tup{\CF};\\
\abs{k} < k_i}} \tup{ \wotil\BBta - k }
\in \kf_{k_i}.
\label{pf.main-ind}
\end{equation}
Once this is proved, we will easily obtain
$f\tup{\wotil \BBta} = 0$ by applying
\eqref{pf.main-ind} to $i=m+1$.

We shall prove \eqref{pf.main-ind} by induction on $i$:

\textit{Base case:} The validity of \eqref{pf.main-ind} for $i=0$ is obvious,
since the left hand side is an empty product (thus equals $1$), while the
right hand side $\kf_{k_{0}}$ is the whole algebra
$\kf$.

\textit{Induction step:} Fix $i\in \interval{0, m}$. Assume that
\eqref{pf.main-ind} holds for this $i$. We must prove that \eqref{pf.main-ind}
holds for $i+1$ instead of $i$ as well.

We note that any $k \in \nna\tup{\CF}$ that satisfies $\abs{k}
<k_{i+1}$ must satisfy $\abs{k} \leq k_i$
(since each $k \in \nna\tup{\CF}$ satisfies
$\abs{k} \in \na\tup{\CF} = \set{k_0 < k_1 < \cdots < k_m}$),
and thus
must satisfy either $\abs{k} < k_i$ or $\abs{k} = k_i$.
Hence,
\begin{align*}
\prod_{\substack{k \in \nna\tup{\CF};\\\abs{k}
<k_{i+1}}}\left(  \wotil\BBta-k\right)
& =\underbrace{\prod_{\substack{k \in \nna\tup{\CF};\\
\abs{k} < k_i}}\left(  \wotil\BBta-k\right)  }_{\substack{\in
\kf_{k_i}\\\text{(by our induction hypothesis,
which}\\\text{says that \eqref{pf.main-ind} holds for }i\text{)}}%
}\cdot\underbrace{\prod_{\substack{k \in \nna\tup{\CF};\\
\abs{k} = k_i}}\left(  \wotil\BBta-k\right)  }_{=
\prod\limits_{\substack{h \in \nna\tup{\CF};\\\abs{h} =k_i%
}}\left(  \wotil\BBta-h\right)  }\\
& \in\kf_{k_i}\cdot\prod_{\substack{h \in \nna\tup{\CF};
\\ \abs{h} = k_i}}\left(  \wotil\BBta-h\right)
\subseteq \kf_{k_{i+1}}
\end{align*}
(by Lemma \ref{lem.main.last}). This shows that \eqref{pf.main-ind} holds for
$i+1$ instead of $i$. Thus, the induction is complete, so that
\eqref{pf.main-ind} is proved.

Now, applying \eqref{pf.main-ind} to $i=m+1$, we obtain
\[
\prod_{\substack{k \in \nna\tup{\CF};\\
\abs{k} <k_{m+1}}} \left(  \wotil\BBta-k\right)
\in \kf_{k_{m+1}}.
\]
Since each $k \in \nna\tup{\CF}$ satisfies $\abs{k} <k_{m+1}$
(because $k_{m+1}=\infty$), and since $\CF_{k_{m+1}}=\varnothing$, we
can rewrite this as
\[
\prod_{k \in \nna\tup{\CF}}\left(  \wotil\BBta-k\right)
\in\kk\varnothing = 0.
\]
In other words,
\begin{align}
\prod_{k \in \nna\tup{\CF}}\left(
\wotil\BBta-k\right)  =0.
\label{pf.eq.main.2.at}
\end{align}
Since the left hand side of this equality is
$f\tup{\wotil \BBta}$ (by the definition of $f\tup{x}$),
we can rewrite this as $f\tup{\wotil \BBta} = 0$.
As we explained above, this proves Theorem~\ref{main.thm.3}.
\end{proof}

\section{\label{sec.minpol}Identifying the minimal polynomial}

Recall that if $\kk$ is a field, and if $a$ is an element of a
$\kk$-algebra $A$, then the \defin{minimal polynomial} of $a$
(over $\kk$) means the smallest-degree monic polynomial
$p \in \kk\sbr{x}$ satisfying $p\tup{a} = 0$.
This always exists when $A$ is finite-dimensional over $\kk$. In
particular, if $A$ is a matrix ring $\kk^{r\times r}$, then this is
just the minimal polynomial of a matrix.

In this section, \textbf{we shall assume that $\kk$ is a field of
characteristic $0$}. Note that the minimal polynomial of an element
$a$ of a $\kk$-algebra $A$ does not change if we replace $A$
by a larger algebra that contains $A$ as a subalgebra. Thus, we shall denote
this minimal polynomial by $\mu\tup{a}$ without explicitly
specifying $A$.

If two elements of a $\kk$-algebra are conjugate, then they have the
same minimal polynomial. Thus,
\[
\mu\tup{\BBa\wo} = \mu\tup{\wo\BBa} ,
\]
since the two elements $\BBa\wo$ and $\wo\BBa$ of $\ksn$ are conjugate
(indeed, conjugating $\BBa\wo$ by $\wo$ yields
$\wo\tup{\BBa\wo}\wo^{-1} = \wo\BBa$). This minimal polynomial
$\mu\tup{\BBa\wo} = \mu\tup{\wo\BBa}$ will be the focus of our study
in this section.

Injective $\kk$-algebra morphisms preserve minimal polynomials: That
is, if $f:A\to B$ is an injective $\kk$-algebra morphism, and
if $a\in A$ has a minimal polynomial, then
$f\tup{a}$ has the same minimal polynomial as $a$
(that is, $\mu\tup{f\tup{a}}  =\mu\tup{a}$).
This is because each polynomial $p\in
\kk\sbr{x}  $ satisfies $p\tup{f\tup{a}}
= f\tup{p\tup{a}}$, but the injectivity
of $f$ ensures that $f\tup{p\tup{a}} = 0$ if and
only if $p\tup{a} = 0$. The same reasoning works when
$f$ is an injective $\kk$-algebra anti-morphism instead of a
$\kk$-algebra morphism; thus, we conclude that
\begin{equation}
\mu\tup{ f\tup{a} } = \mu\tup{a}
\label{eq.mu.fa=a.antimor}
\end{equation}
holds in this case as well. In particular,
we can apply this to $f$ being the $\kk$-algebra anti-isomorphism
$\rho:\dsn \to \kfsn$, and conclude that
\[
\mu\tup{\rho\tup{a}}
= \mu\tup{a}
\qquad\text{for each } a \in \dsn.
\]
Applying this to $a = \BBa\wo$,
we obtain $\mu\tup{\rho\tup{\BBa\wo}}
= \mu\tup{\BBa\wo}$.
In view of \eqref{pf.lem.standin.rhoBw},
we can rewrite this as
\begin{equation}
\mu\tup{\wotil \BBta}
= \mu\tup{\BBa\wo}
= \mu\tup{\wo\BBa}
\label{eq.minpol.red-to-F}
\end{equation}
(as we saw above).
We shall use this to compute $\mu\tup{\wo\BBa}  $.

\begin{thm}
\label{thm.mp.1}
Assume that $\kk$ is a field of characteristic $0$.
Let $\alpha \models n$.
Then, the minimal polynomial of the element $\wo\BBa$ of $\ksn$ is
\begin{align}
\mu\tup{\wo\BBa}
=
\prod_{k \in \nna\tup{\CF}}\tup{x-k}  .
\label{eq.thm.mp.1.eq}
\end{align}
\end{thm}

This will follow from a description of the eigenvalues of the ``right multiplication by $\wotil\BBta$'' operator on $\kf$:

\begin{thm}
\label{thm.spec.w0Baright}
Assume that $\kk$ is a field of characteristic $0$.
Let $\alpha \models n$.
Then, the $\kk$-linear operator
\begin{align*}
\kf  & \to \kf,\\
\FF  & \mapsto \FF\wotil\BBta
\end{align*}
is diagonalizable, and its eigenvalues are the signed knapsack numbers
$\nna\tup{F} $ of the faces $F\in \CF$.
(The algebraic multiplicity of each eigenvalue $\lambda$ is the number
of all faces $F$ whose signed knapsack number is $\lambda$.)
\end{thm}

To prove this, we will need the following two lemmas:

\begin{lem}[More blocks lemma]
    \label{mblks_lem}
    Let $F \in \CF$ and $\alpha \models n$. Then we have
    \[
    F\cdot \wotil\BBta = \nna \tup{F}F + \sum_{\substack{G \in \CF;\\ \ell \tup{G} > \ell \tup{F}}} c_{G}G
	\qquad \text{for some scalars $c_G \in \kk$.}
    \]
\end{lem}

\begin{proof}
    Let $I$ be the set of all faces $G \in \CF$ satisfying $\ell\tup{ G } > \ell\tup{ F } $.
	Let $\kk I$ be the $\kk$-linear span of this set $I$ in $\kf$.
	This span $\kk I$ is an ideal of $\kf$ (by an argument analogous to the proof of Proposition~\ref{prop.kfk-ideal}, but now using Proposition~\ref{cp_prop} \textbf{(d)} and \textbf{(e)} instead of Proposition \ref{na.prop.0}).
	Our goal is to show that
	$F\cdot \wotil\BBta \equiv  \nna \tup{F}F \mod \kk I$.

	But Definition \ref{BBta_defn} yields
    \begin{align}
           F \cdot \BBta &= F \cdot \sum_{\substack{G \in \CF; \\ \type G = \alpha}} G
		   = \sum_{\substack{G \in \CF; \\ \type G = \alpha}} FG \nonumber\\
           &= \sum_{\substack{G \in \CF; \\ \type G = \alpha; \\ F \con G}} \ \ \underbrace{FG}_{\substack{= F \\ \text{(by Proposition \ref{cp_prop} \textbf{(c)})}}}
		   + \sum_{\substack{G \in \CF; \\ \type G = \alpha; \\ F \not\con G}} \ \ \underbrace{FG}_{\substack{\equiv 0 \mod \kk I \\ \text{(since Proposition \ref{cp_prop} \textbf{(d)}} \\ \text{yields $\ell\tup{FG} > \ell\tup{F}$,} \\ \text{so that $FG \in I$)}}} \nonumber\\
           &\equiv \sum_{\substack{G \in \CF; \\ \type G = \alpha; \\ F \con G}} F
		   = \na\tup{F} F \mod \kk I
		   \label{pf.mblks_lem.1}
    \end{align}
	(since the sum consists of $\na\tup{F}$ equal addends, by Definition \ref{knap_num_def}).
    Furthermore, \eqref{eq.w0face} yields
    \begin{align*}
           F \cdot \wotil &= F \cdot \sum_{G \in \CF} \tup{ -1 }^{n-\ell\tup{ G }}G
		   = \sum_{G \in \CF} \tup{ -1 }^{n-\ell\tup{ G }} FG 		   \\
           &= \sum_{\substack{G \in \CF; \\ F \con G}} \tup{ -1 }^{n-\ell\tup{ G }} \underbrace{FG}_{\substack{= F \\ \text{(by Proposition \ref{cp_prop} \textbf{(c)})}}}
		   + \sum_{\substack{G \in \CF; \\ F \not\con G}} \tup{ -1 }^{n-\ell\tup{ G }}\underbrace{FG}_{\substack{\equiv 0 \mod \kk I \\ \text{(since Proposition \ref{cp_prop} \textbf{(d)}} \\ \text{yields $\ell\tup{FG} > \ell\tup{F}$,} \\ \text{so that $FG \in I$)}}} \\
           &\equiv \underbrace{\sum_{\substack{G \in \CF; \\ F \con G}} \tup{ -1 }^{n-\ell\tup{ G }}}_{\substack{= \tup{-1}^{n - \ell\tup{F}} \\ \text{(by Lemma \ref{lem.altsum.faces})}}} F
		   = \tup{-1}^{n - \ell\tup{F}} F \mod \kk I.
    \end{align*}
	Since $\kk I$ is an ideal, we can multiply the latter congruence by $\BBta$ from the right, obtaining
    \begin{align*}
           F \cdot \wotil\BBta
           &\equiv \tup{-1}^{n - \ell\tup{F}} F \cdot \BBta
           \equiv \underbrace{\tup{-1}^{n - \ell\tup{F}} \na\tup{F}}_{= \nna\tup{F}} F
		   \qquad \tup{\text{by \eqref{pf.mblks_lem.1}}} \\
		   &= \nna\tup{F} F \mod \kk I.
    \end{align*}
	Recalling how $I$ was defined, we see that this is the claim of the lemma.
\end{proof}

\begin{lem}
\label{lem.muPhi}
Let $\kk$ be a field.
Let $A$ be a $\kk$-algebra. Let $a \in A$ be an element
that has a minimal polynomial.
Let
\begin{align*}
R_a : A &\to A, \\
      x &\mapsto xa
\end{align*}
be the map known as ``right multiplication by $a$''.
This $R_a$ is a $\kk$-linear map, i.e., belongs to
the endomorphism ring $\End A$.
Then,
\[
\mu\tup{R_a} = \mu\tup{a}.
\]
\end{lem}

\begin{proof}
Forget that we fixed $a$.
We thus have defined an endomorphism $R_a \in \End A$
for each $a \in A$.
The map $f : A \to \End A$ that sends each $a \in A$
to this endomorphism $R_a$ is easily seen to be
an injective $\kk$-algebra anti-morphism
(injective because $R_a\tup{1} = 1a = a$;
anti-morphism because $R_b\tup{R_a\tup{x}}
= \tup{xa}b = x\tup{ab} = R_{ab}\tup{x}$).
Thus, \eqref{eq.mu.fa=a.antimor} yields
$\mu\tup{f\tup{a}} = \mu\tup{a}$ for each $a \in A$.
In other words,
$\mu\tup{R_a} = \mu\tup{a}$ for each $a \in A$
(since $f\tup{a} = R_a$).
This proves the lemma.
\end{proof}

\begin{proof}[Proof of Theorem \ref{thm.spec.w0Baright}.]
Let us denote this operator by $\Phi$. Define a polynomial
\[
f\tup{x}
:= \prod_{k \in \nna\tup{\CF}}\tup{x-k}
\in\kk\sbr{x} .
\]
This polynomial factors into distinct monic linear factors
(distinct because $\kk$ has characteristic $0$).
Clearly, \eqref{pf.eq.main.2.at} says that
$f\left( \wotil\BBta \right) =0$.
Thus, the minimal polynomial $\mu\tup{\wotil\BBta}$
must divide $f$ (since the minimal polynomial
$\mu\tup{a}$ of an element $a$ of a $\kk$-algebra always
divides any polynomial that annihilates $a$).

As we recall, the operator $\Phi$ sends each element
$\FF \in \kf$ to $\FF \wotil\BBta$.
In other words, this operator $\Phi$ is
right multiplication by $\wotil\BBta$.
Hence, Lemma~\ref{lem.muPhi} (applied to
$A = \kf$ and $a = \wotil\BBta$ and $R_a = \Phi$)
shows that
$\mu\tup{\Phi} = \mu\tup{\wotil\BBta}$.
Hence, the polynomial $\mu\tup{\Phi}$ divides $f$
(since the polynomial $\mu\tup{\wotil\BBta}$
divides $f$).

Therefore, this polynomial $\mu\tup{\Phi}$
factors into distinct monic linear factors
(because the same is true of $f$).
Hence, the operator
$\Phi$ is diagonalizable (because $\kf$ is a
finite-dimensional $\kk$-vector space, and thus
any endomorphism of $\kf$ whose minimal polynomial
factors into distinct monic linear factors must
be diagonalizable).

It remains to determine its eigenvalues.
The faces in $\CF$ form a basis of the
$\kk$-vector space $\kf$.
Let us order this basis by increasing length of
the face (breaking ties arbitrarily).
Then, the operator $\Phi$ is represented by a
lower-triangular matrix with respect to this
ordered basis,
because Lemma \ref{mblks_lem} gives an expansion of
$\Phi\tup{F} = F\cdot \wotil\BBta$ as a linear
combination of $F$ and longer (thus later) faces.
But the eigenvalues of a lower-triangular matrix are
its diagonal entries (and their algebraic
multiplicities are how often they appear on the diagonal).
In our case, the diagonal entries of the matrix
that represents $\Phi$ are the numbers
$\nna\tup{F}$, again according to Lemma
\ref{mblks_lem}. Hence, the eigenvalues of $\Phi$ are these
numbers $\nna\tup{F}$. This completes the
proof of Theorem \ref{thm.spec.w0Baright}.
\end{proof}

\begin{proof}[Proof of Theorem \ref{thm.mp.1}.]
Let $\Phi:\kf\to\kf$ be the operator defined in Theorem
\ref{thm.spec.w0Baright}.
As shown in that theorem, the eigenvalues of this
operator $\Phi$ are the signed knapsack numbers
$\nna\tup{F}$ of the faces $F \in \CF$.
Thus, the set of all eigenvalues of $\Phi$ is
$\nna\tup{\CF}$.
Since $\Phi$ is furthermore diagonalizable
(again by Theorem \ref{thm.spec.w0Baright}), this shows that the
minimal polynomial $\mu\tup{\Phi}$ of $\Phi$ is the polynomial
\[
f\tup{x}  :=\prod_{k \in \nna\tup{\CF}}\tup{x-k}
\in\kk\sbr{x}
\]
(since the minimal polynomial of a diagonalizable
endomorphism is the product of $x-\lambda$ over the
distinct eigenvalues $\lambda$ of the endomorphism).

However, $\Phi$ is just the right multiplication by
$\wotil\BBta$. Hence, Lemma~\ref{lem.muPhi}
(applied to $A = \kf$ and $a = \wotil\BBta$ and $R_a = \Phi$)
shows that
$\mu\tup{\Phi} = \mu\tup{\wotil\BBta}$.
Thus,
$\mu\tup{\wotil\BBta} = \mu\tup{\Phi} = f\tup{x}$,
as we just saw.

Comparing this with \eqref{eq.minpol.red-to-F}, we obtain
$\mu\tup{\wo\BBa} = f\tup{x}$.
Upon recalling the definition of $f\tup{x}$,
this proves Theorem \ref{thm.mp.1}.
\end{proof}

\section{\label{sec.ttr}Application to top-to-random-and-reverse}

In this section, we shall focus on the case when $\alpha$ is the composition $\tup{1,n-1}$
(for $n > 1$).
In this case, $\BBa = \BB_{\tup{1, n-1}}$ is the image of the \defin{top-to-random shuffle} under the \defin{antipode} map of the symmetric group algebra $\ksn$.
We recall their definitions:

\begin{defn}
The \defin{top-to-random shuffle} is the element
\begin{align}
\A := \sum_{i=1}^n \cyc{1,2,\ldots,i} \in \ksn ,
\label{eq.A=sum-cyc}
\end{align}
where $\cyc{k_1,k_2,\ldots,k_m}$ denotes the $m$-cycle in $S_n$ that sends $k_1 \mapsto k_2 \mapsto \cdots \mapsto k_m \mapsto k_1$.
\end{defn}

\begin{defn}
The \defin{antipode} of the symmetric group algebra $\ksn$ is the $\kk$-linear map $S : \ksn \to \ksn$ that sends each basis vector (permutation) $w \in S_n$ to $w^{-1}$.
It is well-known that $S$ is a $\kk$-algebra anti-isomorphism.
\end{defn}

For the rest of this section, \textbf{we assume that $n > 1$}.

Now, our above claim about $\BB_{\tup{1, n-1}}$ can be stated as follows:

\begin{prop}
\label{prop.BBa-A}
Assume that $n > 1$. Then,
\[
S\tup{\A} = \BB_{\tup{1, n-1}}.
\]
\end{prop}

\begin{proof}
This is well-known, but we give a proof for completeness.
By definition of $\BBa$, we have
\begin{align}
\BB_{\tup{1, n-1}}
&= \BB_{\gaps^{-1}\tup{1, n-1}}
= \BB_{\set{1}}
\qquad \tup{\text{since }
\gaps^{-1}\tup{1, n-1} = \set{1}}
\nonumber\\
& =\sum_{\substack{w\in S_n;\\
\Des w\subseteq \set{1}}} w.
\label{pf.prop.BBa-A.1}
\end{align}
But the permutations $w\in S_n$ that satisfy $\Des %
w\subseteq \set{1}$ are precisely those permutations that satisfy
$w\tup{2} < w\tup{3} < \cdots < w\tup{n}$; in other
words, they consist of an arbitrary element $i\in \interval{n}$ in their
first position, followed by the remaining elements of $\interval{n}$ in
increasing order. Such permutations can be written as cycles
$\cyc{i,i-1,\ldots,2,1}$. Hence, we can rewrite
\eqref{pf.prop.BBa-A.1} as
\[
\BB_{\tup{1, n-1}}
=\sum_{i=1}^{n}\cyc{i,i-1,\ldots,2,1}.
\]
On the other hand, the definition of $\A$ yields
\[
S \tup{\A}
= S \tup{\sum_{i=1}^{n} \cyc{1,2,\ldots,i} }
= \sum_{i=1}^{n}\cyc{1,2,\ldots,i}^{-1},
\]
since the antipode $S$ sends each permutation to its inverse. The right hand
sides of these two equalities are equal, since
$\cyc{i,i-1,\ldots,2,1} = \cyc{1,2,\ldots,i}^{-1}$ for each
$i$. Therefore, the left hand sides are equal as well. In other words,
$\BB_{\tup{1, n-1}} = S\tup{\A}$.
\end{proof}

Since the antipode $S$ is a $\kk$-algebra anti-morphism, we have
\begin{align}
S\tup{\A\wo} = S\tup{\wo} S\tup{\A} = \wo \BBA
\label{eq.SAwo}
\end{align}
(since $S\tup{\wo} = \wo^{-1} = \wo$ and $S\tup{\A} = \BBA$ by Proposition~\ref{prop.BBa-A}).
Thus, if $\kk$ is a field of characteristic $0$,
then 
the element $\wo\BBA = S\tup{ \A\wo } $ has the same minimal polynomial as $\A\wo$
(by \eqref{eq.mu.fa=a.antimor}, since $S : \ksn \to \ksn$ is an algebra anti-isomorphism).
From Theorem \ref{thm.mp.1}, this minimal polynomial is 
\[
\mu\tup{\wo\BBA} = \prod_{k \in \nn_{\tup{ 1,n-1 }}\tup{\CF}} \tup{ x-k }.
\]
To make this concrete, we will determine the set
$\nn_{\tup{1,n-1}}\tup{\CF}$ indexing this product.

\begin{prop}
\label{prop.1n-1.specs}
Assume that $n > 1$. 
Then, the set $\nn_{\tup{1,n-1}}\tup{\CF}$ is the set
\[
L\tup{ n } := \set{-n+2} \cup \interval{-n+4, n-3} \cup \set{0} \cup \set{n}.
\]
\end{prop}

Once this proposition is proved, we will be able to conclude:

\begin{thm}
    \label{bb1_thm2}
    Assume that $\kk$ is a field of characteristic $0$,
	and that $n > 1$.
	Then, the elements $ \wo\A $ and $ \A\wo $ of $ \ksn $ have minimal polynomial
    \[
        \mu\tup{\wo\A}
		= \mu\tup{\A\wo}
		= \prod_{k \in L\tup{ n }} \tup{ x-k }.
    \]
\end{thm}

\begin{proof}
    The elements $ \wo\A $ and $ \A\wo $ are conjugate (via conjugation by $\wo$), and thus have the same minimal polynomial.
	That is, $\mu\tup{\wo\A} = \mu\tup{\A\wo}$.
	Furthermore, $S : \ksn \to \ksn$ is a $\kk$-algebra anti-isomorphism.
	Hence, \eqref{eq.mu.fa=a.antimor} yields
	$\mu\tup{S\tup{\A\wo}} = \mu\tup{\A\wo}$.
	By \eqref{eq.SAwo}, this rewrites as
	$\mu\tup{\wo\BBA} = \mu\tup{\A\wo}$.
	But by Theorem \ref{thm.mp.1}, the minimal polynomial of $ \wo\BBA $ is
    \[
        \mu\tup{\wo\BBA} = \prod_{k \in \nn_{\tup{1,n-1}}\tup{\CF}} \tup{ x-k } = \prod_{k \in L\tup{ n }} \tup{ x-k } \quad\tup{ \text{by Proposition \ref{prop.1n-1.specs}} }.
    \]
    Comparing these two equalities, we find $\mu\tup{\A\wo} = \prod_{k \in L\tup{ n }} \tup{ x-k }$. Together with $\mu\tup{\wo\A} = \mu\tup{\A\wo}$, this completes our proof.
\end{proof}

It remains to prove Proposition \ref{prop.1n-1.specs}. We begin with a definition.

\begin{defn}[Singleton blocks]
    For each face $F = \tup{F_1, F_2, \ldots, F_k} \in \CF$,
    we set
    \begin{align*}
    \Sin F &:= \set{F_i : \card{F_i} = 1}
	= \set{\text{all blocks of $F$ that have size $1$}}
    \qquad \text{ and } \\
    \sin F &:= \card{\Sin F}.
    \end{align*}
    Thus, $\sin F$ is the number of all size-$1$ blocks (``singleton blocks'') of $F$.
    (We believe that the reader will not be tempted to mistake $\sin F$ for a sine.)
\end{defn}
\begin{exmp}
The set composition $\tup{4, 25, 1, 7, 36}$ of $\interval{7}$ has singleton blocks $\set{4}$, $\set{1}$, $\set{7}$, so we have $\Sin\tup{4, 25, 1, 7, 36} = \set{\set{4},\set{1},\set{7}}$ and $\sin\tup{4, 25, 1, 7, 36} = 3$.

\end{exmp}

\begin{lem}
    \label{lem.spec.form}
    Consider the composition $\tup{1, n-1}$ of $n>1$.
	For each $F \in \CF$, we have
    \begin{align}
        \nA\tup{ F } &= \sin F
		\qquad \text{and}
		\label{eq.lem.spec.form.na}
		\\
		\nn_{\tup{1,n-1}}\tup{F}
		&= \tup{ -1 }^{n-\ell\tup{ F }} \sin F.
		\label{eq.lem.spec.form.nna}
    \end{align}
\end{lem}

\begin{proof}
Let $\alpha = \tup{1, n-1}$.
Fix $F \in \CF$.

The faces $G \in \CF$ having type $\alpha = \tup{1, n-1}$
are precisely the $n$ faces of the form
\[
\tup{\set{k}, \interval{n} \setminus \set{k}}
\qquad \text{ for some }  k \in \interval{n}.
\]
Hence, in particular, all the faces $G \in \Na\tup{F}$ 
have this form.

Moreover, if a face $G \in \Na\tup{F}$ is written as $\tup{\set{k}, \interval{n} \setminus \set{k}}$, then $\set{k}$ is a block of $F$ (indeed, $G \in \Na\tup{F}$ entails $F \con G$, so that each block of $F$ is a subset of a block of $G$;
in other words, each block of $F$ is either a subset of $\set{k}$ or a subset of $\interval{n}\setminus\set{k}$;
but the only way this is possible is when $\set{k}$ is a block of $F$).
Thus, it must satisfy $\set{k} \in \Sin F$.
Hence, we can define a map
\begin{align*}
\Na\tup{F} &\to \Sin F, \\
\tup{\set{k}, \interval{n}\setminus\set{k}} &\mapsto \set{k}.
\end{align*}
This map is injective (obviously) and surjective (since each singleton block $\set{k}$ of $F$ does satisfy $F \con \tup{\set{k}, \interval{n}\setminus\set{k}}$ and therefore $\tup{\set{k}, \interval{n}\setminus\set{k}} \in \Na\tup{F}$).
Hence, it is bijective.
Thus, $\card{\Na\tup{F}} = \card{\Sin F} = \sin F$, so that
\begin{align}
\na\tup{F} = \card{\Na\tup{F}} = \sin F.
\label{pf.prop.1n-1.specs.sin=na}
\end{align}
This proves \eqref{eq.lem.spec.form.na} (since $\alpha = \tup{1, n-1}$).
Moreover, the definition of $\nna\tup{F}$ shows that
\begin{align}
\nna\tup{F}
= \tup{-1}^{n - \ell\tup{F}} \na\tup{F}
= \tup{-1}^{n - \ell\tup{F}} \sin F
\label{pf.prop.1n-1.specs.sin=nta}
\end{align}
(by \eqref{pf.prop.1n-1.specs.sin=na}).
This proves \eqref{eq.lem.spec.form.nna}, and thus completes the proof.
\end{proof}

\begin{proof}[Proof of Proposition \ref{prop.1n-1.specs}.]
We assume that $n>3$; the other cases ($n=2$ and $n=3$) are left to the reader.

Let $\alpha = \tup{1,n-1}$.
Thus, we must prove that
$\nna\tup{\CF} = L\tup{n}$.
In other words, we must prove that each
$r \in L\tup{n}$ belongs to $\nna\tup{\CF}$,
and conversely, that each
$k \in \nna\tup{\CF}$ belongs to $L\tup{n}$.

In order to prove that each $r \in L\tup{n}$ belongs to $\nna\tup{\CF}$, we
need to construct a face $F\in \CF$ satisfying $\nna \tup{F} = r$.
We do this case by case:

\begin{enumerate}
\item If $r = \pm s$ for some $s\in \interval{0, n-4}$, then we set
\begin{align*}
\text{either }F  & =\left(  \set{1} , \set{2},
\ldots, \set{s}  , \interval{s+1, n} \right)  \\
\text{or }F  & = \left(  \set{1} , \set{2},
\ldots, \set{s}  , \set{s+1, s+2} , \interval{s+3, n}
\right)  .
\end{align*}
Both of these faces satisfy $\sin F=s$ (here we are using $s\leq n-4$, which
ensures that the last block $\interval{s+3, n}$ of the second $F$ has
size $>1$), but their lengths are $s+1$ and $s+2$ respectively, so that the
corresponding $\tup{-1}^{n-\ell\tup{F} }$ factors are $1$
and $-1$ in some order. Thus, by \eqref{eq.lem.spec.form.nna}, we
conclude that the signed knapsack numbers
$\nna\tup{F} = \nn_{\tup{1, n-1}}\tup{F}$ of these two faces
$F$ are $s$ and $-s$ in some order; in particular, one of them is
$r$ (since $r = \pm s$).
This shows that $r$ is the signed knapsack number of a face.
In other words, $r \in \nna\tup{\CF}$.

\item If $r = n-3$, then we set $F = \left(  \set{1}, \set{2},
\ldots, \set{n-3}, \interval{n-2, n}  \right)  $.
This face satisfies $\sin F=n-3$ and $\ell\tup{F} =n-2$, so that
\eqref{eq.lem.spec.form.nna} yields
$\nna\tup{F} =\tup{-1}^{n-\tup{n-2}} \tup{n-3} = n-3 = r$.
Therefore, $r$ is the signed knapsack number of a face.
That is, $r \in \nna\tup{\CF}$.

\item If $r = -n+2$, then we set $F = \left(  \set{1}, \set{2},
\ldots, \set{n-2}, \interval{n-1, n}  \right)  $.
This face satisfies $\sin F=n-2$ and $\ell\tup{F} =n-1$, so that
\eqref{eq.lem.spec.form.nna} yields
$\nna\tup{F} =\tup{-1}^{n-\tup{n-1}} \tup{n-2} = -n+2 = r$.
Therefore, $r$ is the signed knapsack number of a face.
In other words, $r \in \nna\tup{\CF}$.

\item If $r = n$, then we set $F = \left(  \set{1}, \set{2},
\ldots, \set{n}  \right)  $, which satisfies
$\nna\tup{F} =n=r$.
Thus, $r = \nna\tup{F} \in \nna\tup{\CF}$.
\end{enumerate}

This shows that each $r \in L\tup{n}$ belongs to
$\nna\tup{\CF}$.

It remains to prove that conversely, each $k \in \nna\tup{\CF}$
belongs to $L\tup{n}$. So we fix a $k \in \nna\tup{\CF}$. Then,
$k=\nna\tup{F} $ for some face $F\in \CF$. Consider this $F$.
Recall that $\alpha = \tup{1, n-1}$.
Hence, by \eqref{eq.lem.spec.form.nna}, we have
$\nna\tup{F} = \tup{-1}^{n-\ell\tup{F} }\sin F = \pm\sin F$.
But $\sin F$ is
clearly an element of $\interval{0, n}$, and cannot be $n-1$ (because if
$F$ contains $n-1$ singleton blocks, then the last remaining block must also be
a singleton block, so $F$ actually has $n$ singleton blocks). Thus, $\sin
F\in \interval{0, n}  \setminus \set{n-1} $, so that
\[
k = \nna\tup{F}
= \pm \underbrace{\sin F}_{\in \interval{0, n} \setminus \set{ n-1 } }
\in \interval{-n, n} \setminus \set{ -n+1,n-1 } .
\]
This is almost enough to show that $k\in L\left(  n\right)$; we only need to
prove that $k$ cannot be $-n$ or $-n+3$ or $n-2$. This can be done fairly easily:

\begin{enumerate}
\item If $k$ were $-n$, then $\sin F$ would be $n$ (since $k=\pm\sin F$), which
would mean that $F$ has $n$ singleton blocks; thus, $F$ would have the form
$\left(  \set{\sigma\tup{1}}  ,\set{\sigma\tup{2}}  ,\ldots,\set{\sigma\tup{n}}
\right)$ for some $\sigma \in S_n$.
But such faces $F$ have signed knapsack number
$\nna\tup{F} = \tup{-1}^{n-\ell\tup{F} }\sin F
= \tup{-1}^{n-n} n = n$, not $-n$.

\item If $k$ were $-n+3$, then $\sin F$ would be $n-3$ (since $k=\pm\sin F$),
which would mean that $F$ has $n-3$ singleton blocks, which account for all
but three elements of $\interval{n}$; the remaining three elements would
then be all in one single block (as there is no way to break them into more
than one block without creating singleton blocks). Thus, $F$ would have $n-2$
blocks in total, and therefore the signed knapsack number $\nna \tup{F}$
would be $\tup{-1}^{n-\left(  n-2\right)
}\left(  n-3\right)  =n-3$, not $-n+3$.

\item If $k$ were $n-2$, then a similar argument would show that
$\nna \tup{F}$ is $-n+2$, not $n-2$.
\end{enumerate}

Hence, $k$ cannot be $-n$ or $-n+3$ or $n-2$. This shows that $k\in L\left(
n\right)$.
In other words, $k$ belongs to $L\tup{n}$.
This completes the proof of Proposition \ref{prop.1n-1.specs}.
\end{proof}

\section{\label{sec.pos}Positive linear combinations}


Our above results can be generalized with fairly little effort, from single elements $\BBa$ to their linear combinations $\sum_{\alpha \models n} \gamma_\alpha \BBa \in \dsn$.
However, since the proofs use inequalities between the knapsack numbers $\na\tup{F}$ (such as Proposition~\ref{na.prop.1}), we cannot let the coefficients $\gamma_\alpha$ be arbitrary.
It turns out that requiring them to be nonnegative reals is sufficient.
This idea goes back to Brown \cite[Theorem 5]{Brown-SRM} and was stated explicitly for the $\BBa$ by Schocker \cite[Theorem 4.1]{Schocker}, but we shall now extend it to $\wo\BBa$ and prove it anew.

Instead of fixing a composition $\alpha \models n$,
\textbf{we now fix a family $\bgamma = \tup{\gamma_\alpha}_{\alpha \models n}$ of nonnegative real numbers $\gamma_\alpha \geq 0$ indexed by the compositions $\alpha$ of $n$}.
We assume that the ring $\kk$ is ordered and contains these numbers $\gamma_\alpha$ (for instance, $\kk = \RR$, or $\kk = \QQ$ if these numbers are rational).

We define the element
\[
\BB_{\bgamma} := \sum_{\alpha \models n} \gamma_\alpha \BBa \in \dsn.
\]

For each face $F \in \CF$, we define the \defin{$\bgamma$-weighted knapsack number of $F$} to be the number
\[
n_{\bgamma}\tup{F} := \sum_{\alpha \models n} \gamma_\alpha n_\alpha\tup{F} \in \kk,
\]
and we define the \defin{$\bgamma$-weighted signed knapsack number of $F$} to be the number
\[
\nn_{\bgamma}\tup{F} := \sum_{\alpha \models n} \gamma_\alpha \nn_\alpha\tup{F} \in \kk.
\]
We set $n_{\bgamma}\tup{\CF} = \set{n_{\bgamma}\tup{F} \mid F \in \CF}$
and $\nn_{\bgamma}\tup{\CF} = \set{\nn_{\bgamma}\tup{F} \mid F \in \CF}$.

Now, we claim the following generalizations of Theorem \ref{main.thm.3a} and Theorem \ref{main.thm.3}:

\begin{thm}
\label{thm.gen.3a}
The element $\BB_{\bgamma} \in \dsn$ satisfies
\begin{equation}
\prod_{k\in n_{\bgamma}\tup{\CF}} \tup{\BB_{\bgamma} - k}
=0.
\label{eq.gen.3a}
\end{equation}
\end{thm}

\begin{thm}
\label{thm.gen.3}
The elements $\wo \BB_{\bgamma}$ and $\BB_{\bgamma} \wo$ of $\dsn$ satisfy
\begin{equation}
\prod_{k \in \nn_{\bgamma}\tup{\CF}} \tup{\wo\BB_{\bgamma} - k}
=0
\label{eq.gen.2a}
\end{equation}
and
\begin{equation}
\prod_{k \in \nn_{\bgamma}\tup{\CF}} \tup{\BB_{\bgamma}\wo - k}
=0.
\label{eq.gen.2b}
\end{equation}
\end{thm}

The proofs of Theorem \ref{thm.gen.3a} and Theorem \ref{thm.gen.3} proceed mostly analogously to the proofs of Theorem \ref{main.thm.3a} and Theorem \ref{main.thm.3};
we just need to replace every object $O_\alpha$ that depends on a composition $\alpha$ by the corresponding linear combination $\sum_{\alpha \models n} \gamma_\alpha O_\alpha$ of such objects.
In particular, we need to
\begin{itemize}
\item define $\CF_k$ for all $k \in \RR_+\cup\set{\infty}$ (rather than only for $k \in \NN\cup\set{\infty}$), where $\RR_+$ is the set of all nonnegative reals;
\item replace $\BBa$ by $\BB_{\bgamma}$,
\item replace $\BBta$ by $\BBt_{\bgamma} := \sum_{\alpha \models n} \gamma_\alpha \BBta \in \kfsn$ (noting that $\rho\tup{\BB_{\bgamma}} = \BBt_{\bgamma}$ follows from \eqref{eq.rhoBa} by linearity),
\item replace $\BBt_{\rev \alpha}$ by $\BBt_{\rev \bgamma} := \sum_{\alpha \models n} \gamma_\alpha \BBt_{\rev \alpha}$,
\item replace $\na\tup{F}$ by $n_{\bgamma}\tup{F}$ (thus, in particular, $\CF_k$ now means $\set{F \in \CF : n_{\bgamma}\tup{F} \geq k }$),
\item replace $n_{\rev \alpha}\tup{F}$ by $n_{\rev \bgamma}\tup{F} := \sum_{\alpha \models n} \gamma_\alpha n_{\rev \alpha}\tup{F}$,
\item replace $\nna\tup{F}$ by $\nn_{\bgamma}\tup{F}$ (which equals $\tup{-1}^{n-\ell\tup{F}} n_{\bgamma}\tup{F}$),
\item etc. (beware: $\BBt_{\rev \alpha} \BBta$ must be replaced by $\BBt_{\rev \bgamma} \BBt_{\bgamma}$, not by $\sum_{\alpha \models n} \gamma_\alpha \BBt_{\rev \alpha} \BBta$).
\end{itemize}
The reader can easily make the necessary changes to almost all relevant results in Sections \ref{sec.facealg}, \ref{sec.longpol} and \ref{sec.shortpol}, including the lemmas.\footnote{Do not bother changing Proposition~\ref{prop.k0} and Theorem~\ref{thm.main.1}, as these are not needed anyway. No changes are needed in Section~\ref{sec.altsum}.}
Some of the modified lemmas can be derived from the original unmodified ones, whereas others can be proved analogously.
For instance:
\begin{itemize}
\item Corollary~\ref{w0conj_cor} must be replaced by the statement that $\wotil \BBt_{\bgamma} \wotil = \BBt_{\rev \bgamma}$ and $\tup{\wotil \BBt_{\bgamma}}^2 = \BBt_{\rev \bgamma} \BBt_{\bgamma}$. The first of these equalities follows by applying \eqref{eq.w0conj_cor.1} to each composition $\alpha$ of $n$, multiplying it by $\gamma_\alpha$ and summing the resulting equalities. The second equality follows from the first in the same way as \eqref{eq.w0conj_cor.2} was derived from \eqref{eq.w0conj_cor.1}.
\item Proposition~\ref{na.prop.2} must be replaced by the statement that each $F \in \CF$ satisfies $n_{\bgamma}\tup{F} = n_{\rev \bgamma}\tup{F}$. This follows from applying Proposition~\ref{na.prop.2} to each $\alpha \models n$, multiplying the result by $\gamma_\alpha$ and summing over all $\alpha$'s.
\item Lemma~\ref{main.cor.1} must be replaced by the statement that every $F \in \CF$ and every $i \in \interval{0, m}$ satisfying $n_{\bgamma} \tup{ F } = k_i $ satisfy $F \BBt_{\rev \bgamma} \BBt_{\bgamma} \equiv k_i^2 F \mod \kf_{k_{i+1}}$, where we are now writing the set $n_{\bgamma}\tup{\CF}$ (not $\na\tup{\CF}$) as $ \set{k_0 < k_1 < k_2 < \cdots < k_m} $. The proof of this is analogous to the above proof of Lemma~\ref{main.cor.1}, of course using the generalized versions of Proposition \ref{na.prop.2} and Lemma \ref{main.lem.1}.
\end{itemize}

More substantial changes are needed to the proof of Proposition~\ref{na.prop.1} and to the proof of Lemma~\ref{main.lem.1}; thus we shall dwell on them now.
Proposition~\ref{na.prop.1} must be replaced by the following:

\begin{prop}
    \label{gen.na.prop.1}
	Let $F,G \in \CF$ be such that $\gamma_{\type G} > 0$ and $F \not\con G$.
	Then, we have $n_{\bgamma}\tup{FG} > n_{\bgamma}\tup{F}$.
\end{prop}

\begin{proof}
Proposition~\ref{na.prop.0} shows that $\na\tup{FG} \geq \na\tup{F}$ for each composition $\alpha$.
Multiplying this inequality by the nonnegative scalar $\gamma_\alpha$ and summing over all $\alpha$, we obtain
\[
\sum_{\alpha \models n} \gamma_\alpha n_\alpha\tup{FG}
\geq \sum_{\alpha \models n} \gamma_\alpha n_\alpha\tup{F}.
\]
In other words, $n_{\bgamma}\tup{FG} \geq n_{\bgamma}\tup{F}$.
It remains to prove that this inequality is strict (that is, $>$ rather than $=$).
For this purpose, we need to ensure that at least one of the inequalities $\gamma_\alpha n_\alpha\tup{FG} \geq \gamma_\alpha n_\alpha\tup{F}$ that were summed is strict.
In other words, we need to find a composition $\alpha \models n$ satisfying $\gamma_\alpha > 0$ and $\na\tup{FG} > \na\tup{F}$.
But $\type G$ is such a composition $\alpha$, since $\gamma_{\type G} > 0$ (by assumption) and $n_{\type G}\tup{FG} > n_{\type G}\tup{F}$ (by Proposition~\ref{na.prop.1}).
\end{proof}

We now come to the generalized Lemma~\ref{main.lem.1}. The changes to the lemma itself are pretty obvious, but the proof requires some adapting:

\begin{lem}
    \label{gen.lem.1}
    Let $ F \in \CF $. Let $i \in \interval{0, m}$ be such that $n_{\bgamma}\tup{F} = k_i$. Then we have
    \[
    F\BBt_{\bgamma} \equiv k_i F  \mod \kf_{k_{i+1}}.
    \]
	(Note that $ \set{k_0 < k_1 < k_2 < \cdots < k_m} $ now stands for $n_{\bgamma}\tup{\CF}$, not for $\na\tup{\CF}$.)
\end{lem}

\begin{proof}
   The definition of $\BBt_{\bgamma}$ yields
   \[
   \BBt_{\bgamma}
   = \sum_{\alpha \models n} \gamma_\alpha \BBta
   = \sum_{\alpha \models n} \gamma_\alpha \sum_{\substack{G \in \CF; \\ \type G = \alpha}}G
   = \sum_{G \in \CF} \gamma_{\type G} G.
   \]
   Multiplying this by $F$ from the left, we find
   \begin{align}
       F\BBt_{\bgamma}
	   &= \sum_{G \in \CF} \gamma_{\type G} FG
	   = \sum_{\substack{G \in \CF; \\ F \con G}} \gamma_{\type G} \underbrace{FG}_{\substack{=F \\ \text{(by Proposition \ref{cp_prop} \textbf{(c)})}}}
	   + \sum_{\substack{G \in \CF; \\ F \not\con G}} \gamma_{\type G} FG
	   \nonumber\\
       &= \sum_{\substack{G \in \CF; \\ F \con G}} \gamma_{\type G} F + \sum_{\substack{G \in \CF; \\ F \not\con G}} \gamma_{\type G} FG.
	   \label{eq.gen.lem.1.1}
   \end{align}
   But
   \begin{align}
		\sum_{\substack{G \in \CF; \\ F \con G}} \gamma_{\type G}
		&= \sum_{\alpha \models n} \underbrace{\sum_{\substack{G \in \CF; \\ F \con G; \\ \type G = \alpha}} \gamma_\alpha}_{= \card{\Na\tup{F}} \gamma_\alpha = \na\tup{F} \gamma_\alpha}
		= \sum_{\alpha \models n} \na\tup{F} \gamma_\alpha
		= \sum_{\alpha \models n} \gamma_\alpha \na\tup{F}
		= n_{\bgamma}\tup{F}
		\nonumber \\
		&= k_i.
		\label{eq.gen.lem.1.3}
   \end{align}
   On the other hand, if $G \in \CF$ satisfies $\gamma_{\type G} > 0$ and $F \not\con G$, then Proposition \ref{gen.na.prop.1} yields $ n_{\bgamma}\tup{ FG } > n_{\bgamma}\tup{ F } = k_i $ and thus $ n_{\bgamma}\tup{ FG } \ge k_{i+1}$ (by the standard argument), so that $FG \in \CF_{k_{i+1}} \subseteq \kf_{k_{i+1}}$ and therefore
   \begin{align}
	   \gamma_{\type G} FG \equiv 0 \mod \kf_{k_{i+1}}.
	   \label{eq.gen.lem.1.2}
   \end{align}
   The same conclusion can be reached if $\gamma_{\type G} = 0$, albeit for the trivial reason that the left hand side of this congruence is $0$ then.
   Hence, \eqref{eq.gen.lem.1.2} holds for each $G \in \CF$ satisfying $F \not\con G$.
   Thus, \eqref{eq.gen.lem.1.1} becomes
   \begin{align*}
       F\BBt_{\bgamma}
       &= \underbrace{\sum_{\substack{G \in \CF; \\ F \con G}} \gamma_{\type G}}_{\substack{= k_i \\ \text{(by \eqref{eq.gen.lem.1.3})}}} F
	   + \sum_{\substack{G \in \CF; \\ F \not\con G}}\ \ \underbrace{\gamma_{\type G} FG}_{\substack{\equiv 0 \mod \kf_{k_{i+1}} \\ \text{(by \eqref{eq.gen.lem.1.2})}}}
	   \equiv k_i F \mod \kf_{k_{i+1}}.
   \end{align*}
\end{proof}

By making similar changes to the results in Section~\ref{sec.minpol}, we obtain the following generalizations of Theorem~\ref{thm.mp.1} and Theorem~\ref{thm.spec.w0Baright}:

\begin{thm}
\label{thm.gen.mp.1}
Assume that $\kk$ is a field of characteristic $0$.
Then, the minimal polynomial of the element $\wo\BB_{\bgamma}$ of $\ksn$ is
\[
\mu\tup{\wo\BB_{\bgamma}}
=
\prod_{k \in \nn_{\bgamma}\tup{\CF}}\tup{x-k}  .
\]
\end{thm}

\begin{thm}
\label{thm.gen.spec.w0Baright}
Assume that $\kk$ is a field of characteristic $0$.
Then, the $\kk$-linear operator
\begin{align*}
\kf  & \to \kf,\\
\FF  & \mapsto \FF\wotil\BBt_{\bgamma}
\end{align*}
is diagonalizable, and its eigenvalues are the $\bgamma$-weighted signed knapsack numbers
$\nn_{\bgamma}\tup{F} $ of the faces $F\in \CF$.
(The algebraic multiplicity of each eigenvalue $\lambda$ is the number of faces $F$
whose $\bgamma$-weighted signed knapsack number is $\lambda$.)
\end{thm}

Actually, the claim about eigenvalues and their algebraic multiplicities in Theorem~\ref{thm.gen.spec.w0Baright} holds even if $\kk$ is just a field (not necessarily ordered, not necessarily of characteristic $0$) and the $\gamma_\alpha$ are arbitrary scalars (not necessarily $\geq 0$).
However, the operator is not always diagonalizable in this generality.
Similarly, if we drop all assumptions on the field $\kk$ and the scalars $\gamma_\alpha$ in Theorem~\ref{thm.gen.mp.1}, then the minimal polynomial $\mu\tup{\wo\BB_{\bgamma}}$ is still a product of factors of the form $x - k$ with $k \in \nn_{\bgamma}\tup{\CF}$, but some of these factors can appear multiple times.

\appendix

\section{\label{sec.omitted_proofs}Omitted proofs}

\subsection{\label{subsec.omitted_proofs.ccb_prop}Proof of Proposition \ref{ccb_prop}}

\begin{proof}[Proof of Proposition \ref{ccb_prop}]
    Let $A = \set{G \in \CF : F \con G}$, and $B = \set{H : H \models \interval{m}}$.
	Define the maps $f : A \to B$ and $g : B \to A$ by
    \begin{align*}
    f\tup{G_1, G_2, \ldots, G_k} &:= \tup{H_1, H_2, \ldots, H_k} \text{ where } H_{i} = \set{j \in \interval{m} : F_{j} \subseteq G_{i}}\\
    &\qquad\text{and }\\
    g\tup{H_1, H_2, \ldots, H_k} &:= \tup{G_1, G_2, \ldots, G_k} \text{ where } G_{i} = \bigcup\limits_{j \in H_{i}} F_{j}.
    \end{align*}
    We first show these maps are well-defined:
	
	\begin{enumerate}
	\item To see that $f$ is well-defined, we let $G = \tup{G_1, G_2, \ldots, G_k} \in A$. Define the set $H_{i} = \set{j \in \interval{m} : F_{j} \subseteq G_{i}}$ for each $i \in \interval{k}$.
	We must show that $\tup{H_1, H_2, \ldots, H_k} \in B$, that is, $\tup{H_1, H_2, \ldots, H_k} \models \interval{m}$.
	
	From $G = \tup{G_1, G_2, \ldots, G_k} \in A$, we see that
	$G = \tup{G_1, G_2, \ldots, G_k} \models \interval{n}$
	and $F \con G$.
	Thus, each block $F_{j}$ of $F$ is a subset of some $G_{i}$ (since $F \con G$) and therefore satisfies $j \in H_{i}$ for this $i$, so that $j \in \bigcup_{i \in \interval{k}} H_i$.
	Hence, $\bigcup_{i \in \interval{k}} H_i = \interval{m}$.
	
	To see that the sets $H_{i}$ are disjoint, note that if $j \in H_{i} \cap H_{i'}$ for two distinct indices $i$ and $i'$, then $F_{j} \subseteq G_{i}$ and $F_{j} \subseteq G_{i'}$, so that $F_{j} \subseteq G_{i} \cap G_{i'} = \varnothing$ (since the blocks of $G$ are disjoint), which contradicts the nonemptiness of $F_j$.
	
	Finally, we claim that the sets $H_i$ are nonempty.
	Indeed, let $i \in \interval{k}$. Then, the block $G_i$ is nonempty, so there exists some $x \in G_i$. Pick this $x$, and let $F_j$ be the block of $F$ that contains $x$. From $F \con G$, we see that $F_j$ is a subset of \textbf{some} $G_p$. This $G_p$ must then satisfy $x \in F_j \subseteq G_p$ and thus $x \in G_p \cap G_i$ (since $x \in G_i$), so that $G_p \cap G_i \neq \varnothing$, which is impossible unless $p = i$ (since different blocks of $G$ are disjoint). Hence, $p = i$, so that $F_j \subseteq G_p = G_i$.
	Therefore, $j \in H_i$, and thus $H_i$ is nonempty.
	
	Thus we have shown that the sets $H_i$ are disjoint and nonempty and satisfy $\bigcup_{i \in \interval{k}} H_i = \interval{m}$. In other words, $\tup{H_1, H_2, \ldots, H_k} \models \interval{m}$.
	This shows that $f$ is well-defined.
	
	\item To see that $g$ is well-defined, we let $H = \tup{H_1, H_2, \ldots, H_k} \in B$, that is, $\tup{H_1, H_2, \ldots, H_k} \models \interval{m}$.
	We define the set $G_i = \bigcup_{j \in H_{i}} F_{j}$ for each $i \in \interval{k}$.
	We must show that $\tup{G_1, G_2, \ldots, G_k} \in A$, that is, we must show that $\tup{G_1, G_2, \ldots, G_k} \in \CF$ and $F \con \tup{G_1, G_2, \ldots, G_k}$.
	
	Each of the sets $G_i$ is nonempty, since it is a nonempty union of nonempty sets (as each $H_i$ and each $F_j$ are nonempty).
	Furthermore, their definition shows that
	$\bigcup_{i \in \interval{k}} G_i
	= \bigcup_{i \in \interval{k}} \bigcup_{j \in H_i} F_j
	= \bigcup_{j \in \interval{m}} F_j
	= \interval{n}$,
	where the second equality sign is a consequence of the fact that each $j \in \interval{m}$ belongs to some $H_i$.
	Finally, the sets $G_i$ are disjoint, since each element of $\interval{n}$ belongs to exactly one $F_j$, and since the corresponding $j$ belongs to exactly one $H_i$.
	These three facts together show that $\tup{G_1, G_2, \ldots, G_k} \models \interval{n}$, that is, $\tup{G_1, G_2, \ldots, G_k} \in \CF$.
	
	It remains to show that $F \con \tup{G_1, G_2, \ldots, G_k}$.
	This is again clear: Each block $F_j$ of $F$ is a subset of a block of $G$ (namely, of $G_i$, where $i$ is such that $j \in H_i$; this $i$ exists because $H$ is a set composition of $\interval{m}$).
	
	Altogether, we have thus shown that $\tup{G_1, G_2, \ldots, G_k} \in A$.
	The map $g$ is therefore well-defined.
	\end{enumerate}
	
	Now we show that these maps $f$ and $g$ are mutually inverse:
	\begin{enumerate}
	\item To show that $g \circ f = \id$, we must prove that $g\tup{f\tup{G}} = G$ for each set composition $G = \tup{G_1, G_2, \ldots, G_k} \in A$. So let us fix such a set composition $G$. 
	Let $i \in \interval{k}$.
	Then, the $i$-th block of $f\tup{G}$ is $f\tup{G}_i = \set{j \in \interval{m} : F_{j} \subseteq G_i}$ (by the definition of $f$).
	Hence, the $i$-th block of $g\tup{f\tup{G}}$ is in turn
    \begin{align*}
    g\tup{f\tup{G}}_i &= \bigcup_{j \in f\tup{G}_i} F_{j} = \bigcup_{j \in \interval{m} \text{ satisfies } F_{j} \subseteq G_i} \underbrace{F_{j}}_{\subseteq G_i} \subseteq G_i.
    \end{align*}
	Now, let us prove the reverse inclusion $G_i \subseteq g\tup{f\tup{G}}_i$.
	Indeed, let $x \in G_i$.
	Since $g\tup{f\tup{G}}$ is a set composition of $\interval{n}$, we have $\interval{n} = \bigcup_{s \in \interval{k}} g\tup{f\tup{G}}_{s}$.
	Hence, $x \in G_i \subseteq \interval{n} = \bigcup_{s \in \interval{k}} g\tup{f\tup{G}}_{s}$.
	In other words, $x \in g\tup{f\tup{G}}_{s}$ for some $s \in \interval{k}$.
	Consider this $s$.
	Just as we showed $g\tup{f\tup{G}}_i \subseteq G_i$, we see that $g\tup{f\tup{G}}_{s} \subseteq G_{s}$.
	Hence, $x \in g\tup{f\tup{G}}_{s} \subseteq G_{s}$.
	Combined with $x \in G_i$, this leads to $x \in G_i \cap G_{s}$, so that $G_i \cap G_s \neq \varnothing$.
	But this is only possible if $i = s$, since otherwise the blocks $G_i$ and $G_{s}$ would be disjoint.
	Hence, we have $i = s$. Thus, from $x \in g\tup{f\tup{G}}_{s}$, we obtain $x \in g\tup{f\tup{G}}_i$.
	Since we have proved this for each $x \in G_i$, we thus conclude that $G_i \subseteq g\tup{f\tup{G}}_i$. Combined with $g\tup{f\tup{G}}_i \subseteq G_i$, this shows that $g\tup{f\tup{G}}_i = G_i$.
	
	So we have shown that $g\tup{f\tup{G}}_i = G_i$ for each $i \in \interval{k}$.
	In other words, $g\tup{f\tup{G}} = G$. Hence, $g \circ f = \id$ is proved.
	\item Next, we show that $f \circ g = \id$.
	Indeed, let $H = \tup{H_1, H_2, \ldots, H_k} \models \interval{m}$. Then, for each $i \in \interval{k}$, we have
	(using the definitions of $f$ and $g$ as before)
    \begin{align*}
    f\tup{g\tup{H}}_{i} &= \set{j \in \interval{m} : F_{j} \subseteq g\tup{H}_{i}}
	= \set{p \in \interval{m} : F_p \subseteq g\tup{H}_{i}}
	\\
	&= \set{p \in \interval{m} : F_p \subseteq \bigcup_{j \in H_{i}} F_{j}}
    = H_i
    \end{align*}
	(since the blocks $F_j$ of $F$ are disjoint and nonempty, so that the only blocks $F_p$ that are subsets of the union $\bigcup_{j \in H_{i}} F_{j}$ are the very addends of this union).
    That is, $f\tup{g\tup{H}} = H$.
	This shows that $f \circ g = \id$.
	\end{enumerate}
	
	Thus, $f$ and $g$ are mutually inverse, and hence are bijections.
	These bijections $f$ and $g$ are obviously length-preserving.
	In particular, $f$ is a length-preserving bijection.
	This proves Proposition \ref{ccb_prop}, since our $f$ is precisely the $f$ defined in the proposition.
\end{proof}

\subsection{\label{subsec.omitted_proofs.cp_prop}Proof of Proposition \ref{cp_prop}}

\begin{proof}[Proof of Proposition~\ref{cp_prop}.]
    Write the faces $F$ and $G$ in the forms
	$F = \tup{F_1, F_2, \ldots, F_k}$ and
	$G = \tup{G_1, G_2, \ldots, G_m}$.
	Thus, of course, $\ell\tup{F} = k$ and $\ell\tup{G} = m$.

    \begin{enumerate}[]
    \item\textbf{(b)} Each block of $FG$ is of the form $F_{i} \cap G_{j}$
    for some $i \in \interval{k}$ and $j \in \interval{m}$ (by the definition of $FG$).
	Hence, it is a subset of a block of $F$ (namely, of $F_{i}$)
	and a subset of a block of $G$ (namely, of $G_{j}$).
	Thus, $FG \con F$ and $FG \con G$, as desired.

    \item\textbf{(c)} ``If'': Assume that $F$ is contained in $G$.
    Then, each block $F_i$ of $F$ is a subset of some block
    $G_{j\tup{i}}$ of $G$,
    and therefore disjoint from all the other blocks of $G$
    (since the blocks of $G$ are disjoint); therefore,
    the intersection $F_i \cap G_{j\tup{i}}$ is nonempty (and
    equals $F_i$),
    whereas all other intersections
	$F_i \cap G_r$ with $r \neq j \tup{i}$ are empty.
    But the definition of $FG$ yields
    \[
            FG = \tup{F_{1} \cap G_{1}, \dots, F_{1} \cap G_m, \dots, F_k \cap G_m}^{\red}
            = \tup{F_1 \cap G_{j\tup{1}}, F_2 \cap G_{j\tup{2}}, \dots, F_k \cap G_{j\tup{k}}},
    \]
    since the reduction operation removes all the empty intersections
    $F_i \cap G_r$ and leaves only the nonempty intersections
    $F_i \cap G_{j\tup{i}}$ around. But since the latter nonempty
    intersections $F_i \cap G_{j\tup{i}}$ are just the blocks $F_i$
    (because $F_i \subseteq G_{j\tup{i}}$), we can rewrite this as
    \[
            FG = \tup{F_1, F_2, \dots, F_k} = F.
    \]
    This proves the ``if'' direction.
    
    ``Only if'': Assume that $FG = F$. But part \textbf{(b)} shows
	that $FG$ is contained in $G$.
	Hence, $F$ is contained in $G$ (since $FG = F$).
    This proves the ``only if'' direction.

    \item\textbf{(d)} The blocks of $FG$ are the nonempty intersections of the
    form $F_{i}\cap G_{j}$ (by the definition of $FG$).
    Each block $F_{i}$ of $F$ gives rise to at least one
    such nonempty intersection $F_{i}\cap G_{j}$ (since $\bigcup\limits_{j}\left(
    F_{i}\cap G_{j}\right)  = F_i \cap \underbrace{\bigcup\limits_j G_j}_{= \interval{n}}
    = F_i \cap \interval{n} = F_{i}$ is nonempty). Consequently, there are at least
    $k$ nonempty intersections of the form $F_{i}\cap G_{j}$ (since there are
    $k$ blocks $F_i$ of $F$).
    In other words, $FG$ has at least $k$ blocks.
    In other words, $\ell\tup{FG} \ge k = \ell\tup{F}$.
    
    It remains to show that this inequality becomes an equality if and only if
    $F$ is contained in $G$.
    
    ``If'': Assume that $F$ is contained in $G$. Then, part \textbf{(c)}
    shows that $FG = F$. Hence, $\ell\tup{FG} = \ell\tup{F}$. Thus, equality
    holds.
    
    ``Only if'': Now assume that $F$ is not contained in $G$.
    Then, there is a block of $F$ -- say, $F_{i_{0}}$ -- that is not a subset of
    any block of $G$. Hence, $F_{i_{0}}$ has nonempty intersection with more than
    one block of $G$. But the blocks of $FG$ are the nonempty intersections of the
    form $F_{i}\cap G_{j}$. Each block $F_{i}$ of $F$ gives rise to at least one
    such nonempty intersection $F_{i}\cap G_{j}$ (as we already saw above),
    but the block $F_{i_{0}}$ gives
    rise to more than one such nonempty intersection (since $F_{i_{0}}$ has
    nonempty intersection with more than
    one block of $G$). Consequently, the total
    number of nonempty intersections of the form $F_{i}\cap G_{j}$ is larger than
    the number of blocks $F_{i}$ of $F$. In other words,
    $\ell\tup{FG} > \ell\tup{F}$.
    Hence, equality does not hold in $\ell\tup{FG} \ge \ell\tup{F}$.
    The proof of part \textbf{(d)} is thus complete.
    
    \item\textbf{(e)} This is similar to the proof of
    $\ell\tup{FG} \ge \ell\tup{F}$ in part \textbf{(d)}.
    (This time, we have to argue that each block $G_j$ of $G$
    gives rise to at least one nonempty intersection of the
    form $F_i \cap G_j$.)
	Alternatively, we can apply part \textbf{(d)} with the
	roles of $F$ and $G$ interchanged, and then observe that
	$\ell\tup{FG} = \ell\tup{GF}$ because the set compositions
	$FG$ and $GF$ differ only in the order of their blocks.
    
    \item\textbf{(a)} Assume that $F$ is contained in $G$.
    Then, $FG = F$ by part \textbf{(c)}.
    But part \textbf{(e)} yields $\ell\tup{FG} \ge \ell\tup{G}$.
    Since $FG = F$, we can rewrite this as
    $\ell\tup{F} \ge \ell\tup{G}$.
    \qedhere
    \end{enumerate}
\end{proof}

\subsection{\label{subsec.omitted_proofs.rho}Proof of Theorem \ref{dsn_kfsn_morph}}

We shall now give a self-contained proof of Theorem \ref{dsn_kfsn_morph}.
This is essentially just
Bidigare's proof (\cite[Theorem 3.8.1]{Bidigare-thesis},
\cite[proof of Theorem 2.1]{Saliola}, \cite[Theorem 1]{Hsiao}, etc.),
restated in our language for the sake of consistency.

We begin with a lemma that is so simple that its usefulness is hard to
believe. It is implicit in most expositions of Theorem \ref{dsn_kfsn_morph},
but in our view it is worth isolating, as it makes the rest of the proof
more transparent.

\begin{lem}
\label{lem.rho.monoid}
Let $B$ and $C$ be two monoids (written
multiplicatively).
Let $f:B\to C$ be any map.

Let $S$ be a set on which both monoids $B$ and $C$ act from the left. Let
$s\in S$ be an element that generates a free $C$-orbit -- i.e., that has the
property that
\begin{equation}
\text{if }c_{1},c_{2}\in C\text{ satisfy }c_{1}s=c_{2}s\text{, then }
c_{1}=c_{2}\text{.}
\label{eq.lem.rho.monoid.free}
\end{equation}
Assume furthermore that
\begin{equation}
b\tup{cs} = c\tup{bs}
\qquad\text{for all }b\in B\text{ and }c\in C.
\label{eq.lem.rho.monoid.bcs}
\end{equation}
Assume moreover that
\begin{equation}
bs=f\tup{b}   s
\qquad\text{for each }b\in B.
\label{eq.lem.rho.monoid.rhobs}
\end{equation}
Then, $f$ is a monoid anti-morphism, i.e., it satisfies
\[
f\tup{ab}  =f\tup{b}   f\tup{a} 
\qquad\text{for all }a,b\in B
\]
and $f\left(  1_{B}\right)  =1_{C}$.
\end{lem}

\begin{proof}
From \eqref{eq.lem.rho.monoid.rhobs}, we have $1_{B}s=f\left(  1_{B}\right)
s$, so that $f\left(  1_{B}\right)  s=1_{B}s=s=1_{C}s$. Thus,
\eqref{eq.lem.rho.monoid.free} (applied to $c_{1}=f\left(  1_{B}\right)$ and
$c_{2}=1_{C}$) yields $f\left(  1_{B}\right)  =1_{C}$.

It thus remains to show that $f\tup{ab}  =f\tup{b}   f\left(
a\right)$ for all $a,b\in B$.

To prove this, we fix any $a,b\in B$. Then, \eqref{eq.lem.rho.monoid.rhobs}
(applied to $ab$ instead of $b$) shows that $abs=f\tup{ab}  s$.
However, \eqref{eq.lem.rho.monoid.rhobs} also yields $bs=f\tup{b}   s$
and $as=f\tup{a}   s$.

On the other hand, \eqref{eq.lem.rho.monoid.bcs} (applied to $a$ and $f\left(
b\right)$ instead of $b$ and $c$) shows that $a\left(  f\tup{b} 
s\right)  =f\tup{b}   \left(  as\right)$. In view of $bs=f\left(
b\right)  s$, we can rewrite this as $a\tup{bs}  =f\tup{b} 
\underbrace{\left(  as\right)  }_{=f\tup{a}   s}=f\tup{b} 
f\tup{a}   s$. Comparing this with $a\tup{bs}  =abs=f\left(
ab\right)  s$, we find $f\tup{ab}  s=f\tup{b}   f\left(
a\right)  s$. Therefore, \eqref{eq.lem.rho.monoid.free} (applied to
$c_{1}=f\tup{ab}$ and $c_{2}=f\tup{b}   f\tup{a} 
$) yields $f\tup{ab}  =f\tup{b}   f\tup{a} $. This
completes our proof of Lemma \ref{lem.rho.monoid}.
\end{proof}

\begin{proof}[Proof of Theorem \ref{dsn_kfsn_morph}.]
Recall the left action of
$S_n$ on $\CF$ introduced in Definition \ref{kfsn_defn}. We shall
need one more piece of notation:

For any permutation $w\in S_n$, we let $P_{w}$ be the face
\[
P_{w}
:= \left(  \set{w\tup{1}} ,\ \set{w\tup{2}}  ,
\ \ldots,\ \set{w\tup{n}}  \right)
\in\CF.
\]
In particular, $P_{\id} = \left(  \set{1},\ \set{2},
\ \ldots,\ \set{n}  \right)$. It is easy to see that
each permutation $w\in S_n$ satisfies
\begin{equation}
P_{w}=wP_{\id},
\label{eq.Fw.=wFid}
\end{equation}
where the action of $S_n$ on $\CF$ is the one from Definition
\ref{kfsn_defn}. Note that the $\kk$-module $\kf$ has
a linear left $S_n$-action (as we know from Definition \ref{kfsn_defn}), and
thus is a left $\ksn$-module. Furthermore, the
$\kk$-module $\kf$ is a left $\kfsn$-module
(by left multiplication, since $\kfsn$ is a subalgebra of $\kf$).
We shall use both of these module structures on $\kf$ in what follows.

Now, we recall that the family $\left(  \BBta\right)  _{\alpha\models n}$
is a basis of the $\kk$-module $\kfsn$ (by Proposition \ref{kfsn_prop2}). The
$\kk$-linear map $\rho:\dsn\to\kfsn$ (defined in Definition \ref{defn_rho})
sends the basis $\left(   \BBa \right)  _{\alpha\models n}$ of
$\dsn$ to the basis
$\left(  \BBta\right)  _{\alpha\models n}$ of $\kfsn$
(by its definition), and thus is a $\kk$-module
isomorphism (like any $\kk$-linear map that sends a basis to a basis).
Hence, it has a $\kk$-linear inverse $\rho^{-1} : \kfsn
\to \dsn$. We shall
view $\rho^{-1}$ as a $\kk$-linear map $\kfsn \to \ksn$
(since $\dsn$ is a $\kk$-submodule of $\ksn$).

We shall now attempt to apply Lemma \ref{lem.rho.monoid} to $B=\kfsn$ and $C=\ksn
$ (both viewed as monoids with respect to multiplication) and $f=\rho^{-1}$
and $S=\kf$ (where the monoids $B$ and $C$ act on $S$ since
$S=\kf$ is both a left $\kfsn$-module and a left $\ksn
$-module) and $s=P_{\id}\in\kf$. If we can
show that all the conditions of Lemma \ref{lem.rho.monoid} are satisfied for
these inputs, then Lemma \ref{lem.rho.monoid} will yield that $\rho^{-1}$ is a
monoid anti-morphism (with respect to multiplication), and this will quickly
finish our proof of Theorem \ref{dsn_kfsn_morph} (see below for the details).

We shall now show that the conditions of Lemma \ref{lem.rho.monoid} are
satisfied:

\begin{enumerate}
\item We claim that \eqref{eq.lem.rho.monoid.free} is satisfied. In other
words, we claim that
\[
\text{if }c_{1},c_{2}\in\ksn  \text{ satisfy }
c_{1}P_{\id} = c_{2}P_{\id}\text{, then }
c_{1} = c_{2}\text{.}
\]

\textit{Proof:} In other words, we must show that
the map $\ksn \to \kf,\ c\mapsto cP_{\id}$ is
injective. But this map is $\kk$-linear, and it sends the basis vectors
$w\in S_n$ of the $\kk$-module $\ksn$ to
the elements $wP_{\id}=P_{w}$ (by \eqref{eq.Fw.=wFid}), which
are distinct basis vectors of $\kf$ (since all the faces
$P_{w}$ for $w\in S_n$ are distinct) and therefore are $\kk$-linearly
independent. Thus, this map must be injective (since a $\kk$-linear map
that sends a basis of its domain to a $\kk$-linearly independent family
in its target must always be injective).
This proves that \eqref{eq.lem.rho.monoid.free} is satisfied.

\item We claim that \eqref{eq.lem.rho.monoid.bcs} is satisfied. In other
words, we claim that
\[
b\left(  cP_{\id}\right)  =c\left(  bP_{\id}\right)
\qquad \text{for all } b\in\kfsn \text{ and } c\in\ksn .
\]

\textit{Proof:} Let $b\in\kfsn$ and
$c\in\ksn$. We must prove the equality $b\left(
cP_{\id}\right)  =c\left(  bP_{\id}\right)$.
Note that this equality depends $\kk$-linearly on $c$. Hence, by
linearity, we can WLOG assume that $c$ is a permutation $w\in S_n$. Assume
this, and consider this $w$. Recall that $S_n$ acts on the monoid algebra
$\kf$ by $\kk$-algebra automorphisms (see the proof
of Proposition \ref{kfsn_prop1}); thus,
\[
w\left(  pq\right)  =\left(  wp\right)  \left(  wq\right)
\qquad\text{for all }p,q\in\kf.
\]
Applying this to $p=b$ and $q=P_{\id}$, we obtain $w\left(
bP_{\id}\right)  =\left(  wb\right)  \left(
wP_{\id}\right)$. But $b\in\kfsn$ and thus $wb=b$
(by the definition of $\kfsn$). Hence, $w\left(
bP_{\id}\right)  =\underbrace{\left(  wb\right)  }_{=b}\left(
wP_{\id}\right)  =b\left(  wP_{\id}\right)$.
Since $c=w$, we can rewrite this as $c\left(  bP_{\id}\right)
=b\left(  cP_{\id}\right)$. In other words, $b\left(
cP_{\id}\right)  =c\left(  bP_{\id}\right)$.
This proves that \eqref{eq.lem.rho.monoid.bcs} is satisfied.

\item We claim that \eqref{eq.lem.rho.monoid.rhobs} is satisfied.
In other words, we claim that
\[
bP_{\id}=\rho^{-1}\tup{b}   P_{\id}
\qquad\text{for each }b\in\kfsn.
\]

\textit{Proof:} Let $b\in\kfsn$. We
must show that $bP_{\id}=\rho^{-1}\tup{b} 
P_{\id}$.

Since this equality depends $\kk$-linearly on $b$, we can WLOG assume
that $b = \BBta$ for some composition $\alpha$ of $n$
(since we know from Proposition \ref{kfsn_prop2} that the family $\left(
 \BBta \right)  _{\alpha\models n}$ is a basis of
$\kfsn$). Assume this, and consider
this $\alpha$. Thus, $b = \BBta$. But the definition
of $\rho$ yields $\rho\tup{\BBa} = \BBta $.
Hence, $b = \BBta = \rho\tup{\BBa}$, so that
\[
\rho^{-1}\tup{b}   = \BBa
=\sum\limits_{\substack{w\in S_n;\\
\Des w\subseteq\gaps^{-1}\tup{\alpha}  }}w
\qquad \qquad \left(  \text{by the definition of
} \BBa \right)  .
\]
Hence,
\begin{align}
\rho^{-1}\tup{b}   P_{\id} &  =\left(  \sum
\limits_{\substack{w\in S_n;\\\Des w\subseteq
\gaps^{-1}\tup{\alpha}  }}w\right)
P_{\id}=\sum\limits_{\substack{w\in S_n;\\\Des %
w\subseteq\gaps^{-1}\tup{\alpha}
}}\underbrace{wP_{\id}}_{\substack{=P_{w}\\\text{(by
\eqref{eq.Fw.=wFid})}}}\nonumber\\
&  =\sum\limits_{\substack{w\in S_n;\\\Des w\subseteq
\gaps^{-1}\tup{\alpha}  }}P_{w} .
\label{pf.eq.lem.rho.monoid.rhobs.1}
\end{align}
On the other hand, from $b= \BBta =\sum
\limits_{\substack{F\in\CF;\\\type F=\alpha}}F$, we
obtain
\begin{equation}
bP_{\id}
= \left(  \sum\limits_{\substack{F\in\CF;\\
\type F=\alpha}}F\right)  P_{\id}
= \sum\limits_{\substack{F\in\CF;\\\type F=\alpha
}} FP_{\id}.
\label{pf.eq.lem.rho.monoid.rhobs.2}
\end{equation}
Let us now prove that the right hand sides of
\eqref{pf.eq.lem.rho.monoid.rhobs.1} and of
\eqref{pf.eq.lem.rho.monoid.rhobs.2} are the same.

Indeed, write the composition $\alpha$ as $\alpha=\left(  \alpha_{1}%
,\alpha_{2},\ldots,\alpha_{k}\right)$. Let $\left(  I_{1},I_{2},\ldots
,I_{k}\right)  \in\CF$ be the unique set composition of $\interval{n}$
that has type $\alpha$ and whose blocks $I_{1},I_{2},\ldots,I_{k}$
are intervals of $\interval{n}$ arranged from left to right. That is,
the block $I_{1}$ consists of the $\alpha_{1}$ smallest elements of
$\interval{n}$ (that is, $I_{1}= \interval{1,\alpha_1}$); the next block
$I_{2}$ consists of the $\alpha_{2}$ next-smallest elements of
$\interval{n}$ (that is, $I_{2}= \interval{ \alpha_{1}+1,\alpha_{1}+\alpha
_{2} }$); the next block $I_{3}$ consists of the $\alpha_{3}$
next-smallest elements of $\interval{n}$; and so on. Explicitly, these
blocks $I_{j}$ are given by
\[
I_{j}
= \interval{ \alpha_{1}+\alpha_{2}+\cdots+\alpha_{j-1}+1,
           \ \alpha_{1}+\alpha_{2}+\cdots+\alpha_{j} }
\qquad\text{for all } j \in \interval{k} .
\]
For example, if $n = 8$ and $\alpha = \tup{2,3,1,2}$, then
\[
\tup{I_1,I_2,\ldots,I_k}
= \tup{\set{1,2},\ \set{3,4,5},\ \set{6},\ \set{7,8}}.
\]

Note that the set $\gaps^{-1}\tup{\alpha}$
consists precisely of the numbers $\alpha_{1}+\alpha_{2}+\cdots+\alpha_{j}$
for all $j\in \interval{k-1}$; these numbers are the maxima of the
intervals $I_{1},I_{2},\ldots,I_{k-1}$ (whereas the maximum of $I_{k}$ is
$n$, at least if $k>0$). Thus, the permutations $w\in S_n$ that satisfy
$\Des w \subseteq\gaps^{-1}\tup{\alpha}$ are
exactly those permutations $w\in S_n$ that have no descents except at the
maxima of the intervals $I_{1},I_{2},\ldots,I_{k-1}$ (and possibly not even at
those maxima); in other words, they are exactly those permutations $w\in
S_n$ that are increasing on each of the intervals $I_{1},I_{2}%
,\ldots,I_{k}$. Obviously, such a permutation $w$ is
uniquely determined by the \textbf{sets}
$w \tup{I_1},\ w \tup{I_2},\ \ldots, \ w \tup{I_k}$
(since $w$ must send the elements of each $I_j$ to the elements of
$w\tup{I_j}$ in increasing order), and these sets have respective sizes
$\alpha_{1}, \alpha_{2}, \ldots, \alpha_{k}$
(since $\card{w \tup{I_j}} = \card{I_j} = \alpha_j$
for each $j \in \interval{k}$)
and thus form a set composition $\left(
w \tup{I_1},\ w \tup{I_2},\ \ldots, \ w \tup{I_k}
\right)  \in\CF$ of type $\alpha$.
Thus, we can define an injective map
\begin{align*}
\Omal : \set{ w\in S_n\ \mid\ \Des w\subseteq
\gaps^{-1}\tup{\alpha} }   &
\to \set{ F\in\CF\ \mid\ \type F=\alpha } ,\\
w &  \mapsto\left(
w \tup{I_1},\ w \tup{I_2},\ \ldots, \ w \tup{I_k}
\right)  ,
\end{align*}
which sends each $w$ to the $k$-tuple formed of these sets $w\left(
I_{1}\right)  ,\ w\left(  I_{2}\right)  ,\ \ldots,\ w\left(  I_{k}\right)$.
This map $\Omal$ is a bijection, since any set composition $F
= \left(  F_{1},F_{2},\ldots,F_{k}\right)  \in\CF$ of type
$\alpha$ can be written as
$\left( w \tup{I_1},\ w \tup{I_2},\ \ldots, \ w \tup{I_k}
 \right)$ for a unique permutation $w\in S_n$ that satisfies
$\Des w\subseteq\gaps^{-1}\left(
\alpha\right)$ (namely, the permutation $w$ whose values on each interval
$I_{j}$ are the elements of $F_{j}$ in increasing order).
For example, if $n = 8$ and $\alpha = \tup{2,3,1,2}$ and
if $w = \ponl{3, 5, 1, 4, 8, 2, 6, 7}$, then
$\Omal \tup{w} = \tup{35, 148, 2, 67}$.

Next, we claim that the bijection $\Omal$ has the following
property: For any permutation $w\in S_n$ satisfying
$\Des w\subseteq\gaps^{-1}\tup{\alpha}$, we have
\begin{equation}
\Omal \tup{w}  P_{\id} = P_w .
\label{pf.eq.lem.rho.monoid.rhobs.4}
\end{equation}

[\textit{Proof of \eqref{pf.eq.lem.rho.monoid.rhobs.4}:}
Let $w\in S_n$ be a permutation satisfying
$\Des w\subseteq\gaps^{-1} \tup{\alpha}$. Then,
from $\Omal \left(  w\right)
=\left(  w\left(  I_{1}\right)  ,\ w\left(  I_{2}\right)
,\ \ldots,\ w\left(  I_{k}\right)  \right)$ (which follows from the
definition of $\Omal$) and $P_{\id} = \left(  \set{1}  ,
\ \set{2}  ,\ \ldots,\ \set{n}  \right)$,
we obtain
\begin{align*}
\Omal \tup{w}  P_{\id}
&  = \left(  w\left(
I_{1}\right)  ,\ w\left(  I_{2}\right)  ,\ \ldots,\ w\left(  I_{k}\right)
\right)  \left(  \set{1}  ,\ \set{2}  ,\ \ldots
,\ \set{n}  \right)  \\
&  =\left(  w\left(  I_{1}\right)  \cap\set{1}  ,\ \dots,\ w\left(
I_{1}\right)  \cap\set{n}  ,\ \dots,\ w\left(  I_{k}\right)
\cap\set{n}  \right)  ^{\red}
\end{align*}
(by the definition of the product on $\CF$). In other words,
$\Omal \tup{w}  P_{\id}$ is the set
composition whose blocks are the nonempty intersections of the form $w\left(
I_{j}\right)  \cap \set{p}$ for all $j\in \interval{k}$ and
$p\in\interval{n}$, in the order of lexicographically increasing
$\left(  j,p\right)$. Thus, each block of the set composition
$\Omal\tup{w}  P_{\id}$ has size $1$ (since it has
the form $w\left(  I_{j}\right)  \cap \set{p}$ for some
$j\in \interval{k}$ and $p\in \interval{n}$, and thus is a subset of
the $1$-element set $\set{p}$; but this shows that it has size
$0$ or $1$, and of course it cannot have size $0$ when it is nonempty, so that
it must then have size $1$).

So we have shown that $\Omal \tup{w}  P_{\id}$ is a
set composition of $\interval{n}$ whose each block has size $1$.
Therefore, the blocks of this set composition $\Omal \left(
w\right)  P_{\id}$ are $\set{1},\ \set{2},
\ \ldots,\ \set{n}$ in some order. We are now going
to identify this order. Namely, we shall show that the order in which these
blocks appear in $\Omal \tup{w}  P_{\id}$ is
$\set{w\tup{1}} ,\ \set{w\tup{2}}  ,\ \ldots,\ \set{w\tup{n}}$.
Once this is proved, it will follow that
\[
\Omal \tup{w}  P_{\id}
= \left(  \set{w\tup{1}}  ,\ \set{w\tup{2}},
\ \ldots,\ \set{w\tup{n}}  \right)  =P_{w},
\]
and thus \eqref{pf.eq.lem.rho.monoid.rhobs.4} will be proved.

So it remains to show that the order in which the blocks $\set{1}
,\ \set{2}  ,\ \ldots,\ \set{n}$ appear in
$\Omal \tup{w}  P_{\id}$ is $\set{w\tup{1}}  ,\ \set{w\tup{2}}
,\ \ldots,\ \set{w\tup{n}}$. Clearly, it suffices to
show that if $a,b\in\interval{n}$ are such that $a<b$, then the block
$\set{w\tup{a}}$ appears before\footnote{The word
``before''\ means ``someplace before'', not ``immediately
before''.} $\set{w\tup{b}}$ in
$\Omal \tup{w}  P_{\id}$. So we shall show this.

Fix two elements $a,b\in\interval{n}$ such that $a<b$. We must show
that
\begin{equation}
\set{w\tup{a}}  \text{ appears before }
\set{w\tup{b}}  \text{ in }\Omal \tup{w}
P_{\id}.
\label{pf.eq.lem.rho.monoid.rhobs.4.pf.goal}
\end{equation}

Recall that $\Des w \subseteq \gaps^{-1} \tup{\alpha}$;
thus, the permutation $w$ is increasing on each of
the intervals $I_{1},I_{2},\ldots,I_{k}$ (as we have seen above). We
distinguish between two cases:

\textit{Case 1:} The elements $a$ and $b$ belong to the same interval $I_{i}$
for some $i\in\interval{k}$.

\textit{Case 2:} The elements $a$ and $b$ belong to two different intervals
$I_{i}$. In other words, $a\in I_{i_{1}}$ and $b\in I_{i_{2}}$ for some
$i_{1}\neq i_{2}$ in $\interval{k}$.

Consider Case 1 first. In this case, $a,b\in I_{i}$ for some $i\in
\interval{k}$. Consider this $i$. Then, the permutation $w$ is increasing on
$I_{i}$ (since $w$ is increasing on each of the intervals $I_{1},I_{2},
\ldots,I_{k}$). Thus, from $a,b\in I_{i}$ and $a<b$, we obtain $w\left(
a\right)  <w\tup{b} $. Furthermore, from $a\in I_{i}$, we obtain
$w\tup{a}   \in w\left(  I_{i}\right)$, so that
$w \tup{I_i} \cap \set{w\tup{a}} = \set{w\tup{a}}$.
Likewise,
$w \tup{I_i} \cap \set{w\tup{b}} = \set{w\tup{b}}$.

Now, recall that $\Omal \tup{w}  P_{\id}$ is
the set composition whose blocks are the nonempty intersections of the form
$w\left(  I_{j}\right)  \cap \set{p}$ for all $j\in
\interval{k}$ and $p\in\interval{n}$, in the order of lexicographically
increasing $\tup{j, p}$.
Hence, the intersection $w\left(
I_{i}\right)  \cap\set{w\tup{a}}$ appears before
$w\left(  I_{i}\right)  \cap\set{w\tup{b}}$ in this
set composition (since $w\tup{a}   <w\tup{b} $ and thus
$\tup{i, w\tup{a}} < \tup{i, w\tup{b}}$ in the
lexicographic order).
In other words, the block $\set{w\tup{a}}$ appears before
$\set{w\tup{b}}$ in this set composition
(since $w \tup{I_i} \cap \set{w\tup{a}} = \set{w\tup{a}}$
and $w \tup{I_i} \cap \set{w\tup{b}} = \set{w\tup{b}}$).
This proves
\eqref{pf.eq.lem.rho.monoid.rhobs.4.pf.goal} in Case 1.

Let us now consider Case 2. In this case, $a\in I_{i_{1}}$ and $b\in I_{i_{2}%
}$ for some $i_{1}\neq i_{2}$ in $\interval{k}$. Consider these
$i_{1}\neq i_{2}$. If we had $i_{1}>i_{2}$, then we would have $a>b$ (since
the intervals $I_{1},I_{2},\ldots,I_{k}$ are arranged from left to right, so
that any element of $I_{i_{1}}$ would be larger than any element of $I_{i_{2}}$),
which would contradict $a<b$. Hence, $i_{1}\leq i_{2}$, so that $i_{1}<i_{2}$
(since $i_{1}\neq i_{2}$). Furthermore, from $a\in I_{i_{1}}$, we obtain
$w\tup{a}   \in w\left(  I_{i_{1}}\right)$, so that $w\left(
I_{i_{1}}\right)  \cap\set{w\tup{a}}  = \set{w\tup{a}}$.
Likewise, $w\left(  I_{i_{2}}\right)  \cap \set{w\tup{b}}
= \set{w\tup{b}}$.

Now, recall that $\Omal \tup{w} P_{\id}$ is
the set composition whose blocks are the nonempty intersections of the form
$w\left(  I_{j}\right)  \cap \set{p}$ for all $j\in
\interval{k}$ and $p\in\interval{n}$, in the order of lexicographically
increasing $\left(  j,p\right)$. Hence, the intersection
$w\left(  I_{i_{1} }\right)  \cap\set{w\tup{a}}$ appears
before $w\left(  I_{i_{2}}\right)  \cap\set{w\tup{b}}$ in
this set composition (since $i_{1}<i_{2}$ and thus
$\tup{i_1, w\tup{a}} < \tup{i_2, w\tup{b}}$
in the lexicographic order). In other words, the
block $\set{w\tup{a}}$ appears before $\set{w\tup{b}}$
in this set composition. This proves
\eqref{pf.eq.lem.rho.monoid.rhobs.4.pf.goal} in Case 2.

We have now proved \eqref{pf.eq.lem.rho.monoid.rhobs.4.pf.goal}
in both cases;
thus, \eqref{pf.eq.lem.rho.monoid.rhobs.4.pf.goal} always holds.
As we explained above, this concludes the proof of
\eqref{pf.eq.lem.rho.monoid.rhobs.4}.]

Now, \eqref{pf.eq.lem.rho.monoid.rhobs.2} becomes
\[
bP_{\id}
= \sum\limits_{\substack{F\in\CF;\\
\type F=\alpha}}FP_{\id}
= \sum\limits_{\substack{w\in S_n;\\
\Des w\subseteq \gaps^{-1}\tup{\alpha}  }}
\Omal\tup{w}  P_{\id}
\]
(here, we have substituted $\Omal \tup{w}$ for $F$ in the
sum, since the map $\Omal$ is a bijection). Hence,
\[
bP_{\id}=\sum\limits_{\substack{w\in S_n;\\
\Des w\subseteq\gaps^{-1} \tup{\alpha}  }}
\underbrace{\Omal \tup{w}
P_{\id}}_{\substack{=P_{w}\\\text{(by
\eqref{pf.eq.lem.rho.monoid.rhobs.4})}}}
= \sum\limits_{\substack{w\in S_n;\\
\Des w\subseteq\gaps^{-1} \tup{\alpha} }}P_{w}.
\]
Comparing this with \eqref{pf.eq.lem.rho.monoid.rhobs.1},
we obtain $bP_{\id} = \rho^{-1}\tup{b}  P_{\id}$.
This proves that \eqref{eq.lem.rho.monoid.rhobs} is satisfied.
\end{enumerate}

Thus, all conditions of Lemma \ref{lem.rho.monoid} are satisfied.
Hence, Lemma \ref{lem.rho.monoid} shows that $\rho^{-1}$ (viewed as a map
$\kfsn\to\ksn$, where both $\kfsn$
and $\ksn$ are regarded as monoids with respect to
multiplication) is a monoid anti-morphism.
Hence, $\rho^{-1}$ is a
$\kk$-algebra anti-morphism (since $\rho^{-1}$ is $\kk$-linear).
Thus, its image $\operatorname{Im}\rho^{-1}$ is a $\kk$-subalgebra of
$\ksn$ (since the image of a $\kk$-algebra
anti-morphism is always a $\kk$-subalgebra of the target). Since
$\operatorname{Im}\rho^{-1}=\dsn$ (because $\rho^{-1}$ was
defined as the inverse of the map $\rho:\dsn\to\kfsn$),
we thus have shown that
$\dsn$ is a $\kk$-subalgebra of $\ksn$.
Furthermore, the map $\rho$ is the inverse of the
$\kk$-algebra anti-morphism $\rho^{-1}$, and thus is a
$\kk$-algebra anti-isomorphism itself. This completes the proof of Theorem
\ref{dsn_kfsn_morph}.
\end{proof}

\subsection{\label{subsec.omitted_proofs.perm-fix-basis}The fixed points of a permutation module}

Let us prove the following general fact (\cite[\S 3.3.1, \textquotedblleft Invariants of Permutation Representations\textquotedblright]{Lorenz18}) about permutation modules (which we have used in the proof of Proposition~\ref{kfsn_prop2}):

\begin{prop}
\label{prop.perm-fix-basis}
Let $G$ be a group, and let $X$ be a finite left $G$-set.
Consider the invariant subspace\footnote{to be more precise: the invariant $\kk$-submodule} $\tup{\kk X}^G := \set{\xx \in \kk X : w\xx = \xx \text{ for all } w \in G}$ of the permutation module $\kk X$.

For each orbit $\mathcal{O}$ of the $G$-action on $X$, let $\sum_{x \in \mathcal{O}} x$ be the sum of all elements of this orbit in $\kk X$. This sum is called an \defin{orbit sum}.

Then, the invariant subspace $\tup{\kk X}^G$ (as a $\kk$-module) has a basis consisting of the orbit sums (i.e., of the sums $\sum_{x \in \mathcal{O}} x$ for each orbit $\mathcal{O}$ of the $G$-action on $X$).
\end{prop}

\begin{proof}
This fact is proved by combining three simple observations:
	\begin{enumerate}
	\item Each orbit sum $\sum_{x \in \mathcal{O}} x$ belongs to $\tup{\kk X}^G$, because the action of a $w \in G$ on this sum merely permutes its addends: $w \sum_{x \in \mathcal{O}} x = \sum_{x \in \mathcal{O}} wx = \sum_{x \in w\mathcal{O}} x = \sum_{x \in \mathcal{O}} x$ because $w\mathcal{O} = \mathcal{O}$.
	\item If an element $\xx = \sum_{x \in X} \lambda_x x \in \kk X$ (with coefficients $\lambda_x \in \kk$) belongs to $\tup{\kk X}^G$, then every two elements $x,y \in X$ that lie in the same $G$-orbit have the same coefficient (i.e., satisfy $\lambda_x = \lambda_y$), because the action of an appropriately chosen $w \in G$ puts $\lambda_x$ in the position of $\lambda_y$.
	Thus, given such an element $\xx \in \tup{\kk X}^G$, we can rename its coefficients $\lambda_x$ as $\lambda_{Gx}$, where $Gx$ (as usual) denotes the $G$-orbit containing $x$.
	Hence, we can rewrite $\xx = \sum_{x \in X} \lambda_x x$ as $\xx = \sum_{x \in X} \lambda_{Gx} x = \sum_{\mathcal{O} \text{ is a $G$-orbit}} \lambda_{\mathcal{O}} \tup{\sum_{x \in \mathcal{O}} x}$, which shows that $\xx$ is a linear combination of orbit sums.
	Hence, the orbit sums span $\tup{\kk X}^G$.
	\item The orbit sums $\sum_{x \in \mathcal{O}} x$ are $\kk$-linearly independent, since any $\kk$-linear combination $\sum_{\mathcal{O} \text{ is a $G$-orbit}} \lambda_{\mathcal{O}} \tup{\sum_{x \in \mathcal{O}} x}$ can be rewritten as $\sum_{x \in X} \lambda_{Gx} x$ and thus cannot be $0$ unless all the $\lambda_{Gx}$ are $0$, that is, unless all the $\lambda_{\mathcal{O}}$ are $0$.
	\end{enumerate}
Thus, the orbit sums form a basis of $\tup{\kk X}^G$.
\end{proof}

\subsection{\label{subsec.omitted_proofs.smf}Solomon's Mackey formula}

\defin{Solomon's Mackey formula} is an explicit combinatorial formula for
expressing a product of the form $\BB_{\beta}\BB_{\alpha}$ as a
sum of $\BB_{\gamma}$s. While we have not used it in this report, let
us nevertheless state and prove it for its cultural value. The proof we give
is merely a more detailed version of the proof given in \cite[Corollary
2.4]{Saliola}.

First we define some notations. An \defin{$\mathbb{N}$-matrix} shall mean a
$k\times\ell$-matrix $\left(  m_{i,j}\right)  _{i\in\interval{k}
,\ j\in\interval{\ell}}$ (for $k,\ell\in\mathbb{N}$) whose entries
$m_{i,j}$ are nonnegative integers. We let \defin{$\CM$} denote the set of all
$\mathbb{N}$-matrices (of all sizes). If $M=\left(  m_{i,j}\right)
_{i\in\interval{k}  ,\ j\in\interval{\ell}}\in\CM$ is an
$\mathbb{N}$-matrix, then:

\begin{itemize}
\item We define \defin{$\rowsums M$} to be the $k$-tuple $\left(
r_{1},r_{2},\ldots,r_{k}\right)  \in\mathbb{N}^{k}$, where $r_{i}:=\sum
_{j\in\interval{\ell}}m_{i,j}$ is the sum of all entries in the $i$-th
row of $M$. This is called the \defin{row sums vector} of $M$ (despite being,
technically, the entrywise sum of the columns of $M$).

\item We define \defin{$\colsums M$} to be the $\ell$-tuple $\left(
c_{1},c_{2},\ldots,c_{\ell}\right)  \in\mathbb{N}^{\ell}$, where $c_{j}
:=\sum_{i\in\interval{k}  }m_{i,j}$ is the sum of all entries in the
$j$-th column of $M$. This is called the \defin{column sums vector} of $M$
(despite being, technically, the entrywise sum of the rows of $M$).

\item We define \defin{$\readw M$} to be the $k\ell$-tuple $\left(
m_{1,1},m_{1,2},\ldots,m_{k,\ell}\right)  \in\mathbb{N}^{k\ell}$, whose
entries are all the entries $m_{i,j}$ of $M$ listed in the order of
lexicographically increasing $\left(  i,j\right)  $. This is called the
\defin{unreduced reading word} of $M$.

We can rewrite the definition of $\readw M$ as follows:
\begin{equation}
\readw M=\left(  \text{lex-list of all entries }m_{i,j}\text{ of
}M\right)  .
\label{eq.smf.readM=lex}
\end{equation}
Here and in the following, the word \textquotedblleft
\defin{lex-list}\textquotedblright\ means \textquotedblleft list in the order of
lexicographically increasing $\left(  i,j\right)  $\textquotedblright.

\item We define \defin{$\readb M$} to be the composition
obtained from the $k\ell$-tuple $\readw M$ upon removing all zero
entries. This is called the \defin{reading word} of $M$. Thus,
\begin{equation}
\readb M=\left(  \text{lex-list of all
\textbf{nonzero} entries }m_{i,j}\text{ of }M\right)  .
\label{eq.smf.readbM=lex}
\end{equation}

\end{itemize}

For example, if $M=\left(
\begin{array}[c]{ccc}
2 & 0 & 3\\
1 & 2 & 0
\end{array}
\right)  $, then $M\in\CM$ and $\rowsums M=\left(
5,3\right)  $ and $\colsums M=\left(  3,2,3\right)  $ and
$\readw M=\left(  2,0,3,1,2,0\right)  $ and $\readb M
=\left(  2,3,1,2\right)  $.

Now we are ready to state Solomon's Mackey formula in the Garsia--Remmel form
(\cite[Proposition 1.1]{GarReu89}, \cite[Corollary 2.4]{Saliola},
\cite[Proposition 4.3]{BleLau92}, \cite[Theorem 2]{Willigenburg98}):

\begin{thm}
[Solomon's Mackey formula in the Garsia--Remmel form]\label{thm.smf}Let
$\alpha$ and $\beta$ be two compositions of $n$. Then, in the descent algebra
$D\left(  S_{n}\right)  $, we have
\[
\BB_{\beta}\BB_{\alpha}
= \sum_{\substack{M\in\CM;\\ \rowsums M=\alpha;
\\\colsums M=\beta}} \BB_{\readb M}.
\]

\end{thm}

\begin{exmp}
\label{exam.smf.1}Assume that $n\geq2$. Then, Theorem \ref{thm.smf} (applied
to $\alpha=\left(  1,n-1\right)  $ and $\beta=\left(  n-1,1\right)  $) yields
\begin{equation}
\BB_{\left(  n-1,1\right)  }\BB_{\left(  1,n-1\right)  }
=\sum_{\substack{M\in\CM;\\\rowsums M=\left(
1,n-1\right)  ;\\\colsums M=\left(  n-1,1\right)  }}
\BB_{\readb M}.
\label{eq.exam.smf.1.1}
\end{equation}
But the $\mathbb{N}$-matrices $M\in\CM$ satisfying
$\rowsums M=\left(  1,n-1\right)  $ and $\colsums M
=\left(  n-1,1\right)  $ are precisely the $2\times2$-matrices
$\left(  \begin{array}[c]{cc}
a & b\\
c & d
\end{array}
\right)  \in\mathbb{N}^{2\times2}$ satisfying $a+b=1$ and $c+d=n-1$ and
$a+c=n-1$ and $b+d=1$. It is easy to see that there are only two such
matrices, namely $\left(
\begin{array}[c]{cc}
1 & 0\\
n-2 & 1
\end{array}
\right)  $ and $\left(
\begin{array}[c]{cc}
0 & 1\\
n-1 & 0
\end{array}
\right)  $. Their respective unreduced reading words $\readw M$
are $\left(  1,0,n-2,1\right)  $ and $\left(  0,1,n-1,0\right)  $, and their
respective reading words $\readb M$ are $
\begin{cases}
\left(  1,n-2,1\right)  , & \text{if }n\geq3;\\
\left(  1,1\right)  , & \text{if }n=2
\end{cases}
$ and $\left(  1,n-1\right)  $. Hence, (\ref{eq.exam.smf.1.1}) becomes
\[
\BB_{\left(  n-1,1\right)  }\BB_{\left(  1,n-1\right)  }=
\begin{cases}
\BB_{\left(  1,n-2,1\right)  }+\BB_{\left(  1,n-1\right)  }, &
\text{if }n\geq3;\\
2\BB_{\left(  1,1\right)  }, & \text{if }n=2.
\end{cases}
\]

\end{exmp}

We shall prove Theorem \ref{thm.smf} using Theorem \ref{dsn_kfsn_morph}, which
reduces the computation of $\BB_{\beta}\BB_{\alpha}$ in
$\dsn$ to the computation of
$\BBt_{\alpha}\BBt_{\beta}$ in $\kf$. To aid
our latter computation, we assign an $\mathbb{N}$-matrix to each pair of faces:

\begin{defn}
\label{def.face-int-mat}Let $F=\left(  F_{1},F_{2},\ldots,F_{k}\right)
\in\CF $ and $G=\left(  G_{1},G_{2},\ldots,G_{\ell}\right)
\in\CF $ be two faces of $\interval{n}$. Then, we define
$\mat\left(  F,G\right)  $ to be the $\mathbb{N}$-matrix
\[
\left(  \left\vert F_{i}\cap G_{j}\right\vert \right)  _{i\in\interval{k}
,\ j\in\interval{\ell}}\in\mathbb{N}^{k\times\ell}\subseteq\CM.
\]

\end{defn}

For instance, if $n=7$ and $F=\left(  \left\{  2,3\right\}  ,\left\{
1,4,6,7\right\}  ,\left\{  5\right\}  \right)  $ and $G=\left(  \left\{
2,4,5\right\}  ,\left\{  1,3\right\}  ,\left\{  6,7\right\}  \right)  $, then
\begin{align*}
\mat\left(  F,G\right)   &  =\left(
\begin{array}[c]{ccc}
\left\vert \left\{  2,3\right\}  \cap\left\{  2,4,5\right\}  \right\vert  &
\left\vert \left\{  2,3\right\}  \cap\left\{  1,3\right\}  \right\vert  &
\left\vert \left\{  2,3\right\}  \cap\left\{  6,7\right\}  \right\vert \\
\left\vert \left\{  1,4,6,7\right\}  \cap\left\{  2,4,5\right\}  \right\vert
& \left\vert \left\{  1,4,6,7\right\}  \cap\left\{  1,3\right\}  \right\vert
& \left\vert \left\{  1,4,6,7\right\}  \cap\left\{  6,7\right\}  \right\vert
\\
\left\vert \left\{  5\right\}  \cap\left\{  2,4,5\right\}  \right\vert  &
\left\vert \left\{  5\right\}  \cap\left\{  1,3\right\}  \right\vert  &
\left\vert \left\{  5\right\}  \cap\left\{  6,7\right\}  \right\vert
\end{array}
\right) \\
&  =\left(
\begin{array}[c]{ccc}
\left\vert \left\{  2\right\}  \right\vert  & \left\vert \left\{  3\right\}
\right\vert  & \left\vert \varnothing\right\vert \\
\left\vert \left\{  4\right\}  \right\vert  & \left\vert \left\{  1\right\}
\right\vert  & \left\vert \left\{  6,7\right\}  \right\vert \\
\left\vert \left\{  5\right\}  \right\vert  & \left\vert \varnothing
\right\vert  & \left\vert \varnothing\right\vert
\end{array}
\right)  =\left(
\begin{array}[c]{ccc}
1 & 1 & 0\\
1 & 1 & 2\\
1 & 0 & 0
\end{array}
\right)  .
\end{align*}

We can easily connect this matrix to the product $FG$ in the monoid
$\CF $:

\begin{lem}
\label{lem.face-int-mat.type}Let $F,G\in\CF $ be two faces of
$\interval{n}$. Then,
\begin{align}
\rowsums \left(  \mat\left(  F,G\right)  \right)
&  =\type F;\label{eq.lem.face-int-mat.type.F}\\
\colsums \left(  \mat\left(  F,G\right)  \right)
&  =\type G;\label{eq.lem.face-int-mat.type.G}\\
\readb \left(  \mat\left(  F,G\right)
\right)   &  =\type \left(  FG\right)  .
\label{eq.lem.face-int-mat.type.FG}
\end{align}

\end{lem}

\begin{proof}
Write the faces $F$ and $G$ as $F=\left(  F_{1},F_{2},\ldots,F_{k}\right)  $
and $G=\left(  G_{1},G_{2},\ldots,G_{\ell}\right)  $. By Definition
\ref{def.face-int-mat}, we then have
\begin{equation}
\mat\left(  F,G\right)  =\left(  \left\vert F_{i}\cap
G_{j}\right\vert \right)  _{i\in\interval{k}  ,\ j\in\interval{\ell}
}.
\label{pf.lem.face-int-mat.type.mat=}
\end{equation}

\begin{enumerate}
\item Let us prove (\ref{eq.lem.face-int-mat.type.F}). Recall that the entries
of the matrix $\mat\left(  F,G\right)  $ are $\left\vert
F_{i}\cap G_{j}\right\vert $ for all $i\in\interval{k}  $ and $j\in
\interval{\ell}$ (by (\ref{pf.lem.face-int-mat.type.mat=})). Hence,
$\rowsums \left(  \mat\left(  F,G\right)
\right)  $ is defined to be the $k$-tuple $\left(  r_{1},r_{2},\ldots
,r_{k}\right)  \in\mathbb{N}^{k}$, where $r_{i}:=\sum_{j\in
\interval{\ell}}\left\vert F_{i}\cap G_{j}\right\vert $ is the sum of all
entries in the $i$-th row of $\mat\left(  F,G\right)  $.
Consider these numbers $r_{i}$.

Fix $i\in\interval{k}  $. Recall that $G=\left(  G_{1},G_{2},
\ldots,G_{\ell}\right)  $ is a face of $\interval{n}$. Hence, the sets
$G_{1},G_{2},\ldots,G_{\ell}$ are disjoint and have union $\bigcup
\limits_{j=1}^{\ell}G_{j}=\interval{n}$. Thus, the intersections
$F_{i}\cap G_{1},\ F_{i}\cap G_{2},\ \ldots,\ F_{i}\cap G_{\ell}$ are also
disjoint and have union
\[
\bigcup\limits_{j=1}^{\ell}\left(  F_{i}\cap G_{j}\right)  =F_{i}
\cap\underbrace{\left(  \bigcup\limits_{j=1}^{\ell}G_{j}\right)  }_{=
\interval{n}}=F_{i}\cap\interval{n}=F_{i}
\]
(since $F_{i}\subseteq\interval{n}$). Hence,
\[
F_{i}=\bigcup\limits_{j=1}^{\ell}\left(  F_{i}\cap G_{j}\right)  ,
\]
so that
\begin{align*}
\left\vert F_{i}\right\vert  &  =\left\vert \bigcup\limits_{j=1}^{\ell}\left(
F_{i}\cap G_{j}\right)  \right\vert =\sum_{j=1}^{\ell}\left\vert F_{i}\cap
G_{j}\right\vert \ \ \ \ \ \ \ \ \ \ \left(
\begin{array}[c]{c}
\text{by the sum rule, since the}\\
\text{sets }F_{i}\cap G_{1},\ F_{i}\cap G_{2},\ \ldots,\ F_{i}\cap G_{\ell}\\
\text{are disjoint}
\end{array}
\right) \\
&  =\sum_{j\in\interval{\ell}}\left\vert F_{i}\cap G_{j}\right\vert
=r_{i}\ \ \ \ \ \ \ \ \ \ \left(  \text{by the definition of }r_{i}\right)  .
\end{align*}
Forget that we fixed $i$. We thus have shown that $\left\vert F_{i}\right\vert
=r_{i}$ for each $i\in\interval{k}  $. Now, from
$F=\left(  F_{1},F_{2},\ldots,F_{k}\right)  $, we obtain
\begin{align*}
\type F  &  =\type \left(  F_{1},F_{2},
\ldots,F_{k}\right) \\
&  =\left(  \left\vert F_{1}\right\vert ,\left\vert F_{2}\right\vert
,\dots,\left\vert F_{k}\right\vert \right)  \ \ \ \ \ \ \ \ \ \ \left(
\text{by the definition of type}\right) \\
&  =\left(  r_{1},r_{2},\ldots,r_{k}\right)  \ \ \ \ \ \ \ \ \ \ \left(
\text{since }\left\vert F_{i}\right\vert =r_{i}\text{ for each }i\in
\interval{k}\right) \\
&  =\rowsums \left(  \mat\left(  F,G\right)
\right)
\end{align*}
(since $\rowsums \left(  \mat\left(  F,G\right)
\right)  $ is defined to be $\left(  r_{1},r_{2},\ldots,r_{k}\right)  $). This
proves (\ref{eq.lem.face-int-mat.type.F}).

\item A similar argument proves (\ref{eq.lem.face-int-mat.type.G}).
(Essentially, this is just our above proof of
(\ref{eq.lem.face-int-mat.type.F}), but with the roles of $i$ and $j$
interchanged; in particular, instead of using $\bigcup\limits_{j=1}^{\ell
}\left(  F_{i}\cap G_{j}\right)  =F_{i}$ for each $i\in\interval{k}  $, we
now need to use $\bigcup\limits_{i=1}^{k}\left(  F_{i}\cap G_{j}\right)
=G_{j}$ for each $j\in\interval{\ell}$.)

\item It remains to prove (\ref{eq.lem.face-int-mat.type.FG}).

Recall that the entries of the matrix $\mat\left(  F,G\right)
$ are the numbers $\left\vert F_{i}\cap G_{j}\right\vert $ for all
$i\in\interval{k}  $ and $j\in\interval{\ell}$ (by
(\ref{pf.lem.face-int-mat.type.mat=})). Hence, (\ref{eq.smf.readbM=lex})
(applied to $M=\mat\left(  F,G\right)  $ and $m_{i,j}
=\left\vert F_{i}\cap G_{j}\right\vert $) yields
\begin{align}
\readb \left(  \mat\left(  F,G\right)
\right)   &  =\left(  \text{lex-list of all \textbf{nonzero} entries
}\left\vert F_{i}\cap G_{j}\right\vert \text{ of }\mat\left(
F,G\right)  \right) \nonumber\\
&  =\left(  \text{lex-list of the \textbf{nonzero} sizes }\left\vert F_{i}\cap
G_{j}\right\vert \right)  .
\label{pf.lem.face-int-mat.type.readmat}
\end{align}

On the other hand, recall that $F=\left(  F_{1},F_{2},\ldots,F_{k}\right)  $
and $G=\left(  G_{1},G_{2},\ldots,G_{\ell}\right)  $. Hence, Definition
\ref{def.facemon} \textbf{(b)} shows that
\[
FG=\left(  F_{1}\cap G_{1},\dots,F_{1}\cap G_{\ell},\dots,F_{k}\cap G_{\ell
}\right)  ^{\red},
\]
where the list on the right hand side contains all the pairwise intersections
$F_{i}\cap G_{j}$ in the order of lexicographically increasing $\left(
i,j\right)  $. In other words,
\[
FG=\left(  \text{lex-list of all intersections }F_{i}\cap G_{j}\right)
^{\red}.
\]
Recalling that the reduction operator ($^{\operatorname*{red}}$) removes all
empty sets from the list (see Definition \ref{def.facemon} \textbf{(a)}), we
can rewrite this as
\begin{align*}
FG  &  =\left(  \text{lex-list of all intersections }F_{i}\cap G_{j}\text{
except for the empty ones}\right) \\
&  =\left(  \text{lex-list of all nonempty intersections }F_{i}\cap
G_{j}\right) \\
&  =\left(  \text{lex-list of all intersections }F_{i}\cap G_{j}\text{ such
that }\left\vert F_{i}\cap G_{j}\right\vert \neq0\right)
\end{align*}
(since an intersection $F_{i}\cap G_{j}$ is nonempty if and only if
$\left\vert F_{i}\cap G_{j}\right\vert \neq0$). Therefore,
\begin{align*}
\type \left(  FG\right)   &  =\type \left(
\text{lex-list of all intersections }F_{i}\cap G_{j}\text{ such that
}\left\vert F_{i}\cap G_{j}\right\vert \neq0\right) \\
&  =\left(  \text{lex-list of the \textbf{sizes} }\left\vert F_{i}\cap
G_{j}\right\vert \text{ of all intersections }F_{i}\cap G_{j}\right. \\
&  \ \ \ \ \ \ \ \ \ \ \left.  \text{such that }\left\vert F_{i}\cap
G_{j}\right\vert \neq0\right)
\end{align*}
(since the type of a set composition is the list of the sizes of its blocks).
We can rewrite this as
\begin{align*}
\type \left(  FG\right)   &  =\left(  \text{lex-list of the
sizes }\left\vert F_{i}\cap G_{j}\right\vert \text{ such that }\left\vert
F_{i}\cap G_{j}\right\vert \neq0\right) \\
&  =\left(  \text{lex-list of the \textbf{nonzero} sizes }\left\vert F_{i}\cap
G_{j}\right\vert \right)  .
\end{align*}
Comparing this with (\ref{pf.lem.face-int-mat.type.readmat}), we find
$\readb \left(  \mat\left(  F,G\right)
\right)  =\type \left(  FG\right)  $. Thus,
(\ref{eq.lem.face-int-mat.type.FG}) is proved, and the proof of Lemma
\ref{lem.face-int-mat.type} is complete. \qedhere

\end{enumerate}
\end{proof}

\begin{lem}
\label{lem.face-int-bij}Let $\CR$ be the set of all pairs $\left(
H,M\right)  \in\CF \times\CM$ satisfying
$\rowsums M\models n$ and $\colsums M\models n$
and $\type H=\readb M$. Then, the map
\begin{align}
\Phi:\CF \times\CF   &  \rightarrow\CR,\nonumber\\
\left(  F,G\right)   &  \mapsto\left(  FG,\ \mat\left(
F,G\right)  \right)
\label{eq.lem.face-int-bij.bij}
\end{align}
is well-defined and is a bijection.
\end{lem}

\begin{proof}
We must show that the map $\Phi$ is well-defined, injective and surjective.
\medskip

\textit{Well-definedness:} In order to prove that $\Phi$ is well-defined, we
must show that every $\left(  F,G\right)  \in\CF \times\CF $
satisfies $\left(  FG,\ \mat\left(  F,G\right)  \right)
\in\CR$.

So let $\left(  F,G\right)  \in\CF \times\CF $. That is,
$F,G\in\CF $. In other words, $F$ and $G$ are two faces of
$\interval{n}$. Thus, Lemma \ref{lem.face-int-mat.type} shows that
\begin{align*}
\rowsums \left(  \mat\left(  F,G\right)  \right)
&  =\type F;\\
\colsums \left(  \mat\left(  F,G\right)  \right)
&  =\type G;\\
\readb \left(  \mat\left(  F,G\right)
\right)   &  =\type \left(  FG\right)  .
\end{align*}
Hence,
\begin{align*}
\rowsums \left(  \mat\left(  F,G\right)  \right)
&  =\type F\models n\ \ \ \ \ \ \ \ \ \ \left(  \text{since
}F\text{ is a face of }\interval{n}\right)
\ \ \ \ \ \ \ \ \ \ \text{and}\\
\colsums \left(  \mat\left(  F,G\right)  \right)
&  =\type G\models n\ \ \ \ \ \ \ \ \ \ \left(  \text{since
}G\text{ is a face of }\interval{n}\right)
\ \ \ \ \ \ \ \ \ \ \text{and}\\
\type \left(  FG\right)   &  =\readb
\left(  \mat\left(  F,G\right)  \right)  .
\end{align*}
Also, $\left(  FG,\ \mat\left(  F,G\right)  \right)
\in\CF \times\CM$ (since $FG\in\CF $ and
$\mat\left(  F,G\right)  \in\CM$). Therefore, $\left(
FG,\ \mat\left(  F,G\right)  \right)  $ is a pair $\left(
H,M\right)  \in\CF \times\CM$ satisfying
$\rowsums M\models n$ and $\colsums M\models n$
and $\type H=\readb M$. In other words,
$\left(  FG,\ \mat\left(  F,G\right)  \right)  \in\CR$
(by the definition of $\CR$). This completes our proof that $\Phi$ is
well-defined. \medskip

\textit{Injectivity:} Let us now show that $\Phi$ is injective. For this
purpose, we must prove that any pair $\left(  F,G\right)  \in\CF
\times\CF $ can be uniquely reconstructed from its image $\Phi\left(
F,G\right)  =\left(  FG,\ \mat\left(  F,G\right)  \right)  $.

Fix a pair $\left(  F,G\right)  \in\CF \times\CF $. We shall
refer to an object as \textquotedblleft known\textquotedblright\ if it can be
uniquely reconstructed from $\Phi\left(  F,G\right)  =\left(
FG,\ \mat\left(  F,G\right)  \right)  $. Thus, our goal is to
show that $\left(  F,G\right)  $ is known. Obviously, $FG$ and
$\mat\left(  F,G\right)  $ are known; furthermore, any object
that can be uniquely reconstructed from known objects is itself known.

Write the faces $F$ and $G$ as $F=\left(  F_{1},F_{2},\ldots,F_{k}\right)  $
and $G=\left(  G_{1},G_{2},\ldots,G_{\ell}\right)  $. By Definition
\ref{def.face-int-mat}, we then have
\begin{equation}
\mat\left(  F,G\right)  =\left(  \left\vert F_{i}\cap
G_{j}\right\vert \right)  _{i\in\interval{k}  ,\ j\in\interval{\ell}
}.
\label{pf.lem.face-int-bij.inj.mat}
\end{equation}
Thus, $\mat\left(  F,G\right)  $ is a $k\times\ell$-matrix.
Hence, the numbers $k$ and $\ell$ are known (since they are the dimensions of
the known matrix $\mat\left(  F,G\right)  $). Moreover, the
numbers $\left\vert F_{i}\cap G_{j}\right\vert $ for all $\left(  i,j\right)
\in\interval{k}  \times\interval{\ell}$ are known (since they are
the entries of the known matrix $\mat\left(  F,G\right)  $).

Let $\left(  \left(  i_{1},j_{1}\right)  ,\ \left(  i_{2},j_{2}\right)
,\ \ldots,\ \left(  i_{h},j_{h}\right)  \right)  $ be the lex-list of all the
pairs $\left(  i,j\right)  \in\interval{k}  \times\interval{\ell}$
satisfying $\left\vert F_{i}\cap G_{j}\right\vert \neq0$. This lex-list is
known (since all the numbers $\left\vert F_{i}\cap G_{j}\right\vert $ are
known). Clearly, for any pair $\left(  i,j\right)  \in\interval{k}
\times\interval{\ell}$, we have
\begin{equation}
\left(
\begin{array}[c]{c}
\left(  i,j\right)  \in\left\{  \left(  i_{1},j_{1}\right)  ,\ \left(
i_{2},j_{2}\right)  ,\ \ldots,\ \left(  i_{h},j_{h}\right)  \right\} \\
\text{if and only if }\left\vert F_{i}\cap G_{j}\right\vert \neq0
\end{array}
\right)
\label{pf.lem.face-int-bij.inj.ijinlist}
\end{equation}
(since $\left(  \left(  i_{1},j_{1}\right)  ,\ \left(  i_{2},j_{2}\right)
,\ \ldots,\ \left(  i_{h},j_{h}\right)  \right)  $ is a list of all the pairs
$\left(  i,j\right)  \in\interval{k}  \times\interval{\ell}$
satisfying $\left\vert F_{i}\cap G_{j}\right\vert \neq0$).

As in the proof of Lemma \ref{lem.face-int-mat.type}, we can see that
\begin{equation}
FG=\left(  \text{lex-list of all intersections }F_{i}\cap G_{j}\text{ such
that }\left\vert F_{i}\cap G_{j}\right\vert \neq0\right)  .
\label{pf.lem.face-int-bij.inj.FG=lex}
\end{equation}
In other words,
\begin{equation}
FG=\left(  F_{i_{1}}\cap G_{j_{1}},\ F_{i_{2}}\cap G_{j_{2}},\ \ldots
,\ F_{i_{h}}\cap G_{j_{h}}\right)
\label{pf.lem.face-int-bij.inj.FG=list}
\end{equation}
(since $\left(  \left(  i_{1},j_{1}\right)  ,\ \left(  i_{2},j_{2}\right)
,\ \ldots,\ \left(  i_{h},j_{h}\right)  \right)  $ is the lex-list of all the
pairs $\left(  i,j\right)  \in\interval{k}  \times\interval{\ell}$
satisfying $\left\vert F_{i}\cap G_{j}\right\vert \neq0$). Thus, for each
$p\in\interval{h}$, we have
\begin{equation}
\left(  \text{the }p\text{-th block of }FG\right)  =F_{i_{p}}\cap G_{j_{p}}.
\label{pf.lem.face-int-bij.inj.FGp}
\end{equation}

Now, for each $\left(  i,j\right)  \in\interval{k}  \times
\interval{\ell}$, the intersection $F_{i}\cap G_{j}$ is known, because we can
reconstruct $F_{i}\cap G_{j}$ from the known face $FG$ and the known list
$\left(  \left(  i_{1},j_{1}\right)  ,\ \left(  i_{2},j_{2}\right)
,\ \ldots,\ \left(  i_{h},j_{h}\right)  \right)  $ as follows:

\begin{itemize}
\item If $\left(  i,j\right)  \in\left\{  \left(  i_{1},j_{1}\right)
,\ \left(  i_{2},j_{2}\right)  ,\ \ldots,\ \left(  i_{h},j_{h}\right)
\right\}  $, then $\left(  i,j\right)  =\left(  i_{p},j_{p}\right)  $ for some
$p\in\interval{h}$, and therefore this $p$ satisfies
\[
F_{i}\cap G_{j}=F_{i_{p}}\cap G_{j_{p}}=\left(  \text{the }p\text{-th block of
}FG\right)  \ \ \ \ \ \ \ \ \ \ \left(  \text{by
(\ref{pf.lem.face-int-bij.inj.FGp})}\right)  ,
\]
so that $F_{i}\cap G_{j}$ is known (since $FG$ and $p$ are known).

\item Otherwise, we don't have $\left(  i,j\right)  \in\left\{  \left(
i_{1},j_{1}\right)  ,\ \left(  i_{2},j_{2}\right)  ,\ \ldots,\ \left(
i_{h},j_{h}\right)  \right\}  $; therefore, we don't have $\left\vert
F_{i}\cap G_{j}\right\vert \neq0$ either (by
(\ref{pf.lem.face-int-bij.inj.ijinlist})); hence, we have $\left\vert
F_{i}\cap G_{j}\right\vert =0$, and therefore $F_{i}\cap G_{j}=\varnothing$ is known.
\end{itemize}

Thus, we have shown that the intersection $F_{i}\cap G_{j}$ is known for each
$\left(  i,j\right)  \in\interval{k}  \times\interval{\ell}$.

However, for each $i\in\interval{k}  $, we have $F_{i}=\bigcup
\limits_{j=1}^{\ell}\left(  F_{i}\cap G_{j}\right)  $ (as we saw in the proof
of Lemma \ref{lem.face-int-mat.type}), and therefore $F_{i}$ is known (since
each intersection $F_{i}\cap G_{j}$ is known). Hence, $F$ is known (since
$F=\left(  F_{1},F_{2},\ldots,F_{k}\right)  $). Similarly, $G$ is known (since
we can show that $G_{j}=\bigcup\limits_{i=1}^{k}\left(  F_{i}\cap
G_{j}\right)  $ for each $j\in\interval{\ell}$). Therefore, the pair
$\left(  F,G\right)  $ is known. This completes the proof of the injectivity
of $\Phi$. \medskip

\textit{Surjectivity:} Finally, we must show that $\Phi$ is surjective. In
other words, we must show that each $\left(  H,M\right)  \in\CR$ lies
in the image of $\Phi$.

So let $\left(  H,M\right)  \in\CR$ be arbitrary. Thus, by the
definition of $\CR$, we have $\left(  H,M\right)  \in\CF
\times\CM$ and $\rowsums M\models n$ and
$\colsums M\models n$ and $\type H=\readb M$.

Write the matrix $M\in\CM$ in the form $M=\left(  m_{i,j}\right)
_{i\in\interval{k}  ,\ j\in\interval{\ell}}$ for some $k,\ell
\in\mathbb{N}$.

Let $\left(  \left(  i_{1},j_{1}\right)  ,\ \left(  i_{2},j_{2}\right)
,\ \ldots,\ \left(  i_{h},j_{h}\right)  \right)  $ be the lex-list of all the
pairs $\left(  i,j\right)  \in\interval{k}  \times\interval{\ell}$
satisfying $m_{i,j}\neq0$. Then, the pairs $\left(  i_{1},j_{1}\right)
,\ \left(  i_{2},j_{2}\right)  ,\ \ldots,\ \left(  i_{h},j_{h}\right)  $ are
distinct (since a lex-list cannot contain equal entries). Furthermore, any
pair $\left(  i,j\right)  \in\interval{k}  \times\interval{\ell}$
that satisfies $m_{i,j}\neq0$ must be contained in the list $\left(  \left(
i_{1},j_{1}\right)  ,\ \left(  i_{2},j_{2}\right)  ,\ \ldots,\ \left(
i_{h},j_{h}\right)  \right)  $ (since this list was defined as the lex-list of
all such pairs $\left(  i,j\right)  $); in other words, it must satisfy
\begin{equation}
\left(  i,j\right)  =\left(  i_{r},j_{r}\right)  \ \ \ \ \ \ \ \ \ \ \text{for
some }r\in\interval{h} .
\label{pf.lem.face-int-bij.sur.irjr}
\end{equation}

Recall again that $\left(  \left(  i_{1},j_{1}\right)  ,\ \left(  i_{2}
,j_{2}\right)  ,\ \ldots,\ \left(  i_{h},j_{h}\right)  \right)  $ is the
lex-list of all the pairs $\left(  i,j\right)  \in\interval{k}
\times\interval{\ell}$ satisfying $m_{i,j}\neq0$. In other words, the
pairs $\left(  i,j\right)  \in\interval{k}  \times\interval{\ell}$
satisfying $m_{i,j}\neq0$ are $\left(  i_{1},j_{1}\right)  ,\ \left(
i_{2},j_{2}\right)  ,\ \ldots,\ \left(  i_{h},j_{h}\right)  $, listed here in
the order of lexicographically increasing $\left(  i,j\right)  $. Hence, the
nonzero entries $m_{i,j}$ of $M$ are precisely $m_{i_{1},j_{1}},m_{i_{2}
,j_{2}},\ldots,m_{i_{h},j_{h}}$, listed here in the order of lexicographically
increasing $\left(  i,j\right)  $. But (\ref{eq.smf.readbM=lex}) becomes
\begin{align*}
\readb M &  =\left(  \text{lex-list of all
\textbf{nonzero} entries }m_{i,j}\text{ of }M\right)  \\
&  =\left(  m_{i_{1},j_{1}},m_{i_{2},j_{2}},\ldots,m_{i_{h},j_{h}}\right)
\end{align*}
(since the nonzero entries $m_{i,j}$ of $M$ are precisely $m_{i_{1},j_{1}
},m_{i_{2},j_{2}},\ldots,m_{i_{h},j_{h}}$, listed here in the order of
lexicographically increasing $\left(  i,j\right)  $).

We have
\begin{equation}
\type H=\readb M=\left(  m_{i_{1},j_{1}
},m_{i_{2},j_{2}},\ldots,m_{i_{h},j_{h}}\right)
.
\label{pf.lem.face-int-bij.sur.typeH}
\end{equation}
Hence, $\type H$ is an $h$-tuple. Hence, the face $H$ is an
$h$-tuple as well (since $\type H$ has the same length as $H$).
Therefore, we write this $h$-tuple $H$ as $H=\left(  H_{1},H_{2},\ldots
,H_{h}\right)  $. Then, by the definition of a type, we have
\[
\type H=\type \left(  H_{1},H_{2},\ldots
,H_{h}\right)  =\left(  \left\vert H_{1}\right\vert ,\left\vert H_{2}
\right\vert ,\ldots,\left\vert H_{h}\right\vert \right)  ,
\]
so that
\[
\left(  \left\vert H_{1}\right\vert ,\left\vert H_{2}\right\vert
,\ldots,\left\vert H_{h}\right\vert \right)  =\type H=\left(
m_{i_{1},j_{1}},m_{i_{2},j_{2}},\ldots,m_{i_{h},j_{h}}\right)  .
\]
In other words,
\begin{equation}
\left\vert H_{p}\right\vert =m_{i_{p},j_{p}}\ \ \ \ \ \ \ \ \ \ \text{for each
}p\in \interval{h} .
\label{pf.lem.face-int-bij.sur.sizeH}
\end{equation}
Moreover, the sets $H_{1},H_{2},\ldots,H_{h}$ are the blocks of the face $H$
(since $H=\left(  H_{1},H_{2},\ldots,H_{h}\right)  $); hence, they are
disjoint and satisfy $\interval{n}=\bigcup_{p=1}^{h}H_{p}=H_{1}\cup
H_{2}\cup\cdots\cup H_{h}$.

Now, for each $i\in\interval{k}  $, define a subset $F_{i}$ of
$\interval{n}$ by
\[
F_{i}=\bigcup\limits_{\substack{p\in\interval{h} ;\\i_{p}=i}}H_{p}.
\]
Furthermore, for each $j\in\interval{\ell}$, define a subset $G_{j}$ of
$\interval{n}$ by
\[
G_{j}=\bigcup\limits_{\substack{p\in\interval{h} ;\\j_{p}=j}}H_{p}.
\]
We shall now prove that $\left(  F_{1},F_{2},\ldots,F_{k}\right)  $ and
$\left(  G_{1},G_{2},\ldots,G_{\ell}\right)  $ are faces of $\interval{n}$.
Indeed:

\begin{enumerate}
\item If $i$ and $j$ are two distinct elements of $\interval{k}  $, then
the two sets $F_{i}$ and $F_{j}$ are disjoint.

(\textit{Proof:} Let $i$ and $j$ be two distinct elements of
$\interval{k}$. The two sets $F_{i}$ and $F_{j}$ are defined as $F_{i}
=\bigcup\limits_{\substack{p\in\interval{h} ;\\i_{p}=i}}H_{p}$ and
$F_{j}=\bigcup\limits_{\substack{p\in\interval{h} ;\\i_{p}=j}}H_{p}$.
Thus, if they had an element $x$ in common, then $x$ would belong to $H_{p}$
for some number $p\in\interval{h} $ satisfying $i_{p}=i$ (since $x\in
F_{i}=\bigcup\limits_{\substack{p\in\interval{h} ;\\i_{p}=i}}H_{p}$) and
would also belong to $H_{q}$ for some number $q\in\interval{h} $
satisfying $i_{q}=j$ (since $x\in F_{j}=\bigcup\limits_{\substack{p\in
\interval{h};\\i_{p}=j}}H_{p}=\bigcup\limits_{\substack{q\in\interval{h}
;\\i_{q}=j}}H_{q}$). These two numbers $p$ and $q$ would then have to be equal
(because if they were not, then $H_{p}$ and $H_{q}$ would be disjoint (since
$H_{1},H_{2},\ldots,H_{h}$ are disjoint), contradicting the fact that $x$
belongs to both $H_{p}$ and $H_{q}$); but then $i_{p}=i$ would rewrite as
$i_{q}=i$, so that $i=i_{q}=j$, contradicting the distinctness of $i$ and $j$.
Hence, the sets $F_{i}$ and $F_{j}$ cannot have an element in common. In other
words, they are disjoint.)

Thus we have shown that $F_{i}$ and $F_{j}$ are disjoint whenever $i$ and $j$
are two distinct elements of $\interval{k}  $. In other words,
\begin{equation}
\text{the sets }F_{1},F_{2},\ldots,F_{k}\text{ are disjoint.}
\label{pf.lem.face-int-bij.sur.Fdj}
\end{equation}
Similarly,
\begin{equation}
\text{the sets }G_{1},G_{2},\ldots,G_{\ell}\text{ are disjoint.}
\label{pf.lem.face-int-bij.sur.Gdj}
\end{equation}

\item For any $i\in\interval{k}  $ and $j\in\interval{\ell}$, we
have
\begin{equation}
F_{i}\cap G_{j}=
\begin{cases}
H_{r}, & \text{if }\left(  i,j\right)  =\left(  i_{r},j_{r}\right)  \text{ for
some }r\in\interval{h} ;\\
\varnothing, & \text{if not.}
\end{cases}
\label{pf.lem.face-int-bij.sur.FiGj}
\end{equation}

(\textit{Proof:} Let $i\in\interval{k}  $ and $j\in\interval{\ell}$.
The definitions of $F_{i}$ and $G_{j}$ yield $F_{i}=\bigcup
\limits_{\substack{p\in\interval{h} ;\\i_{p}=i}}H_{p}$ and $G_{j}
=\bigcup\limits_{\substack{p\in\interval{h} ;\\j_{p}=j}}H_{p}
=\bigcup\limits_{\substack{q\in\interval{h} ;\\j_{q}=j}}H_{q}$.

We are in one of the following two cases:

\textit{Case 1:} We have $\left(  i,j\right)  =\left(  i_{r},j_{r}\right)  $
for some $r\in\interval{h} $.

\textit{Case 2:} There exists no $r\in\interval{h} $ such that $\left(
i,j\right)  =\left(  i_{r},j_{r}\right)  $.

In order to prove (\ref{pf.lem.face-int-bij.sur.FiGj}), we must show that
$F_{i}\cap G_{j}=H_{r}$ in Case 1, and we must show that $F_{i}\cap
G_{j}=\varnothing$ in Case 2.

Let us consider Case 1. In this case, we have $\left(  i,j\right)  =\left(
i_{r},j_{r}\right)  $ for some $r\in\interval{h} $. This $r$ is unique
(since $\left(  i_{1},j_{1}\right)  ,\ \left(  i_{2},j_{2}\right)
,\ \ldots,\ \left(  i_{h},j_{h}\right)  $ are distinct). We must now show that
$F_{i}\cap G_{j}=H_{r}$.

Indeed, since $\left(  i,j\right)  =\left(  i_{r},j_{r}\right)  $, we have
$i=i_{r}$ and $j=j_{r}$. Thus, $i_{r}=i$, so that the set $H_{r}$ is contained
in the union $\bigcup\limits_{\substack{p\in\interval{h} ;\\i_{p}
=i}}H_{p}$. Thus, $H_{r}\subseteq\bigcup\limits_{\substack{p\in
\interval{h};\\i_{p}=i}}H_{p}=F_{i}$. Similarly, $H_{r}\subseteq G_{j}$ (since
$j_{r}=j$). Combining these two inclusions, we find $H_{r}\subseteq F_{i}\cap
G_{j}$.

Now, let us conversely show that $F_{i}\cap G_{j}\subseteq H_{r}$. Indeed, let
$x\in F_{i}\cap G_{j}$. Thus, $x\in F_{i}$ and $x\in G_{j}$. From $x\in
F_{i}=\bigcup\limits_{\substack{p\in\interval{h} ;\\i_{p}=i}}H_{p}$, it
follows that $x\in H_{p}$ for some $p\in\interval{h} $ satisfying
$i_{p}=i$. From $x\in G_{j}=\bigcup\limits_{\substack{q\in\interval{h}
;\\j_{q}=j}}H_{q}$, it follows that $x\in H_{q}$ for some $q\in\interval{h}$
satisfying $j_{q}=j$. Consider these $p$ and $q$. The sets $H_{p}$
and $H_{q}$ are not disjoint (since $x$ belongs to both); thus, $p=q$ (since
$H_{1},H_{2},\ldots,H_{h}$ are disjoint). Hence, $j_{p}=j_{q}=j$. Combining
$i_{p}=i$ with $j_{p}=j$, we obtain $\left(  i_{p},j_{p}\right)  =\left(
i,j\right)  =\left(  i_{r},j_{r}\right)  $, so that $p=r$ (since the pairs
$\left(  i_{1},j_{1}\right)  ,\ \left(  i_{2},j_{2}\right)  ,\ \ldots
,\ \left(  i_{h},j_{h}\right)  $ are distinct). Hence, we can rewrite $x\in
H_{p}$ as $x\in H_{r}$.

Now, forget that we fixed $x$. We thus have proved that $x\in H_{r}$ for each
$x\in F_{i}\cap G_{j}$. In other words, $F_{i}\cap G_{j}\subseteq H_{r}$.
Combining this with $H_{r}\subseteq F_{i}\cap G_{j}$, we obtain $F_{i}\cap
G_{j}=H_{r}$. This completes the proof of (\ref{pf.lem.face-int-bij.sur.FiGj})
in Case 1.

Now let us consider Case 2. In this case, there exists no $r\in
\interval{h}$ such that $\left(  i,j\right)  =\left(  i_{r},j_{r}\right)  $.

We must show that $F_{i}\cap G_{j}=\varnothing$. Indeed, let $x\in F_{i}\cap
G_{j}$. Thus, $x\in F_{i}$ and $x\in G_{j}$. From $x\in F_{i}=\bigcup
\limits_{\substack{p\in\interval{h} ;\\i_{p}=i}}H_{p}$, it follows that
$x\in H_{p}$ for some $p\in\interval{h} $ satisfying $i_{p}=i$. From
$x\in G_{j}=\bigcup\limits_{\substack{q\in\interval{h} ;\\j_{q}=j}}H_{q}
$, it follows that $x\in H_{q}$ for some $q\in\interval{h} $ satisfying
$j_{q}=j$. Consider these $p$ and $q$. The sets $H_{p}$ and $H_{q}$ are not
disjoint (since $x$ belongs to both); thus, $p=q$ (since $H_{1},H_{2}
,\ldots,H_{h}$ are disjoint). Hence, $j_{p}=j_{q}=j$. Combining $i_{p}=i$ with
$j_{p}=j$, we obtain $\left(  i_{p},j_{p}\right)  =\left(  i,j\right)  $, so
that $\left(  i,j\right)  =\left(  i_{p},j_{p}\right)  $. Hence, $p$ is an
$r\in\interval{h} $ such that $\left(  i,j\right)  =\left(  i_{r}
,j_{r}\right)  $. But this contradicts the fact that there exists no
$r\in\interval{h} $ such that $\left(  i,j\right)  =\left(  i_{r}
,j_{r}\right)  $.

Forget that we fixed $x$. We thus have obtained a contradiction for each $x\in
F_{i}\cap G_{j}$. Hence, there exists no such $x$. In other words, $F_{i}\cap
G_{j}=\varnothing$. This completes the proof of
(\ref{pf.lem.face-int-bij.sur.FiGj}) in Case 2.

We have now proved (\ref{pf.lem.face-int-bij.sur.FiGj}) in both Cases 1 and 2;
this completes the proof of (\ref{pf.lem.face-int-bij.sur.FiGj}).)

\item For any $i\in\interval{k}  $ and $j\in\interval{\ell}$, we
have
\begin{equation}
\left\vert F_{i}\cap G_{j}\right\vert =m_{i,j} .
\label{pf.lem.face-int-bij.sur.FiGj-size}
\end{equation}

(\textit{Proof:} Let $i\in\interval{k}  $ and $j\in\interval{\ell}$.
We are in one of the following two cases:

\textit{Case 1:} We have $\left(  i,j\right)  =\left(  i_{r},j_{r}\right)  $
for some $r\in\interval{h} $.

\textit{Case 2:} There exists no $r\in\interval{h} $ such that $\left(
i,j\right)  =\left(  i_{r},j_{r}\right)  $.

Let us first consider Case 1. In this case, we have $\left(  i,j\right)
=\left(  i_{r},j_{r}\right)  $ for some $r\in\interval{h} $. Consider
this $r$. Thus, (\ref{pf.lem.face-int-bij.sur.FiGj}) yields $F_{i}\cap
G_{j}=H_{r}$ (since $\left(  i,j\right)  =\left(  i_{r},j_{r}\right)  $ for
this $r$). Hence, $\left\vert F_{i}\cap G_{j}\right\vert =\left\vert
H_{r}\right\vert =m_{i_{r},j_{r}}$ (by (\ref{pf.lem.face-int-bij.sur.sizeH}),
applied to $p=r$). But from $\left(  i,j\right)  =\left(  i_{r},j_{r}\right)
$, we obtain $m_{i,j}=m_{i_{r},j_{r}}$. Comparing this with $\left\vert
F_{i}\cap G_{j}\right\vert =m_{i_{r},j_{r}}$, we obtain $\left\vert F_{i}\cap
G_{j}\right\vert =m_{i,j}$. This proves
(\ref{pf.lem.face-int-bij.sur.FiGj-size}) in Case 1.

Let us now consider Case 2. In this case, there exists no $r\in
\interval{h}$ such that $\left(  i,j\right)  =\left(  i_{r},j_{r}\right)  $.
Thus, (\ref{pf.lem.face-int-bij.sur.FiGj}) yields $F_{i}\cap G_{j}
=\varnothing$. Hence, $\left\vert F_{i}\cap G_{j}\right\vert =\left\vert
\varnothing\right\vert =0$. But we also have $m_{i,j}=0$ (because otherwise,
we would have $m_{i,j}\neq0$, and therefore
(\ref{pf.lem.face-int-bij.sur.irjr}) would show that $\left(  i,j\right)
=\left(  i_{r},j_{r}\right)  $ for some $r\in\interval{h} $; but this
would contradict the fact that there exists no $r\in\interval{h} $ such
that $\left(  i,j\right)  =\left(  i_{r},j_{r}\right)  $). Comparing this with
$\left\vert F_{i}\cap G_{j}\right\vert =0$, we obtain $\left\vert F_{i}\cap
G_{j}\right\vert =m_{i,j}$. This proves
(\ref{pf.lem.face-int-bij.sur.FiGj-size}) in Case 2.

We have now proved (\ref{pf.lem.face-int-bij.sur.FiGj-size}) in both Cases 1
and 2; hence (\ref{pf.lem.face-int-bij.sur.FiGj-size}) is proved.)

\item Each $p\in\interval{h} $ satisfies
\begin{equation}
F_{i_{p}}\cap G_{j_{p}}=H_{p}.
\label{pf.lem.face-int-bij.sur.FGp}
\end{equation}

(\textit{Proof:} Let $p\in\interval{h} $. Then, applying
(\ref{pf.lem.face-int-bij.sur.FiGj}) to $i=i_{p}$ and $j=j_{p}$, we find
\begin{align*}
F_{i_{p}}\cap G_{j_{p}} &  =
\begin{cases}
H_{r}, & \text{if }\left(  i_{p},j_{p}\right)  =\left(  i_{r},j_{r}\right)
\text{ for some }r\in\interval{h} ;\\
\varnothing, & \text{if not}
\end{cases}
\\
&  =H_{p}
\end{align*}
(since $\left(  i_{p},j_{p}\right)  =\left(  i_{r},j_{r}\right)  $ for some
$r\in\interval{h} $, namely for $r=p$). This proves
(\ref{pf.lem.face-int-bij.sur.FGp}).)

\item We have
\begin{equation}
\interval{n}=F_{1}\cup F_{2}\cup\cdots\cup F_{k}.
\label{pf.lem.face-int-bij.sur.Funion}
\end{equation}

(\textit{Proof:} Let $x\in\interval{n}$. Then, $x\in\interval{n}
=H_{1}\cup H_{2}\cup\cdots\cup H_{h}$. In other words, $x\in H_{p}$ for some
$p\in\interval{h} $. Consider this $p$. Now,
(\ref{pf.lem.face-int-bij.sur.FGp}) shows that $H_{p}=F_{i_{p}}\cap G_{j_{p}}
$. Thus, $x\in H_{p}=F_{i_{p}}\cap G_{j_{p}}\subseteq F_{i_{p}}\subseteq
F_{1}\cup F_{2}\cup\cdots\cup F_{k}$ (since $i_{p}\in\interval{k}  $).

Forget that we fixed $x$. We now have shown that $x\in F_{1}\cup F_{2}
\cup\cdots\cup F_{k}$ for each $x\in\interval{n}$. In other words,
$\interval{n}\subseteq F_{1}\cup F_{2}\cup\cdots\cup F_{k}$. Hence,
$\interval{n}=F_{1}\cup F_{2}\cup\cdots\cup F_{k}$ (since $F_{1}\cup
F_{2}\cup\cdots\cup F_{k}\subseteq\interval{n}$ obviously). This proves
(\ref{pf.lem.face-int-bij.sur.Funion}).)

\item Similarly, we obtain
\begin{equation}
\interval{n}=G_{1}\cup G_{2}\cup\cdots\cup G_{\ell}.
\label{pf.lem.face-int-bij.sur.Gunion}
\end{equation}

\item For each $i\in\interval{k}  $, we have
\begin{equation}
\left\vert F_{i}\right\vert =\sum_{j\in\interval{\ell}}m_{i,j}.
\label{pf.lem.face-int-bij.sur.Fsize}
\end{equation}

(\textit{Proof:} Let $i\in\interval{k}  $. The sets $G_{1},G_{2}
,\ldots,G_{\ell}$ are disjoint (by (\ref{pf.lem.face-int-bij.sur.Gdj})).
Hence, their intersections with $F_{i}$ -- that is, the sets $F_{i}\cap
G_{1},\ F_{i}\cap G_{2},\ \ldots,\ F_{i}\cap G_{\ell}$ -- are disjoint as well.

But $F_{i}\subseteq\interval{n}$, so that
\[
F_{i}=F_{i}\cap\underbrace{\interval{n}}_{\substack{=G_{1}\cup G_{2}
\cup\cdots\cup G_{\ell}\\\text{(by (\ref{pf.lem.face-int-bij.sur.Gunion}))}
}}=F_{i}\cap\underbrace{\left(  G_{1}\cup G_{2}\cup\cdots\cup G_{\ell}\right)
}_{=\bigcup_{j\in\interval{\ell}}G_{j}}=F_{i}\cap\left(  \bigcup
_{j\in\interval{\ell}}G_{j}\right)  =\bigcup_{j\in\interval{\ell}
}\left(  F_{i}\cap G_{j}\right)  ,
\]
so that
\[
\left\vert F_{i}\right\vert =\left\vert \bigcup_{j\in\interval{\ell}
}\left(  F_{i}\cap G_{j}\right)  \right\vert =\sum_{j\in\interval{\ell}
}\left\vert F_{i}\cap G_{j}\right\vert
\]
(since the sets $F_{i}\cap G_{1},\ F_{i}\cap G_{2},\ \ldots,\ F_{i}\cap
G_{\ell}$ are disjoint). Hence,
\[
\left\vert F_{i}\right\vert =\sum_{j\in\interval{\ell}}
\underbrace{\left\vert F_{i}\cap G_{j}\right\vert }_{\substack{=m_{i,j}
\\\text{(by (\ref{pf.lem.face-int-bij.sur.FiGj-size}))}}}=\sum_{j\in
\interval{\ell}}m_{i,j}.
\]
This proves (\ref{pf.lem.face-int-bij.sur.Fsize}).)

\item Similarly, for each $j\in\interval{\ell}$, we have
\begin{equation}
\left\vert G_{j}\right\vert =\sum_{i\in\interval{k}  }m_{i,j}.
\label{pf.lem.face-int-bij.sur.Gsize}
\end{equation}

\item We have $M=\left(  m_{i,j}\right)  _{i\in\interval{k}
,\ j\in\interval{\ell}}$, so that
\begin{equation}
\rowsums M=\left(  \sum_{j\in\interval{\ell}}
m_{1,j},\ \sum_{j\in\interval{\ell}}m_{2,j},\ \ldots,\ \sum_{j\in
\interval{\ell}}m_{k,j}\right)
\label{pf.lem.face-int-bij.sur.rsM}
\end{equation}
(by the definition of $\rowsums M$). Thus, for each $i\in
\interval{k}$, we have
\begin{equation}
\left(  \text{the }i\text{-th entry of }\rowsums M\right)
=\sum_{j\in\interval{\ell}}m_{i,j}.
\label{pf.lem.face-int-bij.sur.rsM2}
\end{equation}

\item All $k$ sets $F_{1},F_{2},\ldots,F_{k}$ are nonempty.

(\textit{Proof:} Recall that $\rowsums M\models n$; thus,
$\rowsums M$ is a composition. Hence, all entries of
$\rowsums M$ are positive.

For each $i\in\interval{k}  $, we have
\begin{align*}
\left\vert F_{i}\right\vert  & =\sum_{j\in\interval{\ell}}
m_{i,j}\ \ \ \ \ \ \ \ \ \ \left(  \text{by
(\ref{pf.lem.face-int-bij.sur.Fsize})}\right)  \\
& =\left(  \text{the }i\text{-th entry of }\rowsums M\right)
\ \ \ \ \ \ \ \ \ \ \left(
\text{by (\ref{pf.lem.face-int-bij.sur.rsM2})}\right)  \\
& >0\ \ \ \ \ \ \ \ \ \ \left(  \text{since all entries of }
\rowsums M\text{ are positive}\right)  ,
\end{align*}
and thus the set $F_{i}$ is nonempty. In other words, all $k$ sets
$F_{1},F_{2},\ldots,F_{k}$ are nonempty.)

\item Similarly, all $\ell$ sets $G_{1},G_{2},\ldots,G_{\ell}$ are nonempty.

\item Define the lists
\[
F:=\left(  F_{1},F_{2},\ldots,F_{k}\right)  \ \ \ \ \ \ \ \ \ \ \text{and}
\ \ \ \ \ \ \ \ \ \ G:=\left(  G_{1},G_{2},\ldots,G_{\ell}\right)  .
\]
The sets $F_{1},F_{2},\ldots,F_{k}$ are disjoint (by
(\ref{pf.lem.face-int-bij.sur.Fdj})) and nonempty (as we just showed) and have
union $\bigcup_{i=1}^{k}F_{i}=F_{1}\cup F_{2}\cup\cdots\cup F_{k}=
\interval{n}$ (by (\ref{pf.lem.face-int-bij.sur.Funion})). Thus, $F=\left(
F_{1},F_{2},\ldots,F_{k}\right)  $ is a face of $\interval{n}$. That is,
$F\in\CF $. Similarly, $G\in\CF $.

\item Furthermore, from $F=\left(  F_{1},F_{2},\ldots,F_{k}\right)  $, we see
that
\begin{align*}
\type F &  =\type \left(  F_{1},F_{2},\ldots
,F_{k}\right)  \\
&  =\left(  \left\vert F_{1}\right\vert ,\left\vert F_{2}\right\vert
,\ldots,\left\vert F_{k}\right\vert \right)  \\
&  =\left(  \sum_{j\in\interval{\ell}}m_{1,j},\ \sum_{j\in
\interval{\ell}}m_{2,j},\ \ldots,\ \sum_{j\in\interval{\ell}}
m_{k,j}\right)  \ \ \ \ \ \ \ \ \ \ \left(  \text{by
(\ref{pf.lem.face-int-bij.sur.Fsize})}\right)  \\
&  =\rowsums M\ \ \ \ \ \ \ \ \ \ \left(  \text{by
(\ref{pf.lem.face-int-bij.sur.rsM})}\right)  .
\end{align*}
Similarly, using (\ref{pf.lem.face-int-bij.sur.Gsize}), we can see that
\[
\type G=\colsums M.
\]

\item We have
\begin{equation}
FG=H.
\label{pf.lem.face-int-bij.sur.FG}
\end{equation}

(\textit{Proof:} As in the proof of Lemma \ref{lem.face-int-mat.type}, we can
see that
\[
FG=\left(  \text{lex-list of all intersections }F_{i}\cap G_{j}\text{ such
that }\left\vert F_{i}\cap G_{j}\right\vert \neq0\right)  .
\]
Thus,
\begin{align*}
FG &  =\left(  \text{lex-list of all intersections }F_{i}\cap G_{j}\text{ such
that }\left\vert F_{i}\cap G_{j}\right\vert \neq0\right)  \\
&  =\left(  \text{lex-list of all intersections }F_{i}\cap G_{j}\text{ such
that }m_{i,j}\neq0\right)  \\
&  \ \ \ \ \ \ \ \ \ \ \ \ \ \ \ \ \ \ \ \ \left(  \text{since
(\ref{pf.lem.face-int-bij.sur.FiGj-size}) lets us rewrite }\left\vert
F_{i}\cap G_{j}\right\vert \text{ as }m_{i,j}\right)  \\
&  =\left(  F_{i_{1}}\cap G_{j_{1}},\ F_{i_{2}}\cap G_{j_{2}},\ \ldots
,\ F_{i_{h}}\cap G_{j_{h}}\right)
\end{align*}
(since $\left(  \left(  i_{1},j_{1}\right)  ,\ \left(  i_{2},j_{2}\right)
,\ \ldots,\ \left(  i_{h},j_{h}\right)  \right)  $ is the lex-list of all the
pairs $\left(  i,j\right)  \in\interval{k}  \times\interval{\ell}$
satisfying $m_{i,j}\neq0$). Therefore,
\begin{align*}
FG &  =\left(  F_{i_{1}}\cap G_{j_{1}},\ F_{i_{2}}\cap G_{j_{2}}
,\ \ldots,\ F_{i_{h}}\cap G_{j_{h}}\right)  \\
&  =\left(  H_{1},H_{2},\ldots,H_{h}\right)  \ \ \ \ \ \ \ \ \ \ \left(
\text{by (\ref{pf.lem.face-int-bij.sur.FGp})}\right)  \\
&  =H.
\end{align*}

This proves (\ref{pf.lem.face-int-bij.sur.FG}).)

\item We have
\begin{equation}
\mat\left(  F,G\right)  =M.
\label{pf.lem.face-int-bij.sur.mat}
\end{equation}

(\textit{Proof:} By the definition of $\mat\left(  F,G\right)
$, we have
\[
\mat\left(  F,G\right)  =\left(  \underbrace{\left\vert
F_{i}\cap G_{j}\right\vert }_{\substack{=m_{i,j}\\\text{(by
(\ref{pf.lem.face-int-bij.sur.FiGj-size}))}}}\right)  _{i\in\interval{k}
,\ j\in\interval{\ell}}=\left(  m_{i,j}\right)  _{i\in\interval{k}
,\ j\in\interval{\ell}}=M,
\]
qed.)
\end{enumerate}

Now, we have $\left(  F,G\right)  \in\CF \times\CF $ (since
$F\in\CF $ and $G\in\CF $) and
\begin{align*}
\Phi\left(  F,G\right)   &  =\left(  \underbrace{FG}_{\substack{=H\\\text{(by
(\ref{pf.lem.face-int-bij.sur.FG}))}}},\ \underbrace{\mat
\left(  F,G\right)  }_{\substack{=M\\\text{(by
(\ref{pf.lem.face-int-bij.sur.mat}))}}}\right)  \ \ \ \ \ \ \ \ \ \ \left(
\text{by the definition of }\Phi\right)  \\
&  =\left(  H,M\right)  .
\end{align*}
Thus, $\left(  H,M\right)  $ lies in the image of $\Phi$. This completes the
proof of the surjectivity of $\Phi$. \medskip

We have now proved that the map $\Phi$ is well-defined, injective and
surjective. Thus, $\Phi$ is a bijection, and Lemma \ref{lem.face-int-bij} is proven.
\end{proof}

Using Lemma \ref{lem.face-int-bij}, we can now easily prove Solomon's Mackey formula:

\begin{proof}
[Proof of Theorem \ref{thm.smf}.]Define the set $\CR$ as in Lemma
\ref{lem.face-int-bij}. From Lemma \ref{lem.face-int-bij}, we know that the
map (\ref{eq.lem.face-int-bij.bij}) is a bijection.

In $\kf$, we have
\[
\BBt_{\alpha}=\sum_{\substack{F\in\CF;\\
\type F=\alpha}}F\ \ \ \ \ \ \ \ \ \ \text{and}
\ \ \ \ \ \ \ \ \ \ \BBt_{\beta}=\sum_{\substack{F\in
\CF ;\\\type F=\beta}}F=\sum_{\substack{G\in
\CF ;\\\type G=\beta}}G.
\]
Multiplying these two equalities, we obtain
\begin{align*}
\BBt_{\alpha}\BBt_{\beta}  &  =\left(
\sum_{\substack{F\in\CF ;\\\type F=\alpha}}F\right)
\left(  \sum_{\substack{G\in\CF ;\\\type G=\beta
}}G\right)  =\underbrace{\sum_{\substack{F\in\CF ;\\\type 
F=\alpha}}\ \ \sum_{\substack{G\in\CF ;\\\type G=\beta}
}}_{\substack{=\sum_{\substack{\left(  F,G\right)  \in\CF
\times\CF ;\\\type F=\alpha;\\\type 
G=\beta}}\\=\sum_{\substack{\left(  F,G\right)  \in\CF \times
\CF ;\\\rowsums \left(  \mat\left(
F,G\right)  \right)  =\alpha;\\\colsums \left(
\mat\left(  F,G\right)  \right)  =\beta}}\\\text{(since
(\ref{eq.lem.face-int-mat.type.F}) and (\ref{eq.lem.face-int-mat.type.G}))
allow us to}\\\text{rewrite }\type F\text{ and }
\type G\text{ as }\rowsums \left(
\mat\left(  F,G\right)  \right)  \\\text{and }
\colsums \left(  \mat\left(  F,G\right)  \right)
\text{)}}}FG\\
&  =\sum_{\substack{\left(  F,G\right)  \in\CF \times\CF
;\\\rowsums \left(  \mat\left(  F,G\right)
\right)  =\alpha;\\\colsums \left(  \mat\left(
F,G\right)  \right)  =\beta}}FG=\sum_{\substack{\left(  H,M\right)
\in\CR;\\\rowsums M=\alpha;\\\colsums M=\beta}}H
\end{align*}
(here, we have substituted $\left(  H,M\right)  $ for $\left(
FG,\ \mat\left(  F,G\right)  \right)  $ in the sum, since the
map (\ref{eq.lem.face-int-bij.bij}) is a bijection). Thus,
\begin{align*}
\BBt_{\alpha}\BBt_{\beta}  &
=\sum_{\substack{\left(  H,M\right)  \in\CR ;\\\rowsums M
=\alpha;\\\colsums M=\beta}}H=\sum_{\substack{\left(
H,M\right)  \in\CF \times\CM;\\\rowsums M\models
n;\\\colsums M\models n;\\\type H=\readb M;
\\\rowsums M=\alpha
;\\\colsums M=\beta}}H\ \ \ \ \ \ \ \ \ \ \left(  \text{by the
definition of }\CR \right) \\
&  =\sum_{\substack{\left(  H,M\right)  \in\CF \times\CM;
\\\type H=\readb M;\\
\rowsums M=\alpha;\\\colsums M=\beta}}H
\end{align*}
(here, we have removed the conditions $\rowsums M\models n$ and
$\colsums M\models n$ from under the summation sign, since these
conditions are redundant: each matrix $M\in\CM$ satisfying
$\rowsums M=\alpha$ and $\colsums M=\beta$ will
automatically satisfy $\rowsums M=\alpha\models n$ and
$\colsums M=\beta\models n$). Continuing our computation, we
have
\[
\BBt_{\alpha}\BBt_{\beta}=\sum
_{\substack{\left(  H,M\right)  \in\CF \times\CM
;\\\type H=\readb M;\\
\rowsums M=\alpha;\\\colsums M=\beta}}H
=\sum_{\substack{M\in\CM ;\\\rowsums M=\alpha
;\\\colsums M=\beta}}\ \ \underbrace{\sum_{\substack{H\in
\CF ;\\\type H=\readb M}}H
}_{\substack{=\sum\limits_{\substack{F\in\CF ;\\
\type F=\readb M}}F\\=\BBt_{\readb M}\\
\text{(by the definition of }
\BBt_{\readb M}\text{)}}}
=\sum_{\substack{M\in\CM ;\\\rowsums M=\alpha;\\
\colsums M=\beta}}\BBt_{\readb M}.
\]

However, recall the map $\rho: \dsn \rightarrow\left(
\kf \right)  ^{S_{n}}$ from Definition \ref{defn_rho}. This
map $\rho$ is a $\kk$-algebra anti-isomorphism (by Theorem
\ref{dsn_kfsn_morph}), therefore an injective $\kk$-algebra
anti-morphism. By its definition, it satisfies
\begin{equation}
\rho\left(  \BB_{\gamma}\right)  =\BBt_{\gamma
}\ \ \ \ \ \ \ \ \ \ \text{for each }\gamma\models n.
\label{pf.thm.smf.rho}
\end{equation}
Hence, in particular, $\rho\left(  \BB_{\alpha}\right)
=\BBt_{\alpha}$ and $\rho\left(  \BB_{\beta}\right)
=\BBt_{\beta}$. Furthermore, for any matrix $M\in
\CM $ satisfying $\rowsums M=\alpha$, we have
$\readb M\models n$\ \ \ \ \footnote{\textit{Proof.}
Let $M\in\CM $ be a matrix satisfying $\rowsums M=\alpha
$. Write this matrix $M$ in the form $M=\left(  m_{i,j}\right)  _{i\in
\interval{k},\ j\in\interval{\ell}}$. Then, $\rowsums M$ is
defined as the $k$-tuple $\left(  r_{1},r_{2},\ldots,r_{k}\right)
\in\mathbb{N}^{k}$, where $r_{i}:=\sum_{j\in\interval{\ell}}m_{i,j}$ is
the sum of all entries in the $i$-th row of $M$. Hence, $\alpha
=\rowsums M=\left(  r_{1},r_{2},\ldots,r_{k}\right)  $, so that
\begin{align}
\left(  \text{the sum of all entries of }\alpha\right)   &  =r_{1}
+r_{2}+\cdots+r_{k}=\sum_{i\in\interval{k}  }
\underbrace{r_{i}}_{=\sum_{j\in\interval{\ell}}m_{i,j}}
\nonumber\\
&  =\sum_{i\in\interval{k}  }\ \ \sum_{j\in\interval{\ell}}
m_{i,j}=\left(  \text{the sum of all entries of }M\right)  \nonumber\\
&  =\left(  \text{the sum of all nonzero entries of }M\right)
\label{pf.thm.smf.fn1.5}
\end{align}
(since the zero entries of $M$ do not affect the sum).
\par
But $\readb M$ is the lex-list of all nonzero entries
of $M$ (by the definition of $\readb M$). Hence,
\[
\left(  \text{the sum of all entries of }\readb M\right)
=\left(  \text{the sum of all nonzero entries of }M\right)  .
\]
Comparing this with (\ref{pf.thm.smf.fn1.5}), we obtain
\[
\left(  \text{the sum of all entries of }\readb M\right)
=\left(  \text{the sum of all entries of }\alpha\right)  =n
\]
(since $\alpha$ is a composition of $n$). Therefore,
$\readb M$ is a composition of $n$ (since $\readb M$
is a composition). In other words,
$\readb M\models n$.} and therefore
\begin{equation}
\rho\left(  \BB_{\readb M}\right)
=\BBt_{\readb M}
\label{pf.thm.smf.rho-rM}
\end{equation}
(by (\ref{pf.thm.smf.rho}), applied to $\gamma=\readb M$).

Since $\rho$ is a $\kk$-algebra anti-morphism (by Theorem
\ref{dsn_kfsn_morph}), we now have
\begin{align*}
\rho\left(  \BB_{\beta}\BB_{\alpha}\right)   &
=\underbrace{\rho\left(  \BB_{\alpha}\right)  }_{=
\BBt_{\alpha}}\ \ \underbrace{\rho\left(  \BB_{\beta}\right)
}_{=\BBt_{\beta}}=\BBt_{\alpha
}\BBt_{\beta}\\
&  =\sum_{\substack{M\in\CM ;\\\rowsums M=\alpha
;\\\colsums M=\beta}}
\ \ \underbrace{\BBt_{\readb M}}_{=\rho\left(  \BB_{\readb M}\right)  }
=\sum_{\substack{M\in\CM ;\\\rowsums M=\alpha;\\
\colsums M=\beta}}\rho\left(  \BB_{\readb M}\right)
=\rho\left(  \sum_{\substack{M\in\CM ;\\\rowsums M
=\alpha;\\\colsums M=\beta}}\BB_{\readb M}\right)
\end{align*}
(since $\rho$ is $\kk$-linear). Since the map $\rho$ is injective
(again by Theorem \ref{dsn_kfsn_morph}), we can now \textquotedblleft
unapply\textquotedblright\ $\rho$ to this equality, and obtain

\[
\BB_{\beta}\BB_{\alpha} = \sum_{\substack{M\in\CM;\\
\rowsums M=\alpha;\\\colsums M=\beta}} \BB_{\readb M}.
\]
This proves Theorem \ref{thm.smf}.
\end{proof}

\printbibliography

@article{BleLau92,
	Author = {Dieter Blessenohl and Hartmut Laue},
	Journal = {S\'{e}minaire Lotharingien de Combinatoire},
	Title = {Algebraic combinatorics related to the free {L}ie algebra},
	Volume = {B29e},
	Year = {1992},
	Zbl = {0978.05543},
	Zbmath = {1660363}}

@book{Sag20,
	Author = {Bruce E. Sagan},
	Publisher = {American Mathematical Society},
	Series = {Graduate studies in mathematics},
	Title = {Combinatorics: The Art of Counting},
	Volume = {210},
	Year = {2020}}

@misc{Saliola,
	Author = {Saliola, Franco},
	Title = {Hyperplane arrangements and descent algebras},
	Url = {https://web.archive.org/web/20151210043745/http://thales.math.uqam.ca/~saliola/maths/publications/LectureNotes/DesAlgLectureNotes.pdf},
	Note = {corrections at \url{https://www.cip.ifi.lmu.de/~grinberg/algebra/saliola-errata.pdf}}}

@misc{GriVas24,
      title={The enriched $q$-monomial basis of the quasisymmetric functions}, 
      author={Darij Grinberg and Ekaterina A. Vassilieva},
      year={2024},
      eprint={2309.01118v4},
      archivePrefix={arXiv},
      primaryClass={math.CO},
      url={https://arxiv.org/abs/2309.01118v4}, 
}

@misc{GriVas26,
 author = {Darij Grinberg and Ekaterina A. Vassilieva},
 title = {The left-to-right minima basis of the group algebra of the symmetric group (updated version)},
 year = {2026},
 howpublished = {Preprint, {arXiv}:2601.02952 [math.{CO}] (2026)},
 url = {https://arxiv.org/abs/2601.02952v2},
 arXiv = {arXiv:2601.02952v2}
}

@article{ReSaWe14,
	AUTHOR = {Reiner, Victor and Saliola, Franco and Welker, Volkmar},
	TITLE = {Spectra of symmetrized shuffling operators},
	JOURNAL = {Memoirs of the American Mathematical Society},
	VOLUME = {228},
	YEAR = {2014},
	NUMBER = {1072},
	PAGES = {vi+109},
	ISSN = {0065-9266},
	ISBN = {978-0-8218-9095-0},
	note = {Preprint available at \href{https://arxiv.org/abs/1102.2460v2}{arXiv:1102.2460v2}}
}

@book{EC1,
	Author = {Richard P. Stanley},
	Publisher = {Cambridge University Press},
	Title = {Enumerative combinatorics},
	Volume = {1},
	edition = {2},
	Year = {2011}}

@article{GarReu89,
	Author = {Garsia, A. and Reutenauer, C.},
	Journal = {Advances in mathematics},
	Pages = {189--262},
	Title = {A decomposition of {S}olomon's descent algebra},
	Volume = {77},
	Year = {1989}}

@article{Sol76,
	Author = {Solomon, Louis},
	Journal = {Journal of Algebra},
	Pages = {255--264},
	Title = {A {M}ackey formula in the group ring of a {C}oxeter group},
	Volume = {41},
	Year = {1976}}

@article{Willigenburg98,
	Author = {van Willigenburg, Stephanie},
	Doi = {10.1006/jabr.1998.7427},
	Fjournal = {Journal of Algebra},
	Issn = {0021-8693},
	Journal = {J. Algebra},
	Number = {2},
	Pages = {693--698},
	Title = {A proof of {Solomon}'s rule},
	Volume = {206},
	Year = {1998},
	Zbl = {0973.20005},
	Zbmath = {1213841}}

@Article{ncsf2,
 Author = {Krob, Daniel and Leclerc, Bernard and Thibon, Jean-Yves},
 Title = {Noncommutative symmetric functions. {II}: {Transformations} of alphabets},
 FJournal = {International Journal of Algebra and Computation},
 Journal = {Int. J. Algebra Comput.},
 ISSN = {0218-1967},
 Volume = {7},
 Number = {2},
 Pages = {181--264},
 Year = {1997},
 DOI = {10.1142/S0218196797000113},
 zbMATH = {1026306},
 Zbl = {0907.05055}
}

@incollection{Schocker,
 author = {Schocker, Manfred},
 title = {The descent algebra of the symmetric group},
 booktitle = {Representations of finite dimensional algebras and related topics in Lie theory and geometry. Proceedings from the 10th international conference on algebras and related topics, ICRA X, Toronto, Canada, July 15--August 10, 2002},
 isbn = {0-8218-3416-9},
 pages = {145--161},
 year = {2004},
 publisher = {Providence, RI: American Mathematical Society (AMS)},
 zbMATH = {2070265},
 Zbl = {1072.20004}
}

@article{Brown-SRM,
 author = {Brown, Kenneth S.},
 title = {Semigroups, rings, and {Markov} chains},
 fjournal = {Journal of Theoretical Probability},
 journal = {J. Theor. Probab.},
 issn = {0894-9840},
 volume = {13},
 number = {3},
 pages = {871--938},
 year = {2000},
 doi = {10.1023/A:1007822931408},
 zbMATH = {1552714},
 Zbl = {0980.60014}
}

@incollection{Brown-Methods,
 author = {Brown, Kenneth S.},
 title = {Semigroup and ring theoretical methods in probability},
 booktitle = {Representations of finite dimensional algebras and related topics in Lie theory and geometry. Proceedings from the 10th international conference on algebras and related topics, ICRA X, Toronto, Canada, July 15--August 10, 2002},
 isbn = {0-8218-3416-9},
 pages = {3--26},
 year = {2004},
 publisher = {Providence, RI: American Mathematical Society (AMS)},
 zbMATH = {2070260},
 Zbl = {1043.60055}
}

@phdthesis{Bidigare-thesis,
	Author = {Bidigare, Patrick},
	School = {Univ. Michigan},
	Title = {Hyperplane arrangement face algebras and their associated {M}arkov chains},
	Year = {1997}
}

@misc{Wallach,
 author = {Wallach, Nolan R.},
 title = {Lie algebra cohomology and holomorphic continuation of generalized {Jacquet} integrals},
 year = {1988},
 howpublished = {Representations of {Lie} groups: analysis on homogeneous spaces and representations of {Lie} groups, {Proc}. {Symp}., {Kyoto}/{Jap}. and {Hiroshima}/{Jap}. 1986, {Adv}. {Stud}. {Pure} {Math}. 14, 123-151 (1988).},
 zbMATH = {4175209},
 Zbl = {0714.17016},
 doi = {10.2969/aspm/01410123}
}

@article{DiaSha81,
  title={Generating a random permutation with random transpositions},
  author={Diaconis, Persi and Shahshahani, Mehrdad},
  journal={Zeitschrift f{\"u}r Wahrscheinlichkeitstheorie und verwandte Gebiete},
  volume={57},
  number={2},
  pages={159--179},
  year={1981},
  publisher={Springer},
  doi={10.1007/BF00535487},
  url={https://www.imo.universite-paris-saclay.fr/~pierre-loic.meliot/symmetric/texts/Diaconis,%20Shahshahani%20-%20Generating%20a%20random%20permutation%20with%20random%20transpositions.pdf},
}

@article{Hsiao,
 author = {Hsiao, Samuel K.},
 title = {A semigroup approach to wreath-product extensions of {Solomon}'s descent algebras},
 fjournal = {The Electronic Journal of Combinatorics},
 journal = {Electron. J. Comb.},
 issn = {1077-8926},
 volume = {16},
 number = {1},
 pages = {research paper r21, 9},
 year = {2009},
 url = {https://doi.org/10.37236/110},
 doi = {10.37236/110},
 zbMATH = {5541086},
 Zbl = {1209.20008}
}

@book{BlessenohlSchocker,
 author = {Blessenohl, Dieter and Schocker, Manfred},
 title = {Noncommutative character theory of the symmetric group},
 isbn = {1-86094-511-2},
 year = {2005},
 publisher = {River Edge, NJ: World Scientific},
 zbMATH = {2139228},
 Zbl = {1089.20004}
}

@article{Phatar91,
	AUTHOR = {Phatarfod, Ravindra M.},
	TITLE = {On the matrix occurring in a linear search problem},
	JOURNAL = {Journal of Applied Probability},
	VOLUME = {28},
	YEAR = {1991},
	NUMBER = {2},
	PAGES = {336--346},
	ISSN = {0021-9002},
	DOI = {10.1017/s0021900200039723},
	LABEL = {Phatar91},
}

@article{GriLaf22,
  title={The one-sided cycle shuffles in the symmetric group algebra},
  author={Grinberg, Darij and Lafreni{\`e}re, Nadia},
  journal={arXiv preprint arXiv:2212.06274},
  year={2022},
  url={https://arxiv.org/abs/2212.06274},
}

@article{DieSal18,
  title={Spectral analysis of random-to-random Markov chains},
  author={Dieker, Antonius B and Saliola, Franco V},
  journal={Advances in Mathematics},
  volume={323},
  pages={427--485},
  year={2018},
  publisher={Elsevier},
  doi={10.1016/j.aim.2017.10.034},
}

@MISC{MO308600,
    TITLE = {Is this sum of cycles invertible in $\mathbb QS_n$?},
    AUTHOR = {Darij Grinberg},
    HOWPUBLISHED = {MathOverflow},
    NOTE = {version: 2025-05-05},
    EPRINT = {https://mathoverflow.net/q/308600},
    URL = {https://mathoverflow.net/q/308600}
}

@misc{21s,
      title={An Introduction to Algebraic Combinatorics}, 
      author={Darij Grinberg},
      year={2025},
      eprint={2506.00738v1},
      archivePrefix={arXiv},
      primaryClass={math.CO},
      url={https://arxiv.org/abs/2506.00738v1}, 
}

@misc{dc2023,
      title={The one-sided cycle shuffles, and other mysteries and wonders of the symmetric group algebra [talk slides]}, 
      author={Darij Grinberg},
      year={2025},
      url={https://www.cip.ifi.lmu.de/~grinberg/algebra/dc2023.pdf}, 
}

@article{KKOPSS24,
 author = {Kenyon, Richard and Kontsevich, Maxim and Ogievetskii, Oleg and Pohoata, Cosmin and Sawin, Will and Shlosman, Semen},
 title = {The miracle of integer eigenvalues},
 fjournal = {Functional Analysis and its Applications},
 journal = {Funct. Anal. Appl.},
 issn = {0016-2663},
 volume = {58},
 number = {2},
 pages = {182--194},
 year = {2024},
 language = {English},
 doi = {10.1134/S0016266324020072},
 zbMATH = {7885505}
}

@phdthesis{Reizen19,
	author = {Reizenstein, Jeremy Francis},
	title = {Iterated-Integral Signatures in Machine
	Learning},
	type = {PhD thesis},
	school = {University of Warwick},
	year = {2019},
	pages = {ix+107},
	url = {http://wrap.warwick.ac.uk/131162/},
	label = {Reizen19}
}

@Book{Lorenz18,
 Author = {Lorenz, Martin},
 Title = {A tour of representation theory},
 FSeries = {Graduate Studies in Mathematics},
 Series = {Grad. Stud. Math.},
 ISSN = {1065-7338},
 Volume = {193},
 ISBN = {978-1-4704-3680-3; 978-1-4704-4905-6},
 Year = {2018},
 Publisher = {Providence, RI: American Mathematical Society (AMS)},
 Language = {English},
 DOI = {10.1090/gsm/193},
 zbMATH = {6979266},
 Zbl = {1407.16001}
}

@book{DiaconisFulman,
 author = {Diaconis, Persi and Fulman, Jason},
 title = {The mathematics of shuffling cards},
 isbn = {978-1-4704-6303-8; 978-1-4704-7290-0},
 year = {2023},
 publisher = {Providence, RI: American Mathematical Society (AMS)},
 language = {English},
 doi = {10.1090/mbk/146},
 zbMATH = {7674931},
 Zbl = {1527.60001}
}

@incollection{Diaconis03,
 author = {Diaconis, Persi},
 title = {Mathematical developments from the analysis of riffle shuffling},
 booktitle = {Groups, combinatorics and geometry. Proceedings of the L. M. S. Durham symposium, Durham, UK, July 16--26, 2001},
 isbn = {981-238-312-3},
 pages = {73--97},
 year = {2003},
 publisher = {River Edge, NJ: World Scientific},
 language = {English},
 zbMATH = {1981749},
 Zbl = {1026.60005},
 note = {\url{https://diaconis.ckirby.su.domains/papers.html}},
}

\end{document}